\documentclass[a4paper,11pt]{amsart}
\usepackage[margin=1.2in]{geometry}
\pdfoutput=1

\usepackage{amsmath,amssymb,amscd,color,xcolor,mathrsfs}
\usepackage{tikz}
\usepackage{tikz-cd}
\usepackage{float}
\usepackage{comment}
\definecolor{green}{RGB}{0,144,0}
\definecolor{bluegreen}{RGB}{17,100,180}
\usepackage[linkcolor = bluegreen, citecolor = bluegreen, colorlinks = True]{hyperref}
\numberwithin{equation}{section}
\numberwithin{figure}{section}
\numberwithin{table}{section}
\usepackage[all]{hypcap}
\usepackage{mathtools}

\usepackage{pythonhighlight}

\newtheorem{theorem}{Theorem}[section]
\newtheorem*{theorem*}{Theorem}
\newtheorem{conjecture}[theorem]{Conjecture}
\newtheorem{lemma}[theorem]{Lemma}
\newtheorem{corollary}[theorem]{Corollary}
\newtheorem{proposition}[theorem]{Proposition}

\newtheorem*{conjwild}{Conjecture~\ref*{conj:wild}}
\newtheorem*{conjDelLel}{Conjecture~\ref*{conj:Del-Lel}}
\newtheorem*{conjPar}{Conjecture~\ref*{conj:parity}}
\newtheorem*{theorem-e3}{Theorem~\ref*{thm:e3}}

\newtheorem{remark}[theorem]{Remark}

\newcommand{\teichmuller}{Teichm{\"u}ller }
\newcommand{\calH}{\mathcal{H}}
\newcommand{\odd}{\mathcal{H}^{odd}}

\newcommand{\hyp}{\mathcal{H}^{hyp}}

\newcommand{\C}{\mathbb{C}}
\newcommand{\Z}{\mathbb{Z}}
\newcommand{\R}{\mathbb{R}}
\newcommand{\Q}{\mathbb{Q}}
\newcommand{\N}{\mathbb{N}}
\newcommand{\cuspwidth}{{\rm{lcm}}\left(\frac{w_{1}}{\gcd(w_{1},h_{1})},\frac{w_{2}}{\gcd(w_{2},h_{2})}\right)}
\newcommand{\orb}{\mathcal{G}}
\DeclareMathOperator{\SL}{{\rm{SL}}}
\DeclareMathOperator{\SO}{{\rm{SO}}}

\DeclareMathOperator{\GL}{{\rm{GL}}}
\DeclareMathOperator{\PSL}{{\rm{PSL}}}
\DeclareMathOperator{\PSO}{{\rm{PSO}}}

\DeclareMathOperator{\Sym}{{\rm{Sym}}}
\DeclareMathOperator{\Alt}{{\rm{Alt}}}

\DeclareMathOperator{\Mod}{{\rm{Mod}}}

\DeclareMathOperator{\End}{{\rm{End}}}
\DeclareMathOperator{\Aut}{{\rm{Aut}}}
\DeclareMathOperator{\Out}{{\rm{Out}}}
\DeclareMathOperator{\Jac}{{\rm{Jac}}}

\author{Luke Jeffreys}
\address{School of Mathematics, University of Bristol, Fry Building, Woodland Road, Bristol BS8 1UG, UK}
\curraddr{}
\email{luke.jeffreys@bristol.ac.uk}
\thanks{The first author is a Leverhulme Early Career Fellow (ECF-2023-553) and so thanks the Leverhulme Trust for their support.}

\author{Carlos Matheus}
\address{Centre de Math{\'e}matiques Laurent Schwartz, {\'E}cole Polytechnique, 91128 Palaiseau Cedex, France}
\curraddr{}
\email{carlos.matheus@math.cnrs.fr}

\title[Euler characteristics of $\SL(2,\Z)$-orbit graphs of origamis]{Euler characteristics of $\boldsymbol{\SL(2,\Z)}$-orbit graphs of origamis}


\begin{document}

\begin{abstract}
    The $\SL(2,\Z)$-orbits of primitive $n$-squared origamis can be represented by finite four-regular graphs. It is a conjecture of McMullen that the family of orbit graphs of such origamis in the stratum $\calH(2)$ is an expander family. In this work, we provide indirect evidence for this conjecture by proving that the absolute values of the Euler characteristics of the graphs in this family go to infinity with the number of squares $n$. This generalises previous work of the authors, in which we established eventual non-planarity for this family, and provides the strongest indirect evidence to date for McMullen's conjecture.

    We also prove that the same phenomenon holds for the $\SL(2,\Z)$-orbits of primitive origamis in the Prym loci of $\calH(4)$ and $\calH(6)$ that have been classified by Lanneau--Nguyen.
    
    Assuming conjectures of Zmiaikou and Delecroix--Leli\`evre concerning $\SL(2,\Z)$-orbits in $\calH(1,1)$ and $\calH(4)$, respectively, we establish the same result for all origamis in $\calH(1,1)$ and for two families of non-Prym origamis in $\calH(4)$.

    Finally, assuming a stronger conjecture concerning orbit growth in low-complexity strata, we prove an analogous result in greater generality. That is, we establish that for any family of $\SL(2,\Z)$-orbit graphs of primitive $n$-squared origamis in a stratum of translation surfaces with one or two conical singularities, the absolute values of the Euler characteristics of the graphs typically go to infinity with the number of squares in the sense that this holds along a density one subsequence.

    By relating the genus of the graphs constructed with elliptic generators to the genus of the associated arithmetic \teichmuller curves, we are also able to recover results of Mukamel and extend results of Torres-Teigel--Zachhuber establishing that the genera of these \teichmuller curves in $\calH(2)$ and the Prym loci of $\calH(4)$ and $\calH(6)$ also go to infinity.

    The proofs rely on counts of integral points on algebraic hypersurfaces using methods of Bombieri--Pila and Browning--Gorodnik, on counts of orbifold points on Teichm\"uller curves (via CM points on Hilbert modular surfaces in some cases), and on counts of pseudo-Anosov diffeomorphisms with bounded dilatation.
\end{abstract}

\maketitle
\tableofcontents

%%%%%%%%%%%%%%%%%%%%%%%%%%%%%%%%%%%%%%%%%
%%%%%%%%%%%%%%%%%%%%%%%%%%%%%%%%%%%%%%%%%

\section{Introduction}

A \textit{square-tiled surface} is an orientable connected surface realised as a finite cover of the unit square torus $\mathbb{T}^2:=\mathbb{C}/(\mathbb{Z}\oplus i\mathbb{Z})$ possibly branched only at $0\in\mathbb{T}^2$. In the literature, square-tiled surfaces are also called \textit{origamis} and we will use this terminology in the sequel. The covering map allows us to pullback the holomorphic one-form $dz$ from $\mathbb{T}^{2}$ to a non-zero holomorphic one-form on the origami. As such, an origami is a special case of a \textit{translation surface} --- a Riemann surface $M$ paired with a non-zero holomorphic one-form $\omega\in\Omega^{1}(M)$.

The moduli space of translation surfaces admits a natural action by the matrix group $\SL(2,\R)$, and origamis have played an important role in understanding this action. For example, since origamis correspond to integral points of such moduli spaces, one can compute Masur--Veech volumes of these spaces by counting origamis (cf. \cite{Zo} and \cite{EsOk}). Moreover, their $\SL(2,\mathbb{R})$-orbits are closed subvarieties of the moduli spaces known as \emph{arithmetic Teichm\"uller curves}. We refer the reader to the recent book \cite{AtMa} and survey \cite{Fil} for further explanations about moduli spaces of translation surfaces and their applications to areas of mathematics including, for example, dynamical systems and algebraic geometry.

It is possible to restrict the above action of $\SL(2,\R)$ to an action of $\SL(2,\Z)$ on origamis. These $\SL(2,\mathbb{Z})$-orbits can be turned into Schreier graphs sitting on the corresponding arithmetic Teichm\"uller curves. That is, one turns an $\SL(2,\Z)$-orbit into a combinatorial graph by setting the vertices to be the origamis in the orbit with the edges corresponding to the action of a generating set $S$ for $\SL(2,\Z)$ --- this is isomorphic to the coset Schreier graph $\text{Sch}(\SL(2,\Z), H ,S)$ with $H$ being the stabiliser of an origami in the orbit and $S$ as before. In the context of the moduli space $\calH(2)$ of genus two translation surfaces whose one-forms have a single zero of order two --- one of only three settings where the $\SL(2,\Z)$-orbits of primitive origamis have been classified (see Hubert--Leli\`evre~\cite{HL} and McMullen~\cite{McM}) --- McMullen gave the following conjecture.

\begin{conjecture}
    Let $(\mathcal{G}_{n})_{n}$ be the family of $\SL(2,\Z)$-orbit graphs of primitive $n$-squared origamis in $\calH(2)$. Then $(\mathcal{G}_{n})_{n}$ is a family of expander graphs.
\end{conjecture}

There are two natural generating sets to consider for these Schreier graphs: the \textit{parabolic generators} $T:=\begin{bsmallmatrix}
    1 & 1 \\
    0 & 1
\end{bsmallmatrix}$ and $S:=\begin{bsmallmatrix}
    1 & 0 \\
    1 & 1
\end{bsmallmatrix}$, and the \textit{elliptic generators} $R:=\begin{bsmallmatrix}
    0 & -1 \\
    1 & 0
\end{bsmallmatrix}$ and $U:=\begin{bsmallmatrix}
    0 & 1 \\
    -1 & 1
\end{bsmallmatrix}$.

In previous work~\cite{JM}, the authors established the first piece of indirect evidence for this conjecture by demonstrating that the family of graphs with the parabolic generators is eventually non-planar. That this gives indirect evidence for expansion follows from the Lipton--Tarjan Separator Theorem~\cite{LT} that asserts that a family of planar graphs cannot be an expander family. Extending this result, the Gilbert--Hutchinson--Tarjan Separator Theorem~\cite{GHT} asserts that a family of graphs with bounded genus cannot be an expander family either. In other words, the absolute values of the Euler characteristics of a family of expander graphs go to infinity. In this work, we establish this behaviour for the family of orbit graphs in $\calH(2)$ giving the strongest indirect evidence so far for McMullen's conjecture. We also establish the same phenomena for the orbit graphs of primitive origamis in the Prym loci of $\calH(4)$ and $\calH(6)$ --- the only places outside of $\calH(2)$ where the $\SL(2,\Z)$-orbits have been classified (see Lanneau--Nguyen~\cite{LN14,LN20}).

\begin{theorem}\label{thm:main}
    Let $(\mathcal{G}_{n})_{n}$ be a family of $\SL(2,\Z)$-orbit graphs of primitive $n$-squared origamis in $\calH(2)$ or in the Prym loci of $\calH(4)$ or $\calH(6)$ with either the parabolic or elliptic generators. Then $|\chi(\mathcal{G}_{n})|\to\infty$ as $n\to\infty$.
\end{theorem}

We prove this result for the parabolic generators by establishing the following result concerning faces of length at most four in the orbit graphs.

\begin{theorem}\label{thm:submain-1}
    Let $(\mathcal{G}_{n})_{n}$ be a family of $\SL(2,\Z)$-orbit graphs of primitive $n$-squared origamis in $\calH(2)$ or in the Prym loci of $\calH(4)$ or $\calH(6)$ with the parabolic generators. The number of faces of length at most four in $\mathcal{G}_{n}$ is $o(|\mathcal{G}_{n}|)$.
\end{theorem}

\begin{proof}[Proof of the first part of Theorem~\ref{thm:main} given Theorem~\ref{thm:submain-1}]
    Let $\chi = \chi(\mathcal{G}_{n})$ and let $V$ be the number of vertices in the graph, $E$ the number of edges, and $F$ the number of faces for the minimal embedding of $\mathcal{G}_{n}$. Then we have, by definition,
    \begin{equation}\label{eqn:F-for-chi}
      V-E+F = \chi \Rightarrow F = E-V+\chi,
    \end{equation}
    and, since each vertex has degree $4$, the degree sum formula yields 
    \begin{equation}\label{eqn:E-for-V}
        2E = 4V \Rightarrow 3E = 6V.
    \end{equation}
    
    Now, let $f_{i}$ denote the number of faces of length $i$. Hence, $F = \sum_{i}f_{i}$. Observe that, since every edge lies in two faces,
    \begin{equation}
        \begin{split}
        2E = \sum_{i} i\cdot f_{i} = \sum_{i\geq 5}i\cdot f_{i} + \sum_{i=1}^{4} i \cdot f_{i} &\geq \sum_{i\geq 5}5\cdot f_{i} + \sum_{i=1}^{4} (5-(5-i)) \cdot f_{i} \\
        &= 5\sum_{i}f_{i} - \sum_{i = 1}^{4}(5-i)\cdot f_{i} \\
        &= 5F - \sum_{i = 1}^{4}(5-i)\cdot f_{i}. 
        \end{split} \label{eqn:main-lower-bound}
    \end{equation}
    We now substitute ~\eqref{eqn:F-for-chi} into ~\eqref{eqn:main-lower-bound} to get
    \[
        2E\geq 5E - 5V + 5\chi - \sum_{i = 1}^{4}(5-i)\cdot f_{i} \Rightarrow -5\chi \geq 3E - 5V - \sum_{i = 1}^{4}(5-i)\cdot f_{i}.
    \]
    From this, on substituting ~\eqref{eqn:E-for-V}, we get
    \[-5\chi \geq V - \sum_{i = 1}^{4}(5-i)\cdot f_{i}.\]
    Finally, by Theorem~\ref{thm:submain-1}, we have $\sum_{i=1}^{4}f_{i} = o(V)$, and so we see that
    \[-5\chi\geq V - o(V),\]
    and so $-\chi$ goes to infinity with $V$ (which goes to infinity with $n$).
\end{proof}

We prove Theorem~\ref{thm:main} for the elliptic generators by establishing that the genus of the orbit graph agrees with the genus of the associated \teichmuller curve. The latter is known to go to infinity in genus two by results of Mukamel~\cite{Muk}. This result for $\calH(2)$ confirms~\cite[Conjecture 1.2]{JM} in the previous work of the authors. For the Prym loci in $\calH(4)$ and $\calH(6)$, we use combinatorial methods to establish that the genus goes to infinity.

%%%%%%%%%%%%%%%%%%%%%%%%%%%%%%%%%%%%%%%%%

\subsection{Conditional extensions to $\calH(1,1)$ and $\calH(4)$}

As discussed above, the $\SL(2,\Z)$-orbits of primitive origamis have only been classified in $\calH(2)$ and in the Prym loci in $\calH(4)$ and $\calH(6)$. In $\calH(1,1)$, Zmiaikou made the following conjecture supported by computational evidence. 

\begin{conjecture}[{\cite[Conjecture 1]{Zm}}]\label{conj:zmiaikou}
    For any $n\geq7$, there are exactly two orbits of primitive $n$-squared origamis in the stratum $\calH(1,1)$. Moreover, when $n$ is odd, the positive integers
    $$\mathcal{A}_n(1,1) = \frac{1}{24}n^2(n-3)(n-5)\prod_{\substack{p|n\\p\text{ prime}}}\left(1-\frac{1}{p^2}\right),$$
    and
    $$\mathcal{S}_n(1,1) = \frac{1}{8}n^2(n-1)(n-3)\prod_{\substack{p|n\\p\text{ prime}}}\left(1-\frac{1}{p^2}\right)$$
    are the cardinals of the orbits of $n$-squared origamis with monodromy group $\Alt(n)$ and $\Sym(n)$, respectively.
    When $n$ is even, the positive integers
    $$\mathcal{A}_n(1,1) = \frac{1}{24}n^3(n-2)\prod_{\substack{p|n\\p\text{ prime}}}\left(1-\frac{1}{p^2}\right),$$
    and
    $$\mathcal{S}_n(1,1) = \frac{1}{8}n^2(n-2)(n-4)\prod_{\substack{p|n\\p\text{ prime}}}\left(1-\frac{1}{p^2}\right)$$
    are the cardinals of the orbits of $n$-squared origamis with monodromy group $\Alt(n)$ and $\Sym(n)$, respectively.
\end{conjecture}

In particular, this would imply that the orbits are of size $\Theta(n^4)$. We direct the reader to Section~\ref{sec:prelims} for the definition of the monodromy group of an origami. Assuming this conjecture, we prove an analogue of Theorem~\ref{thm:main} in this setting.

\begin{theorem}\label{thm:submain-2}
    Assuming Conjecture~\ref{conj:zmiaikou}, let $(\mathcal{G}_{n})_{n}$ be a family of $\SL(2,\Z)$-orbit graphs of primitive $n$-squared origamis in $\mathcal{H}(1,1)$ with the parabolic generators. Then the number of faces of length at most four in $\mathcal{G}_{n}$ is $o(|\mathcal{G}_{n}|)$ and, \emph{a fortiori}, one has $|\chi(\mathcal{G}_{n})|\to\infty$ as $n\to\infty$. 
\end{theorem}

In general, arithmetic \teichmuller curves in $\calH(1,1)$ are associated to subloci $W_{d^2}[n]$ of the moduli space $\mathcal{M}_{2}$. The loci $W_{d^{2}}[1]$ are projections of the $\SL(2,\R)$-orbits of primitive origamis. See Section~\ref{sec:H(1,1)} for the general definition. For $n$ odd, $W_{d^2}[n]$ is already a disjoint union of $W_{d^2}^{\epsilon}[n]$ for $\epsilon\in\{0,1\}$ (see Section~\ref{sec:H(1,1)} for the definitions). In this setting, the following conjecture has been made (see, for example,~\cite[Conjecture 1.1]{Dur}).

\begin{conjecture}[Parity Conjecture]\label{conj:parity}
    Provided that $(d,n) \neq (2,1),(3,1),(4,1)$ or $(5,1)$, it holds that $W_{d^2}[n]$ is irreducible when $n$ is even, and consists of two irreducible components when $n$ is odd (exactly $W_{d^2}^{\epsilon}[n]$ for $\epsilon\in\{0,1\}$).
\end{conjecture}

Assuming this conjecture, we can also prove:

\begin{theorem}\label{thm:submain-3}
    Assuming Conjecture~\ref{conj:parity}, fix $n>1$ and let $(\mathcal{G}_{d})_{d}$ be a family of $\SL(2,\Z)$-orbit graphs of primitive $dn$-squared origamis in $\mathcal{H}(1,1)$ associated to the curves $W_{d^{2}}[n]$ and with the parabolic generators. Then the number of faces of length at most four in $\mathcal{G}_{d}$ is $o(|\mathcal{G}_{d}|)$ and, \emph{a fortiori}, one has $|\chi(\mathcal{G}_{d})|\to\infty$ as $d\to\infty$.
\end{theorem}

In proving these theorems, we extend work of Mukamel~\cite{Muk} counting orbifold points on \teichmuller curves in $\calH(2)$ by proving the following. Here, $\widetilde{h}$ is a reduced class number defined by
\[\tilde{h}(-D) = \frac{2h(-D)}{|\mathcal{O}_{-D}^{\times}|}.\]

\begin{theorem-e3}
    Let $e_{3}$ denote the number of order three orbifold points. If $n$ is even, then $e_{3}(W_{d^2}[n]) = 0$ for all $d\geq 2$. Otherwise, for odd $n$, if $d\equiv 0\!\!\mod 3$, we have
    \[e_{3}(W_{d^2}^{\epsilon}[n]) = \left\lbrace\begin{array}{cl}
        \frac{1}{2}(3\tilde{h}(-d^2/3) + \tilde{h}(-3d^2)), & \text{if $n = 3$ and $\epsilon = 0$} \\
        0, & \text{otherwise},
    \end{array}\right.\]
    if $d\equiv 1\!\!\mod 3$, we have
    \[e_{3}(W_{d^2}^{\epsilon}[n]) = \left\lbrace\begin{array}{cl}
        \frac{1}{2}\tilde{h}(-3d^2), & \text{if $n = 1$ and $\epsilon = 0$} \\
        0, & \text{otherwise},
    \end{array}\right.\]
    and, if $d\equiv 2\!\!\mod 3$, we have
    \[e_{3}(W_{d^2}^{\epsilon}[n]) = \left\lbrace\begin{array}{cl}
        \frac{1}{2}\tilde{h}(-3d^2), & \text{if $n = 3$ and $\epsilon = 0$} \\
        0, & \text{otherwise}.
    \end{array}\right.\]
\end{theorem-e3}

The $\SL(2,\Z)$-orbits of primitive origamis in $\calH(4)$ outside of the Prym locus are subject to the following conjecture of Delecroix--Leli\`evre (supported by substantial computational evidence). We quote it as it appears in~\cite{MMY}. See Section~\ref{sec:prelims} for the definition of the HLK-invariant.

\begin{conjecture}[{\cite[Conjecture 6.8]{MMY}}]\label{conj:Del-Lel}
    For $n > 8$, the number of $\SL(2,\Z)$-orbits of primitive $n$-squared origamis in $\calH(4)$ is as follows:
    \begin{itemize}
        \item there are precisely two such $\SL(2,\Z)$-orbits in $\odd(4)$ outside of the Prym locus, distinguished by their monodromy groups being $\Alt(n)$ or $\Sym(n)$;
        \item for odd $n$, there are precisely four such $\SL(2,\Z)$-orbits in $\hyp(4)$, distinguished by their HLK-invariant being $(4,[1,1,1]), (2,[3,1,1]), (0,[5,1,1])$ or $(0,[3,3,1])$;
        \item for even $n$, there are precisely three such $\SL(2,\Z)$-orbits in $\hyp(4)$, distinguished by their HLK-invariant being $(3,[2,2,0]), (1,[4,2,0])$ or $(1,[2,2,2])$.
    \end{itemize}
\end{conjecture}

Implicit in the conjecture (in the belief of the authors) is that each orbit has size $\Omega(n^{\dim_\C\calH(4) - 1}) = \Omega(n^5)$. Assuming this conjecture, we can prove the following.

\begin{theorem}\label{thm:submain-4}
    Assuming Conjecture~\ref{conj:Del-Lel}, let $(\mathcal{G}_{n})_{n}$ be a family of $\SL(2,\Z)$-orbit graphs of primitive $n$-squared origamis in $\hyp(4)$ with HLK-invariant $(1,[4,2,0])$ or in the non-Prym locus of $\odd(4)$ with symmetric monodromy group, using the parabolic generators in both cases. Then the number of faces of length at most four in $\mathcal{G}_{n}$ is $o(|\mathcal{G}_{n}|)$ and, \emph{a fortiori}, one has $|\chi(\mathcal{G}_{n})|\to\infty$ as $n\to\infty$.
\end{theorem}

In this setting, we rule out the existence of orbifold points by considering a block system (in the sense of imprimitive permutation groups) for the action of $\SL(2,\Z)$ on the orbits.

%%%%%%%%%%%%%%%%%%%%%%%%%%%%%%%%%%%%%%%%%

\subsection{A conditional extension to low-complexity strata}

Based on computational evidence and aligning with Conjectures~\ref{conj:zmiaikou},~\ref{conj:parity}, and~\ref{conj:Del-Lel}, some experts believe in the following conjecture concerning $\SL(2,\Z)$-orbit growth in general strata.

\begin{conjecture}\label{conj:wild}
    Given a family $(\mathcal{G}_{n})_{n}$ of $\SL(2,\Z)$-orbit graphs of primitive $n$-squared origamis in a connected component of a stratum $\mathcal{H}:=\calH_{g}(k_{1},\dots,k_{s})$ not contained in a strictly smaller arithmetic subvariety, we have $|\mathcal{G}_{n}| = \Omega(n^{\dim_{\C}\mathcal{H}-1}) = \Omega(n^{2g+s-2})$.
\end{conjecture}

Assuming this conjecture, we can establish behaviour similar to Theorem~\ref{thm:main} in low complexity strata. Here, we again use the parabolic generators.

\begin{theorem}\label{thm:general}
    Assuming Conjecture~\ref{conj:wild}, let $(\mathcal{G}_{n})_{n}$ be a family of $\SL(2,\Z)$-orbit graphs of primitive $n$-squared origamis in a connected component of a stratum $\calH(k_{1},\ldots,k_s)$ with $1\leq s\leq 2$ and not contained in a strictly smaller arithmetic subvariety. Then there exists a density one subsequence $A\subset\N$ so that, for $n\in A$, the number of faces of length at most four in $\mathcal{G}_{n}$ is $o(|\mathcal{G}_{n}|)$ and, \emph{a fortiori}, one has $|\chi(\mathcal{G}_{n})|\to\infty$ as $n$ goes to infinity in $A$.
\end{theorem}

The requirement $s\leq 2$ is a consequence of the limits of our methods for counting faces corresponding to parabolic elements in $\SL(2,\Z)$. Indeed, we count such faces by counting integer points on certain algebraic hypersurfaces and the state-of-the-art bounds only allow us to handle $s\leq 2$.

%%%%%%%%%%%%%%%%%%%%%%%%%%%%%%%%%%%%%%%%%

\subsection{Applications to arithmetic \teichmuller curve genera}

For any origami $X$ with $-I$ contained in the Veech group $\SL(X,\omega)$, we can again prove that the orbit graph with the elliptic generators has the same genus as the associated \teichmuller curve. Generalising the combinatorial method of proving that the orbit graphs with parabolic generators have genus going to infinity, we can establish the same for the elliptic generators and achieve the following results.

First, in the Prym loci of $\calH(4)$ and $\calH(6)$ we establish:

\begin{theorem}\label{thm:prym-genus}
    Let $W_{n^{2}}$ be a Prym \teichmuller curve associated to an orbit of primitive $n$-squared origamis in $\calH(4)$ or $\calH(6)$. Then the genus of $W_{n^{2}}$ goes to infinity with $n$.
\end{theorem}

Torres-Teigel--Zachhuber~\cite{TTZ18,TTZ19} established this result for non-square discriminants, and so this completes the analysis.

Second, in $\calH(1,1)$, we get:

\begin{theorem}\label{thm:H(1,1)-genus}
    Fix $n\geq 1$. Assuming Conjectures~\ref{conj:zmiaikou} and~\ref{conj:parity}, the genus of $W_{d^{2}}[n]$ goes to infinity as $d$ goes to infinity.
\end{theorem}

Finally, applying the techniques used to prove Theorem~\ref{thm:general}, we get:

\begin{theorem}\label{thm:general-genus}
    Let $(\mathcal{G}_{n})_{n}$ be a family of $\SL(2,\Z)$-orbit graphs of primitive $n$-squared origamis in a hyperelliptic connected component and not contained in a strictly smaller arithmetic subvariety, and let $\mathcal{C}_{n}$ be the associated arithmetic \teichmuller curves. Then, assuming Conjecture~\ref{conj:wild}, there exists a density one subsequence $A\subset\N$ so that the genus of $\mathcal{C}_{n}$ goes to infinity as $n$ goes to infinity in $A$.
\end{theorem}

%%%%%%%%%%%%%%%%%%%%%%%%%%%%%%%%%%%%%%%%%

\subsection{Proof strategy} Each face in an orbit graph corresponds to a word in the generators of $\SL(2,\Z)$ (up to inversion and conjugation). We count faces differently depending on whether the corresponding element in $\SL(2,\Z)$ is hyperbolic, parabolic, or elliptic. Hyperbolic faces give rise to pseudo-Anosov diffeomorphisms on the underlying surface and we count such faces by counting the corresponding pseudo-Anosovs using results of Arnoux--Yoccoz~\cite{AY} and Ivanov~\cite{Iva}. Parabolic faces of bounded length, it turns out, can be counted (up to a constant factor) by determining the number of loops in the graph. Each loop corresponds (up to bounded multiplicity) to an integral solution of some algebraic hypersurface. We use methods of Bombieri--Pila~\cite{BoPi} and Browning--Gorodnik~\cite{BG} to count these. Finally, elliptic faces correspond to orbifold points on the associated \teichmuller curves and so we use known counts of Mukamel~\cite{Muk}, and Torres-Teigel--Zachhuber~\cite{TTZ18,TTZ19}, or produce new counts of such orbifold points to bound these faces.

%%%%%%%%%%%%%%%%%%%%%%%%%%%%%%%%%%%%%%%%%

\subsection{Outline of the paper} In Section~\ref{sec:prelims}, we give the background on origamis and their $\SL(2,\Z)$-orbits. In Section~\ref{sec:blocks}, we describe block systems (in the sense of permutation groups) for the action of $\SL(2,\Z)$ on an orbit. These block systems allow us to rule out certain faces in some of the orbit graphs, but we believe this material to be of independent interest. In Section~\ref{sec:T-S-H(2)}, we count the faces in the $\calH(2)$ orbit graphs corresponding to hyperbolic, parabolic and elliptic elements of $\SL(2,\Z)$ and establish Theorem~\ref{thm:submain-1}. In Section~\ref{sec:prym}, we extend these counts to the Prym loci in $\calH(4)$ and $\calH(6)$. In Section~\ref{sec:H(1,1)}, we remind the reader of what is known about $\SL(2,\Z)$-orbits of origamis in $\calH(1,1)$ and complete the proof of Theorems~\ref{thm:submain-2} and ~\ref{thm:submain-3}. In Section~\ref{sec:non-prym}, we do the same for two families of non-Prym orbits in $\calH(4)$ and prove Theorem~\ref{thm:submain-4}. In Section~\ref{sec:general}, we prove Theorem~\ref{thm:general}. Finally, in Section~\ref{sec:curve-genus}, we analyse the orbit graphs built with the elliptic generators, completing the proof of Theorem~\ref{thm:main} and establishing Theorems~\ref{thm:prym-genus},~\ref{thm:H(1,1)-genus} and~\ref{thm:general-genus}.

%%%%%%%%%%%%%%%%%%%%%%%%%%%%%%%%%%%%%%%%%

\subsection{Acknowledgements} We thank Matthew de Courcy-Ireland for very useful discussions and the first author thanks the Banff International Research Station for their hospitality while these discussions took place. We also thank Pascal Hubert for suggesting the proof method used in Subsection~\ref{subsec:gen-ells}. For the purpose of open access, the authors have applied a Creative Commons Attribution (CC BY) licence to any Author Accepted Manuscript version arising from this submission.

%%%%%%%%%%%%%%%%%%%%%%%%%%%%%%%%%%%%%%%%%
%%%%%%%%%%%%%%%%%%%%%%%%%%%%%%%%%%%%%%%%%

\section{Preliminaries}\label{sec:prelims}

Here, we remind the reader of the necessary background on origamis and their $\SL(2,\Z)$-orbits.

%%%%%%%%%%%%%%%%%%%%%%%%%%%%%%%%%%%%%%%%%

\subsection{Translation surfaces and their moduli spaces} A compact Riemann surface $M$ of genus $g\geq 1$ equipped with a non-trivial Abelian differential $\omega$ is called a \textit{translation surface}. This nomenclature is justified by the fact that the local primitives of $\omega$ away from the set $\Sigma$ of its zeroes yields a collection of charts on $M\setminus\Sigma$ whose changes of coordinates are given by translations of the complex plane $\mathbb{C}$. By the Riemann-Roch theorem, $\omega$ has $2g-2$ zeroes counted with multiplicities, i.e., if $s=|\Sigma|$ is the cardinality of $\Sigma$ and $k_1,\dots,k_{s}$ are the vanishing orders of $\omega$ at the elements of $\Sigma$, then $\sum\limits_{j=1}^{s}k_j=2g-2$. 

By gathering together translation surfaces $X=(M,\omega)$ with a prescribed list $\underline{\kappa}=(k_1,\dots,k_{s})$ of orders of zeroes of $\omega$, one obtains a \textit{stratum} $\mathcal{H}(\underline{\kappa})$ of the moduli space of translation surfaces of genus $g$. This is a complex orbifold of dimension $2g+s-1$ carrying a natural $\SL(2,\mathbb{R})$ action (consisting of post-composing the local primitives of $\omega$ with the usual action of $\SL(2,\mathbb{R})$ on $\mathbb{R}^2=\mathbb{C}$). For further information and references about these objects, the reader is encouraged to consult Athreya--Masur's book \cite{AtMa}.

%%%%%%%%%%%%%%%%%%%%%%%%%%%%%%%%%%%%%%%%%

\subsection{Origamis and their ${\SL(2,\Z)}$-orbits} An \textit{origami} (or square-tiled surface) is a translation surface $(M,\pi^*(dz))$ obtained from a finite cover $\pi:M\to \mathbb{T}^2=\mathbb{C}/(\mathbb{Z}+i\mathbb{Z})$ possibly branched only at $0\in\mathbb{T}^2$. Alternatively, an origami is constructed from a pair $(h,v)$ of permutations acting transitively on $n$ symbols by taking unit squares $sq(i)$, $i=1,\dots,n$, and gluing the rightmost vertical side of $sq(i)$ to the leftmost vertical side of $sq(h(i))$ and the topmost horizontal side of $sq(i)$ to the bottommost horizontal side of $sq(v(i))$. An origami is said to be \textit{primitive} if it is not a cover of another origami different from itself or the unit-square torus. Equivalently, an origami is primitive if the \textit{monodromy group} $\langle h, v\rangle\leqslant\Sym(n)$ is primitive as a permutation group (see Section~\ref{sec:blocks} for a definition).

Restricting the action of $\SL(2,\R)$ on translation surfaces, the group $\SL(2,\mathbb{Z})$ leaves invariant the set of primitive origamis. In fact, $\SL(2,\mathbb{Z})$ is generated by the matrices 
\[T = \begin{bmatrix}
    1 & 1 \\
    0 & 1
\end{bmatrix} \quad \textrm{and} \quad 
S = \begin{bmatrix}
    1 & 0 \\
    1 & 1
\end{bmatrix}\] and one can check that their actions on pairs of permutations are $T(h,v)=(h,vh^{-1})$ and $S(h,v)=(hv^{-1},v)$ (see, for example,~\cite[Figure 2.2]{JM}).

Some of the origamis $\pi:(M,\omega)\to(\mathbb{T}^2,dz)$ studied in this paper come equipped with an involution $\iota$ of $M$ taking $\omega$ to $-\omega$. In this case, one can count the number $l_0$ of fixed points of $\iota$ over $0\in\mathbb{T}^2$ and the numbers $l_1, l_2, l_3$ of fixed points of $\iota$ over the other three $2$-torsion points of $\mathbb{T}^2$. As it turns out, the pair $(l_0,[l_1,l_2,l_3])$, where $[l_1,l_2,l_3]$ is an unordered triple, is an invariant of the $\SL(2,\mathbb{Z})$ orbit of $(M,\omega)$ called its \textit{HLK-invariant}.  

The HLK-invariant was originally used by Hubert and Leli\`evre \cite{HL} to distinguish between two $\SL(2,\mathbb{Z})$ orbits (called A and B) of primitive origamis in $\mathcal{H}(2)$ tiled by a prime number of unit squares. In general, by putting together the works \cite{HL}, \cite{McM} and \cite{LR}, if $X$ is a primitive origami in $\mathcal{H}(2)$ tiled by $n\geq 3$ squares, then 
\begin{itemize}
    \item $X$ falls into a single $\SL(2,\mathbb{Z})$ orbit whenever $n=3$ or $n\geq 4$ is even, and 
    \item $X$ falls into one of two possible $\SL(2,\mathbb{Z})$ orbits (called A and B) whenever $n\geq 5$ is odd: the A orbit has cardinality $\frac{3}{16}(n-1)n^2\prod\limits_{p|n, p \textrm{ prime}}(1-p^{-2})$ and the B orbit has cardinality $\frac{3}{16}(n-3)n^2\prod\limits_{p|n, p \textrm{ prime}}(1-p^{-2})$.
\end{itemize}
In the language of HLK-invariants, the orbit for even $n\geq 4$ corresponds to the HLK-invariant $(1,[2,2,0])$, the A orbit and the case of $n=3$ corresponds to the HLK-invariant $(0,[3,1,1])$, and the B orbit corresponds to the HLK-invariant $(2,[1,1,1])$.

Note that in all cases the size of the orbit is $\Theta(n^{3})$.

For $n = 3$ and $n\geq 4$ even, let $\mathcal{G}_{n}$ denote the 4-regular orbit graph; that is, the graph whose vertices are $n$-squared primitive origamis in $\calH(2)$ with edges given by the action of $T$ and $S$. For $n\geq 5$ odd, let $\mathcal{G}_{n}^{A}$ and $\mathcal{G}_{n}^{B}$ denote the 4-regular orbit graphs for the A and B orbits, respectively. These orbit graphs have as vertices the primitive $n$-squared origamis in $\calH(2)$ with edges coming from the action of $T$ and $S$. See Figure~\ref{fig:G3} for a directed version of the graph $\mathcal{G}_{3}$.

In previous work of the authors~\cite{JM}, we proved that $\mathcal{G}_{3}$ and $\mathcal{G}_{5}^{B}$ are the only orbit graphs that are planar. We extend this result in the current work by proving that in fact the genera of the orbits graphs go to infinity with $n$.

As noted in the introduction, we can also form graphs by using the generators $R:=\begin{bsmallmatrix}
    0 & -1 \\
    1 & 0
\end{bsmallmatrix} = T^{-1}ST^{-1}$ and $U:=\begin{bsmallmatrix}
    0 & 1 \\
    -1 & 1
\end{bsmallmatrix} = TS^{-1}$. Eventual non-planarity for these graphs was not established by the authors but was conjectured to hold true~\cite[Conjecture 1.2]{JM}.

Each orbit is associated to an algebraic curve in the moduli space $\mathcal{M}_{2}$. Indeed, for each $n$-squared $\SL(2,\Z)$-orbit and some choice of vertex stabiliser (i.e., Veech group) $\SL(X,\omega)$, we get a \teichmuller curve $\mathbb{H}/\SL(X,\omega)$ in $\mathcal{M}_{2}$. The primitive \teichmuller curves coming from translation surfaces in $\calH(2)$ are called Weierstrass curves and are denoted by $W_{D}$ for positive discriminants $D$. The Weierstrass curves coming from $n$-squared origamis have discriminant $D = n^2$. See~\cite{McM} for more details.

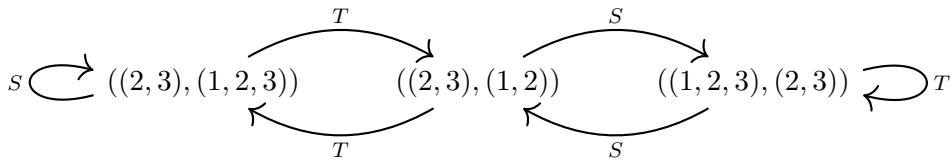
\begin{figure}[b]
\begin{center}
\begin{tikzcd}[cells={nodes={}}, ]
        \arrow[loop left, distance=3em, start anchor={[yshift=-1ex]west}, end anchor={[yshift=1ex]west}, line width = 0.8pt]{}{S} \arrow[bend left, line width = 0.8pt]{r}{T} ((2,3),(1,2,3))  
        & \arrow[bend left, line width = 0.8pt]{l}{T} ((2,3),(1,2)) \arrow[bend left, line width = 0.8pt]{r}{S} & \arrow[bend left, line width = 0.8pt]{l}{S} ((1,2,3),(2,3)) \arrow[loop right, distance=3em, start anchor={[yshift=1ex]east}, end anchor={[yshift=-1ex]east}, line width = 0.8pt]{}{T}
\end{tikzcd}
\end{center}
\caption{The graph $\mathcal{G}_{3}$ with the paraboic generators is the undirected version of this graph.}
\label{fig:G3}
\end{figure}

%%%%%%%%%%%%%%%%%%%%%%%%%%%%%%%%%%%%%%%%%

\subsection{Surface parameters and cusps}\label{subsec:params} 

\begin{figure}[b]
    \centering
    \begin{tikzpicture}[scale = 0.8, line width = 1.5pt]
        \draw (0,0) -- node[rotate = 90]{$|$}(2,1) -- node[rotate = 90]{$||$}(3,4) -- node{$|$}(7,4) -- node[rotate = 90]{$||$}(6,1) -- node{$||$}(9,1) -- node[rotate = 90]{$|$}(7,0) -- node{$||$}(4,0) -- node{$|$} cycle;
        \draw[|<->|,line width = 1pt] (1.5,1) -- node[left]{$h_{1}$}(1.5,4);
        \draw[|<->|,line width = 1pt] (-0.5,0) -- node[left]{$h_{2}$}(-0.5,1);
        \draw[|<->|,line width = 1pt] (3,4.5) -- node[above]{$w_{1}$}(7,4.5);
        \draw[|<->|,line width = 1pt] (0,-0.5) -- node[below]{$w_{2}$}(7,-0.5);
        \draw[|<->,line width = 1pt] (2,4.5) -- node[above]{$t_{1}$}(3,4.5);
        \draw[<->|,line width = 1pt] (7,-0.5) -- node[above]{$t_{2}$}(9,-0.5);
        \foreach \x/\y in {0/0,2/1,3/4,7/4,6/1,9/1,7/0,4/0}{
            \node at (\x,\y) {$\bullet$};
        }
    \end{tikzpicture}
    \caption{Two-cylinder surface parameters in $\calH(2)$.}
    \label{fig:H2-params}
\end{figure}
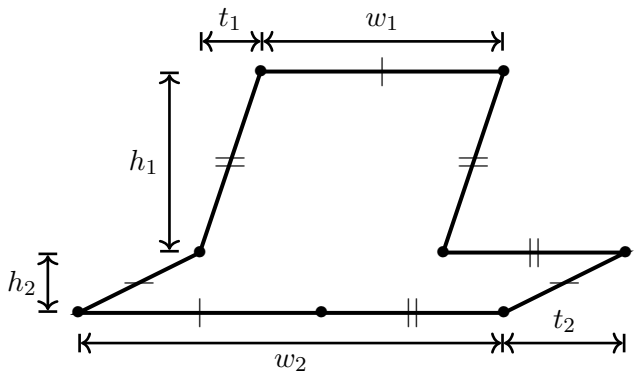

The \emph{cusp} of an origami, $X$, is its orbit under the action of the horizontal shear $T$ and the \emph{cusp width} is the size of this orbit (i.e., the minimal $i$ for which $T^{i}(X) = X$).

Any origami in $\calH(2)$ has one or two horizontal cylinders. If $X$ is a two-cylinder $n$-squared origami then it is characterised by the widths $w_1$, $w_2$ and heights $h_1$, $h_2$ of these cylinders, and the twists $t_1$, $t_2$ (i.e., the relative positions of the conical singularity in the boundaries of these cylinders). See Figure \ref{fig:H2-params}. Note that these parameters satisfy $h_1w_1+h_2w_2=n$ (because the total area of $X$ is $n$). It can then be checked that the origami has cusp width
\[\cuspwidth.\]
The \textit{cusp representative} of such an $X$ is the unique surface in the cusp with $0\leq t_{i} < \gcd(w_{i},h_{i})$ (see \cite[Lemma 3.1]{HL}).
For $n\geq 4$, the cusp width of a one-cylinder origami is $n$ (for $n = 3$ the unique one-cylinder origami has cusp width 1). Since the cusp width grows linearly with $n$, we will not need to consider one-cylinder cusps in the discussion of short faces below.

%%%%%%%%%%%%%%%%%%%%%%%%%%%%%%%%%%%%%%%%%
%%%%%%%%%%%%%%%%%%%%%%%%%%%%%%%%%%%%%%%%%

\section{Block systems for the action of $\SL(2,\Z)$}\label{sec:blocks}

Suppose that a group $G$ acts on some discrete set $\Omega$ and think of $G$ as a subgroup of the symmetric group $\Sym(\Omega)$. A partition $\mathscr{B} = \{\Delta_{1},\ldots,\Delta_{k}\}$ of $\Omega$ is said to be a \emph{block system} for $G$ if for all $\Delta\in \mathscr{B}$ and all $g\in G$ either $g(\Delta) = \Delta$ or $g(\Delta)\cap\Delta=\varnothing$. If the action of $G$ is transitive then it is said to also be \emph{primitive} if the only block systems that exist for $G$ are the trivial block systems $\mathscr{B} = \{\Omega\}$ and $\mathscr{B}' = \{\{x\}\,:\,x\in\Omega\}$. We will analyse the block systems for the permutation action of $\SL(2,\Z)$ on an orbit of an origami.

%%%%%%%%%%%%%%%%%%%%%%%%%%%%%%%%%%%%%%%%%

\subsection{Block systems for orbits of origamis admitting anti-involutions}

Here, we give a block system for the action of $\SL(2,\Z)$ on its orbits of origamis admitting anti-involutions.  The block systems depend on the number of distinct values in the unordered triple $[l_{1},l_{2},l_{3}]$ of the HLK-invariant.

\subsubsection{The case ${l_{1}\neq l_{2}\neq l_{3}\neq l_{1}}$}

\begin{proposition}\label{prop:all_diff}
    Let $\orb$ be an $\SL(2,\Z)$-orbit of primitive origamis all having $-I$ in their Veech group, with $|\orb|\geq 6$, and such that the HLK-invariant $(l_{0},[l_{1},l_{2},l_{3}])$ of any origami in $\orb$ has $l_{1}\neq l_{2}\neq l_{3}\neq l_{1}$. Define the following subsets of $\orb$:
    \[B_{1} = \left\{X \in \orb\,\,: \,\,\text{the ordered HLK-triple is }(l_{1},l_{2},l_{3})\right\},\]
    \[B_{2} = \left\{X \in \orb\,\,: \,\,\text{the ordered HLK-triple is }(l_{2},l_{1},l_{3})\right\},\]
    \[B_{3} = \left\{X \in \orb\,\,: \,\,\text{the ordered HLK-triple is }(l_{2},l_{3},l_{1})\right\},\]
    \[B_{4} = \left\{X \in \orb\,\,: \,\,\text{the ordered HLK-triple is }(l_{3},l_{2},l_{1})\right\},\]
    \[B_{5} = \left\{X \in \orb\,\,: \,\,\text{the ordered HLK-triple is }(l_{3},l_{1},l_{2})\right\},\]
    \[B_{6} = \left\{X \in \orb\,\,: \,\,\text{the ordered HLK-triple is }(l_{1},l_{3},l_{2})\right\}.\]
    Then the sets $\{B_{i}\}_{i = 1}^{6}$ form a block system with the dynamics shown in Figure~\ref{fig:blocksys1}, the sets $\{B_{1}\cup B_{6}, B_{2}\cup B_{5}, B_{3}\cup B_{4}\}$ form a block system with the dynamics shown in Figure~\ref{fig:blocksys2}, and the sets $\{B_{1}\cup B_{3}\cup B_{5},B_{2}\cup B_{4}\cup B_{6}\}$ form a block system with dynamics shown in Figure~\ref{fig:blocksys3}.
\end{proposition}

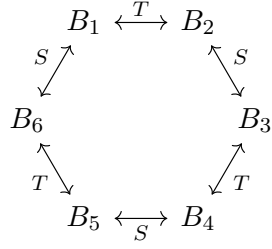
\begin{figure}[t]
\begin{center}
\begin{tikzcd}[column sep={0.75*1cm,between origins}, row sep={0.75*1.732050808cm,between origins},]
    & B_{1} \arrow[rr, "T", leftrightarrow] \arrow[ld, "S"', leftrightarrow] && B_{2} \arrow[rd, "S", leftrightarrow] &  \\
    B_{6} \arrow[rd, "T"', leftrightarrow]&  &&  & B_{3} \\
    & B_{5} \arrow[rr, "S"', leftrightarrow] && B_{4} \arrow[ru, "T"', leftrightarrow]
\end{tikzcd}
\end{center}
\caption{Block system for $\{B_{i}\}_{i = 1}^{6}$.}
\label{fig:blocksys1}
\end{figure}

\begin{figure}[t]
\begin{center}
\begin{tikzcd}[cells={nodes={}}]
        \arrow[loop left, distance=3em, start anchor={[yshift=-1ex]west}, end anchor={[yshift=1ex]west}]{}{S} \arrow[bend left]{r}{T} B_{1}\cup B_{6}  
        & \arrow[bend left]{l}{T} B_{2}\cup B_{5} \arrow[bend left]{r}{S} & \arrow[bend left]{l}{S} B_{3}\cup B_{4} \arrow[loop right, distance=3em, start anchor={[yshift=1ex]east}, end anchor={[yshift=-1ex]east}]{}{T}
\end{tikzcd}
\end{center}
\caption{Block system for $\{B_{1}\cup B_{6}, B_{2}\cup B_{5}, B_{3}\cup B_{4}\}$.}
\label{fig:blocksys2}
\end{figure}
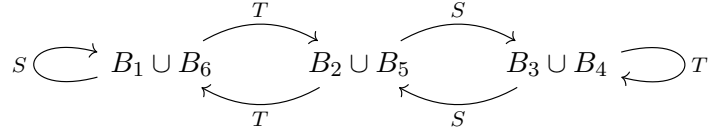

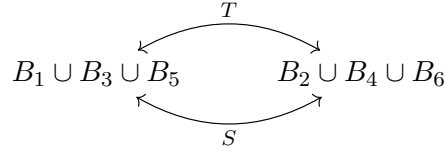
\begin{figure}[t]
\begin{center}
\begin{tikzcd}[cells={nodes={}}]
        \arrow[bend left, leftrightarrow]{r}{T} B_{1}\cup B_{3}\cup B_{5}  
        & \arrow[bend left, leftrightarrow]{l}{S} B_{2}\cup B_{4}\cup B_{6}
\end{tikzcd}
\end{center}
\caption{Block system for $\{B_{1}\cup B_{3}\cup B_{5},B_{2}\cup B_{4}\cup B_{6}\}$.}
\label{fig:blocksys3}
\end{figure}

\begin{proof}
We must prove that the sets above form a partition and that moreover this partition is invariant under the action of $\SL(2,\Z)$.

{\bf Partition:} By the definition of the HLK-invariant, any surface in this orbit lies in one of the sets $B_{i}$. These sets are disjoint as the $l_{i}$ differ. Furthermore, $\SL(2,\Z)$ acts by $\Sym(3)$ on the values $l_{1},l_{2}$, and $l_{3}$. Indeed, this can be seen by considering the action of $T$ and $S$ on the points $(1/2,0),(0,1/2)$ and $(1/2,1/2)$. Hence, we see that each $B_{i}$ is non-empty if $\orb$ is non-empty and that this is a non-trivial partition as long as $|\orb|>6$. 

{\bf Invariance:} The claimed actions of $\SL(2,\Z)$ on the sets $B_{i}$ and their unions follow from the fact that $\SL(2,\Z)$ has the action of $\Sym(3)$ on $l_{1}, l_{2}$, and $l_{3}$.
\end{proof}

\subsubsection{The case ${l_{1}\neq l_{2} = l_{3}}$}

\begin{proposition}\label{prop:l_1diff}
Let $\orb$ be an $\SL(2,\Z)$-orbit of primitive origamis all having $-I$ in their Veech group, with $|\orb|\geq3$, and such that the HLK-invariant $(l_{0},[l_{1},l_{2},l_{3}])$ of any origami in $\orb$ has $l_{1}\neq l_{2} =l_{3}$. Define the following subsets of $\orb$:
\[B_{1} = \left\{X \in \orb\,\,: \,\,\text{the number of fixed points above }\left(0,\frac{1}{2}\right)\text{ is }l_{1}\right\},\]
\[B_{2} = \left\{X\in \orb\,\,:\,\,\text{the number of fixed points above }\left(\frac{1}{2},\frac{1}{2}\right)\text{ is }l_{1}\right\},\]
\[B_{3} = \left\{X\in \orb\,\,:\,\,\text{the number of fixed points above }\left(\frac{1}{2},0\right)\text{ is }l_{1}\right\}.\]
Then the sets $B_{1}$, $B_{2}$, and $B_{3}$ form a block system for the action of $\SL(2,\Z)$ on $\orb$ with the dynamics shown in Figure~\ref{fig:blocksys4}.
\end{proposition}

\begin{figure}[t]
\begin{center}
\begin{tikzcd}[cells={nodes={}}]
        \arrow[loop left, distance=3em, start anchor={[yshift=-1ex]west}, end anchor={[yshift=1ex]west}]{}{S} \arrow[bend left]{r}{T} B_{1}  
        & \arrow[bend left]{l}{T} B_{2} \arrow[bend left]{r}{S} & \arrow[bend left]{l}{S} B_{3} \arrow[loop right, distance=3em, start anchor={[yshift=1ex]east}, end anchor={[yshift=-1ex]east}]{}{T}
\end{tikzcd}
\end{center}
\caption{Block system for $\{B_{i}\}_{i = 1}^{3}$.}
\label{fig:blocksys4}
\end{figure}
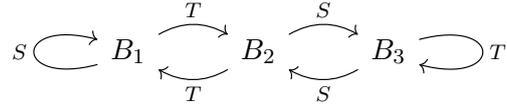

\begin{proof}
Again, we must prove that the sets above form a partition and that moreover this partition is invariant under the action of $\SL(2,\Z)$.

{\bf Partition:} By the definition of the HLK-invariant, any surface in this orbit lies in one of the sets $B_{i}$. These sets are disjoint since $l_{1}$ differs from $l_{2}$ and $l_{3}$. Further, since $\SL(2,\Z)$ acts by $\Sym(3)$ on the values $l_{1},l_{2}$, and $l_{3}$, we see that each $B_{i}$ is non-empty if $\orb$ is non-empty and that this is a non-trivial partition as long as $|\orb|>3$.

{\bf Invariance:} The claimed action of $\SL(2,\Z)$ on the sets $B_{i}$ follows from the fact that $\SL(2,\Z)$ has the action of $\Sym(3)$ on $l_{1}, l_{2}$, and $l_{3}$.
\end{proof}

%%%%%%%%%%%%%%%%%%%%%%%%%%%%%%%%%%%%%%%%%

\subsection{A block system for orbits with symmetric monodromy}

Here we give a block system for $\SL(2,\Z)$-orbits whose monodromy groups are the full symmetric group.

\begin{proposition}\label{prop:symblocksys}
Let $\orb$ be an $\SL(2,\Z)$-orbit of $n$-squared primitive origamis with $|\orb|\geq 3$ and such that the monodromy group of any origami in $\orb$ is the symmetric group $\Sym(n)$. Define the following subsets of $\orb$:
\[B_{1} = \{X = (h,v)\in \orb\,\,|\,\,h\not\in \Alt(n), v\in \Alt(n)\},\]
\[B_{2} = \{X = (h,v)\in \orb\,\,|\,\,h\not\in \Alt(n), v\not\in \Alt(n)\},\]
\[B_{3} = \{X = (h,v)\in \orb\,\,|\,\,h\in \Alt(n), v\not\in \Alt(n)\}.\]
Then the sets $B_{1}$, $B_{2}$, and $B_{3}$ form a block system for the action of $\SL(2,\Z)$ on $\orb$ with the dynamics shown in Figure~\ref{fig:blocksys4}.
\end{proposition}

\begin{proof}
The proof is similar to the proofs above.

{\bf Partition:} We first prove that $\{B_{i}\}$ is a partition of $\orb$. Since we have $\sigma:=\langle h, v\rangle\cong \Sym(n)$ for all $X = (h,v)\in \orb$ we cannot have both $h$ and $v$ being elements of $\Alt(n)$. Therefore, every origami in $\orb$ lies in exactly one of the sets $B_{i}$. We must prove that each $B_{i}$ is nonempty. Suppose first that $X = (h,v) \in \orb$ lies in $B_{1}$. That is, we have $h\not\in \Alt(n)$ and $v\in \Alt(n)$. Under the action of $T$, $X$ is mapped to the origami $T(X) = (h,vh^{-1})$. Since $h\not\in \Alt(n)$ and $v\in \Alt(n)$ we have that $vh^{-1}\not\in \Alt(n)$. Hence, $T(X) \in B_{2}$. Now consider the image of $T(X)$ under the action of $S$. We have $S(T(X)) = (h^2v^{-1},vh^{-1})$ and we see that $h^2v^{-1} \in \Alt(n)$ and again $vh^{-1}\not\in \Alt(n)$. Hence, $S(T(X)) \in B_{3}$. A similar argument also works if $X$ was initially in $B_{2}$ or $B_{3}$. We have therefore shown that $\{B_{i}\}$ is indeed a partition of $\orb$. Note that this will be the trivial partition if $|\orb| = 3$.

{\bf Invariance:} Let $X = (h,v) \in B_{1}$. Then we have $h\not\in \Alt(n)$ and $v\in \Alt(n)$. Since $h^{-i}\in \Alt(n)$ if and only if $i$ is even, we have $T^{2k+1}(X)\in B_{2}$ and $T^{2k}(X)\in B_{1}$. A symmetric argument applies if $X\in B_{2}$. Now suppose that $X\in B_{3}$. Then we have $h\in \Alt(n)$ and $v\not\in \Alt(n)$. In particular, $h^{-i}\in \Alt(n)$ for all $i$ and so $vh^{-i}\not\in \Alt(n)$ for all $i$. Hence, $T^{i}(X)\in B_{3}$ for all $i$. So we have demonstrated that
\[T(B_{1}) = B_{2}, \,\,\,\,\,T(B_{2}) = B_{1},\,\,\,\,\,\,\text{and}\,\,\,\,\,\,T(B_{3}) = B_{3}.\]
A symmetric argument considering the powers of $v$ gives us that 
\[S(B_{1}) = B_{1}, \,\,\,\,\,S(B_{2}) = B_{3},\,\,\,\,\,\,\text{and}\,\,\,\,\,\,S(B_{3}) = B_{2}.\]
\end{proof}

\begin{corollary}
Let $\orb$ be an $\SL(2,\Z)$-orbit as in the statement of Proposition~\ref{prop:symblocksys}. Then the proportion of 1-cylinder origamis in $\orb$ that also have one vertical cylinder is at most $1/2$.
\end{corollary}

\begin{proof}
If $\orb$ contains no 1-cylinder origamis with a single vertical cylinder then we are done. Otherwise, let $X = (h,v)$ be such an origami. By the primitivity of $X$, we then have $h$ and $v$ both being $n$-cycles. Since the monodromy group is $\Sym(n)$ we cannot have $n$ being odd as this would give us that $h,v\in \Alt(n)$ and so $\langle h,v\rangle \not \cong \Sym(n)$. We therefore have that $n$ is even and hence $h$ and $v$ are both not in $\Alt(n)$. That is, $X $ must lie in $B_{2}$. Any other 1-cylinder origami in the same cusp as $X$ having a single vertical cylinder must be an image of $X$ under a multiple of $T^{2}$. Hence at most half of the 1-cylinder origamis in the same cusp as $X$ have one vertical cylinder. This is true for all 1-cylinder cusps and so the statement follows.
\end{proof}

\begin{remark}
    We remark that Proposition~\ref{prop:symblocksys} works with $\Sym(n)$ and $\Alt(n)$ replaced by any monodromy group $G$ and an index two subgroup $H$.
\end{remark}

%%%%%%%%%%%%%%%%%%%%%%%%%%%%%%%%%%%%%%%%%

\subsection{A block system for origamis without an anti-involution}

Here, we give a block system for the action of $\SL(2,\Z)$ on primitive origamis that do not admit $-I$ as a symmetry.

\begin{proposition}\label{prop:nonprymnonhyp}
Let $\orb$ be an $\SL(2,\Z)$-orbit of primitive origamis that do not admit $-I$ as a symmetry (that is, $-I$ is not in the Veech group of any origami in $\orb$). Define the following collection of sets
\[\mathscr{B}:= \left\{\{X,-I(X)\}\right\}_{X\in \orb}.\]
Then $\mathscr{B}$ is a block system for the action of $\SL(2,\Z)$ on $X$.
\end{proposition}

\begin{proof}
Again, the proof is similar to the proofs above.

{\bf Partition:} Since the origamis in $\orb$ are not symmetric by $-I$ we have that $-I(X)\neq X$ for all $X\in \orb$. Hence, all sets in $\mathscr{B}$ are genuine pairs. Disjointness follows from $-I^2 = I$.

{\bf Invariance:} This follows from the fact that $-I$ is central in $\SL(2,\Z)$.
\end{proof}

%%%%%%%%%%%%%%%%%%%%%%%%%%%%%%%%%%%%%%%%%
%%%%%%%%%%%%%%%%%%%%%%%%%%%%%%%%%%%%%%%%%

\section{Parabolic generators in $\calH(2)$}\label{sec:T-S-H(2)}

Here, we prove Theorem~\ref{thm:submain-1}.

%%%%%%%%%%%%%%%%%%%%%%%%%%%%%%%%%%%%%%%%%

\subsection{Hyperbolic faces in $\calH(2)$}\label{subsec:hyp-H2}

Of the words of length at most four in $T$ and $S$ (up to cyclic permutation and inversion), $ST, ST^2, S^2T, ST^3, S^2T^2, S^3T, (ST)^2$, and $(TS)^{-1}ST$ all give rise to hyperbolic elements in $\SL(2,\Z)$. Such a hyperbolic element corresponds to a pseudo-Anosov diffeomorphism on the underlying smooth surface (it can be realised by moving to the flat surface in the $\GL(2,\R)$-orbit that diagonalises the action of the corresponding element of $\SL(2,\Z)$).

We remark that, when $n$ is even and in the A-orbit for $n$ odd, we have the block system of Figure~\ref{fig:blocksys4}. This rules out $ST, ST^3, S^3T$ and $(TS)^{-1}ST$.

Since we have a finite list of words, the stretch factors of the corresponding pseudo-Anosov diffeomorphisms are bounded from above by some constant $D\geq1$. We can then make use of the following theorem of Arnoux-Yoccoz~\cite{AY} and Ivanov~\cite{Iva} which we quote as it appears in~\cite{FM}:

\begin{theorem}[{\cite[Theorem 14.9]{FM}}]
    For any $g,n\geq 0$ and $D\geq 1$, the mapping class group $\Mod(S_{g,n})$ contains only finitely many conjugacy classes of pseudo-Anosov diffeomorphism with stretch factor bounded above by $D$.
\end{theorem}

We obtain the following corollary.

\begin{corollary}\label{cor:hyperbolic}
    Let $M$ be a fixed hyperbolic element of $\SL(2,\Z)$, and fix a stratum of translation surfaces $\calH$. There exists an $N$ so that for all $n\geq N$, the Veech group of any $n$-squared origami in $\calH$ does not contain $M$.
\end{corollary}

As a result, there are no faces in $\mathcal{G}_{n}$ corresponding to the words listed above once $n$ is large enough. In particular, the number of hyperbolic faces of length at most four is $o(|\mathcal{G}_{n}|)$. The following remark discusses computational estimates using~\cite{flatsurf,surf-dyn,Sage} for how large $n$ might need to be for the faces to stop appearing.

\begin{remark}
    The origamis $((1,2,3,4,5),(3,4,5))$ and $((3,4,5),(1,2,3,5,4))$ in $\mathcal{G}_{5}^{B}$ are fixed by $ST$. Computational experiments suggest that these are the only two primitive origamis fixed by $ST$ in $\calH(2)$.

    The origamis $((1,2,3,4),(2,3,4))$ in $\mathcal{G}_{4}$, $((1,2)(3,4,5),(1,3,2,4,5))$ in $\mathcal{G}_{5}^{A}$ and $$((4,5,6,7),(1,2,3,4,7,6,5))$$ in $\mathcal{G}_{7}^{A}$ are fixed by $ST^{2}$, while $((2,3,4),(1,2,4,3))$ in $\mathcal{G}_{4}$, $((1,2,3,4,5),(1,4,2)(3,5))$ in $\mathcal{G}_{5}^{A}$ and $((1,2,3,4,5,6,7),(4,5,6,7))$ in $\mathcal{G}_{7}^{A}$ are fixed by $S^2T$. Computationally, these appear to be the only primitive origamis in $\calH(2)$ fixed by either of these words.

    All three origamis in $\mathcal{G}_{3}$, the origamis $((2,3,4),(1,2,3,4))$ and $((1,2,3,4),(2,4,3))$ in $\mathcal{G}_{4}$, the origamis $((2,3,4,5),(1,2,5,4,3))$ and $((1,2,3,4,5),(2,3,4,5))$ in $\mathcal{G}_{5}^{A}$, and the origamis $((1,2)(3,4)(5,6,7),(1,3,5,6)(2,4,7))$ and $((1,2,3)(4,5,6,7),(1,4)(2,5)(3,7,6))$ in $\mathcal{G}_{7}^{A}$ are fixed by $S^{2}T^{2}$. The origamis $((3,4,5),(1,2,3))$ and $((1,2,3,4,5),(1,2,4,3,5))$ in $\mathcal{G}_{5}^{B}$ are fixed by $(TS)^{-1}ST$. The origamis $((1,2,3,4,5,6,7),(3,6,4,7,5))$ in $\mathcal{G}_{7}^{B}$ and $$((5,6,7,8,9),(1,2,3,4,5,9,8,7,6))$$ in $\mathcal{G}_{9}^{B}$ are fixed by $ST^{3}$. Finally, the origamis $((3,4,5,6,7),(1,2,3,6,4,7,5))$ in $\mathcal{G}_{7}^{B}$ and $$((1,2,3,4,5,6,7,8,9),(5,6,7,8,9))$$ in $\mathcal{G}_{9}^{B}$ are fixed by $S^{3}T$. Computations suggest that these are the only primitive origamis in $\calH(2)$ fixed by any of these words. Notice that $(ST)^2$ does not seem to appear in the Veech group of any primitive origami in $\calH(2)$.
\end{remark}

%%%%%%%%%%%%%%%%%%%%%%%%%%%%%%%%%%%%%%%%%

\subsection{Parabolic faces in $\calH(2)$}\label{subsec:para-H2}

Of the words of length at most four in $T$, $S$ and their inverses (again up to cyclic permutation and inversion), the words $T, S, T^2, S^2, T^3, S^3, T^4, S^4$, and $ST^{-2}S$ give rise to parabolic elements of $\SL(2,\Z)$.

Note that, when $n$ is even and in the A-orbit for $n$ odd, the block system of Figure~\ref{fig:blocksys4} does not rule out any of the above words.
\begin{remark}\label{rem:cusps}
    The words that are powers of $T$ and $S$ correspond to cusps on the \teichmuller curve and it suffices in this instance to use the known result (see~\cite{McM},~\cite{Bai} and~\cite{Muk}) that the total number of cusps and hence also the number of cusps of length at most 4 is $o(|\mathcal{G}_{n}|)$. We use a different method of counting here since it will apply below in situations where the cusp count is not already known.
\end{remark}

We will begin by bounding the number of origamis in an orbit that are fixed by $T$; i.e., the number of loops labelled by $T$ in the graph. Since each loop labelled by $T$ fixing some origami $X$ has a corresponding loop labelled by $S$ fixing the surface $R(X)$ where $R = T^{-1}ST^{-1}$ is anti-clockwise rotation by $\frac{\pi}{2}$, counting $T$-labelled loops gives a count of all loops.

Once $n\geq 4$, one-cylinder origamis have cusp width $n$ and so the only origamis that can be fixed by $T$ are two-cylinder origamis. Recall that the cusp width of a two-cylinder origami with surface parameters $(w_1,h_1,t_1,w_2,h_2,t_2)$ is
\[\cuspwidth.\]
Hence, if we require the cusp to have width one, then we need $\gcd(w_{1},h_{1}) = w_{1}$ and $\gcd(w_{2},h_{2}) = w_{2}$. Therefore, we must have positive integers $x$ and $y$ such that $h_{1} = xw_{1}$ and $h_{2} = yw_{2}$.

Recalling that the origamis have area $n$ so that $w_{1}h_{1}+w_{2}h_{2} = n$. We are therefore looking for lattice points on the ellipse with equation
\[xw_{1}^{2} + yw_{2}^{2} = n.\]

Given such a lattice point $(w_{1},w_{2})$ on the above ellipse, we can obtain the origamis with surface parameters $(w_{1},xw_{1},t_1,w_{2},yw_{2},t_{2})$ for any choice of $0\leq t_{i} < \gcd(w_{i},h_{i}) = w_{i}$. So each lattice point determines $w_{1}w_{2}$ distinct origamis fixed by $T$.

For example, when $n = 5$ the only ellipse of the above form that contains a lattice point is the ellipse
\[w_{1}^{2} + w_{2}^{2} = 5\]
with the solution $(w_{1},w_{2}) = (1,2)$ (recalling that $w_{1} < w_{2}$). This gives the two origamis with surface parameters $(1,1,0,2,2,0)$ and $(1,1,0,2,2,1)$. The first lies in the A orbit while the second lies in the B orbit. The images of these origamis under $R$ will be fixed points for $S$. Hence, each orbit contains two loops.

We will make use of the results of Bombieri-Pila~\cite[Section 3]{BoPi} that prove that an irreducible algebraic curve of degree $d$ contained inside a square of side length $N$ contains $O(N^{\frac{1}{d}+\epsilon})$ lattice points, where the implied constant only depends on $\epsilon>0$ and $d$. Our ellipses are irreducible algebraic curves of degree two.

Let $1\leq \min\{x,y\} < \sqrt{n}$. Each such ellipse is then contained inside the square of side length at most $N = \sqrt{n}$. It therefore contains $O(n^{\frac{1}{4}+\epsilon})$ lattice points. Given the bound on $\min\{x,y\}$, we have $O(n^{\frac{3}{2}})$ such ellipses and on each ellipse $w_{1},w_{2} = O(N) = O(\sqrt{n})$. Hence, each lattice point determines $w_{1}w_{2}=O(n)$ distinct origamis. So, in total, we get $O(n^{\frac{11}{4} + \epsilon})$ origamis.

Now suppose that $\sqrt{n}\leq x,y \leq n$. Here, each ellipse is contained inside a square of side length at most $N = n^{\frac{1}{4}}$. So each ellipse contains $O(n^{\frac{1}{8}+\epsilon})$ lattice points. This time, given the bounds on $x$ and $y$, we have $O(n^{2})$ possible ellipses. However, now that $w_{1},w_{2} = O(N) = O(n^{\frac{1}{4}})$, each lattice point only determines $w_1w_2 = O(n^{\frac{1}{2}})$ distinct origamis. So, in total, we get $O(n^{\frac{21}{8}+\epsilon})$ origamis.

Hence, we have $O(n^{\frac{11}{4}+\epsilon})$ loops in the graph.

Consider now surfaces fixed by $T^2$ or $S^2$. Again, it suffices to only count surfaces fixed by $T^2$ (we can overcount those fixed by $T$). Here, we want to count the two-cylinder origamis whose cusp width
\[\cuspwidth\]
is equal to two. Hence, we have $\gcd(w_{1},h_{1}) = \frac{w_{1}}{2}$ and $\gcd(w_{2},h_{2}) = w_{2}$, giving $h_{1} = x\frac{w_{1}}{2}$ for odd $x$ and $h_{2} = yw_{2}$, or vice versa (i.e., $h_{1} = xw_{1}$ and $h_{2} = y\frac{w_{2}}{2}$ for odd $y$).

So, now we will be searching for lattice points on ellipses of the form
\[\frac{x}{2}w_{1}^{2} + yw_{2}^{2} = n\]
or
\[xw_{1}^{2} + \frac{y}{2}w_{2}^{2} = n.\]

For example, when $n = 5$, only the ellipse
\[3w_{1}^{2} + \frac{1}{2}w_{2}^{2} = 5\]
has an integer solution. This solution being $(w_{1},w_{2}) = (1,2)$ giving rise to a single cusp whose representative has surface parameters $(1,3,0,2,1,0)$. This origami lies in the A orbit.

Since we have a finite number of ellipses with similarly bounded coefficients, the argument from above applies exactly as before and we get $O(n^{\frac{11}{4}+\epsilon})$ origamis fixed by $T^2$ or $S^2$.

This also works for $T^3, S^3, T^4$, and $S^4$.

Finally, we observe that $ST^{-2}S = \begin{bmatrix}
    -1 & -2 \\
    0 & -1
\end{bmatrix}$. Recall that all of our origamis have $-I$ in their Veech group. In particular, an origami fixed by $ST^{-2}S$ also has $\begin{bmatrix}
    1 & 2 \\
    0 & 1
\end{bmatrix} = T^{2}$ in its Veech group. That is, it lies in a cusp of width two or in a loop. We have bounded these above. So, since every face corresponding to $ST^{-2}S$ is attached to a cusp of width two or to a loop, the bound for cusps of width two and loops can be used here which is again $O(n^{\frac{11}{4}+\epsilon})$.

Hence, there are at most $O(n^{\frac{11}{4}+\epsilon})$ parabolic faces of length at most four in $\mathcal{G}_{n}$. Note that this is $o(|G_{n}|) = o(n^3)$.

%%%%%%%%%%%%%%%%%%%%%%%%%%%%%%%%%%%%%%%%%

\subsection{Elliptic faces in $\calH(2)$}\label{subsec:ell-H2}

Of the words of length at most four in $T$, $S$ and their inverses (again up to cyclic permutation and inversion), the words $S^{-1}T, T^{-1}ST^{-1} = R, ST^{-1}S, S^{-1}T^3, T^{-1}S^3, S^{-1}TST$, $T^{-1}STS$, and $(S^{-1}T)^2$ give rise to elliptic elements of $\SL(2,\Z)$.

Recall that for each $n$-squared $\SL(2,\Z)$-orbit and some choice of vertex stabiliser (i.e., Veech group) $\SL(X,\omega)$, we get a \teichmuller curve $\mathbb{H}/\SL(X,\omega)$ in $\mathcal{M}_{2}$ isomorphic to $W_{n^{2}}$. Each origami fixed by an elliptic element of $\SL(2,\Z)$ (other than $-I$ which is a global symmetry) gives rise to an orbifold point on $W_{n^{2}}$.

Since the finite order elements of $\SL(2,\Z)$ have orders 1, 2, 3, 4, or 6 and $-I$ is a global symmetry, the orbifold points on $W_{n^{2}}$ have orders 2 or 3. It follows from {\cite[Theorem 1.1]{Muk}} that all orbifold points on $W_{n^{2}}$ in fact have order 2.

The words $S^{-1}T, S^{-1}T^{3}, T^{-1}S^3, S^{-1}TST, T^{-1}STS$ and $S^{-1}TS^{-1}T$ all cube to $\pm I$ and so would give rise to orbifold points of order 3. Hence, there are no faces in $\mathcal{G}_{n}$ corresponding to these words. We remark that, when $n$ is even and in the A-orbit for $n$ odd, the block system of Figure~\ref{fig:blocksys4} rules out all of these words anyway.

The words $T^{-1}ST^{-1} = R$ and $ST^{-1}S$ both square to $-I$ and correspond to orbifold points on $W_{n^{2}}$ of order 2. It follows from {\cite[Proposition 4.7]{Muk}} that the number of orbifold points on $W_{n^{2}}$ of order 2 is at most $\frac{n^2}{2} = O(n^{2})$.

Hence, the number of elliptic faces of length at most four (in fact, the number of elliptic faces in total) is $O(n^2)$ and hence is $o(|G_{n}|) = o(n^3)$. This completes the proof of Theorem~\ref{thm:submain-1} in the case of $\calH(2)$.

%%%%%%%%%%%%%%%%%%%%%%%%%%%%%%%%%%%%%%%%%
%%%%%%%%%%%%%%%%%%%%%%%%%%%%%%%%%%%%%%%%%

\section{Prym loci in $\calH(4)$ and $\calH(6)$}\label{sec:prym}

Here we extend the above arguments to the setting of Prym loci in $\calH(4)$ and $\calH(6)$. We begin with a brief reminder of the main structures in this setting.

%%%%%%%%%%%%%%%%%%%%%%%%%%%%%%%%%%%%%%%%%

\subsection{Background}

The Prym locus $\Omega\mathcal{E}_{D}(4)$ is the subset of $\calH(4)$ consisting of those translation surfaces $(M,\omega)$ for which $M$ admits a holomorphic involution $\iota$ with 4 fixed points taking $\omega$ to $-\omega$, and admitting real multiplication by $\mathcal{O}_{D}$ with $\mathcal{O}_{D}\cdot\omega\subset\C\cdot\omega$. Lanneau-Nguyen~\cite{LN14} classify the $\GL^{+}(2,\R)$ connected components of $\Omega\mathcal{E}_{D}(4)$ and we direct the reader to their paper for more details on Prym loci. They prove that, for $D\geq 17$,  $\Omega\mathcal{E}_{D}(4)$ is non-empty if and only if $D\equiv 0,1,4\!\!\mod\! 8$.

Lanneau-Nguyen, as a consequence of their determination of the connected components of $\Omega\mathcal{E}_{D}(4)$, give the following classification of the $\SL(2,\Z)$-orbits of Prym origamis in $\calH(4)$.

\begin{theorem}[{\cite[Proposition B.1 and Corollary B.2]{LN14}}]
    Fix $n\geq 5$ and let $X$ be a primitive origami that is a Prym eigenform in $\calH(4)$. Then
    \begin{itemize}
        \item if $n$ is odd, there is a single $\SL(2,\Z)$-orbit of such origamis and the eigenforms have discriminant $D = n^{2}$.
        \item if $n\equiv 0\mod 4$ or $n = 6$, there is a single $\SL(2,\Z)$-orbit of such origamis and the eigenforms have discriminant $D = n^{2}$.
        \item if $n\equiv 2\mod 4$, $n\geq 10$, there are two $\SL(2,\Z)$-orbits of such origamis: one containing eigenforms of discriminant $D = n^{2}$ and one containing eigenforms of discriminant $D = \frac{n^{2}}{4}$.
    \end{itemize}
\end{theorem}

In the language of HLK-invariants, when $n$ is odd the origamis have HLK-invariant $(0,[1,1,1])$; when $n\equiv 0\mod 4$ or $n = 6$ the origamis have HLK-invariant $(1,[2,0,0])$; and when $n\equiv 2\mod 4$, $n\geq 10$, then the origamis of discriminant $n^2$ have HLK-invariant $(1,[2,0,0])$ while the origamis of discriminant $\frac{n^{2}}{4}$ have HLK-invariant $(3,[0,0,0])$.

Similarly, the Prym locus $\Omega\mathcal{E}_{D}(6)$ is the subset of $\calH(6)$ consisting of those $(M,\omega)$ for which $M$ admits a holomorphic involution $\iota$ with 2 fixed points taking $\omega$ to $-\omega$, and admitting real multiplication by $\mathcal{O}_{D}$ with $\mathcal{O}_{D}\cdot\omega\subset\C\cdot\omega$. Lanneau-Nguyen~\cite{LN20} classify the $\GL^{+}(2,\R)$ connected components of $\Omega\mathcal{E}_{D}(6)$ and we again direct the reader to their paper for more details on Prym loci. They prove that, for $D\equiv 0,1\!\!\mod\! 4$ with $D\neq 4,9$,  $\Omega\mathcal{E}_{D}(6)$ is non-empty and connected. Moreoever, they show that $\Omega\mathcal{E}_{4}(6)$ and $\Omega\mathcal{E}_{9}(6)$ are both empty.

The relevant result for origamis is the following.

\begin{theorem}[{\cite[Theorem 1.2]{LN20}}]
    Fix an even $n\geq 8$ and let $X$ be a primitive origami that is a Prym eigenform in $\calH(6)$. Then there is a single $\SL(2,\Z)$-orbit of such origamis and the eigenforms have discriminant $D = \frac{n^{2}}{4}$.

    If $n$ is odd, then there are no such origamis.
\end{theorem}

Here, the origamis have HLK-invariant $(1,[0,0,0])$.

All orbits again have size $\Omega(n^{3})$ and can be turned into 4-regular graphs as in $\calH(2)$.

Origamis in the Prym locus in $\calH(4)$ can have one, two or three cylinders, while the origamis in the Prym locus in $\calH(6)$ can have two or four cylinders.

The cusps of one-cylinder $n$-squared origamis have width $n$, while the cusps of two-cylinder $n$-squared origamis have width $n/2$. Both growing linearly with $n$, they do not feature in our consideration of short faces.

A three-cylinder origami in the Prym locus of $\calH(4)$ is given by surface parameters $(w_1,h_1,t_1,w_2,h_2,t_2)$ as shown in Figure~\ref{fig:H4-params} while a four-cylinder origami in the Prym locus of $\calH(6)$ has similar surface parameters as shown in Figure~\ref{fig:H6-params}. In both cases, the cusp width is given by the same formula used for two-cylinder origamis in $\calH(2)$:
\[\cuspwidth.\]

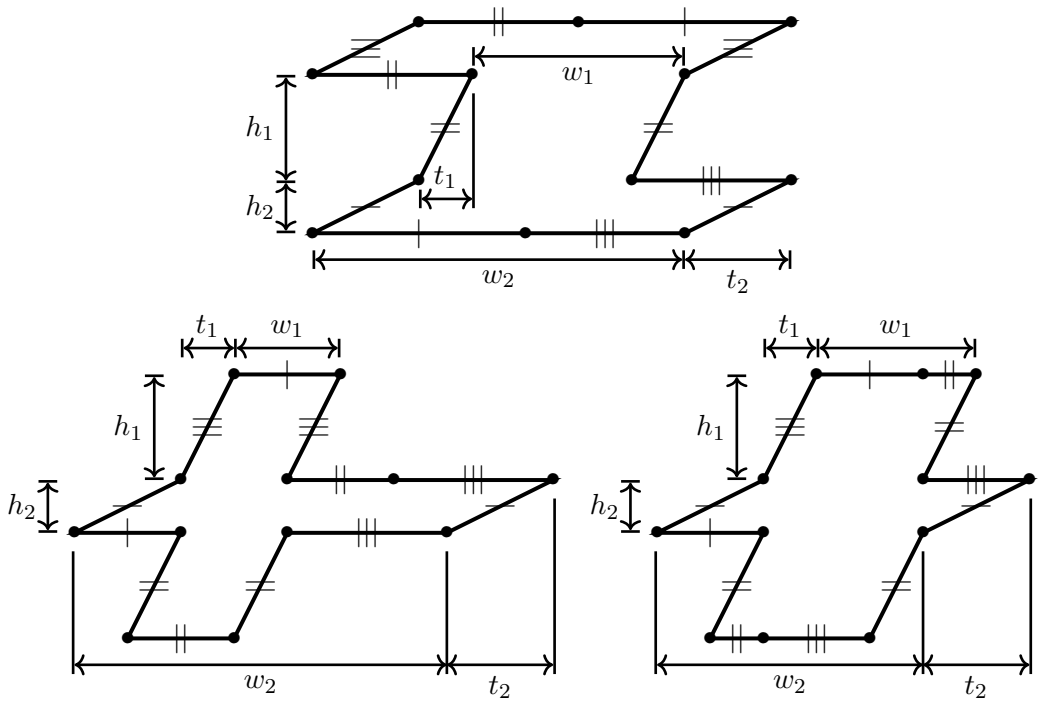
\begin{figure}[t]
    \centering
    \begin{tikzpicture}[scale = 0.7, line width = 1.5pt]
        \draw (0,0) -- node[rotate = 90]{$|$}(2,1) -- node[rotate = 90]{$||$}(3,3) -- node{$||$}(0,3) -- node[rotate = 90]{$|||$}(2,4) -- node{$||$}(5,4) -- node{$|$}(9,4) -- node[rotate = 90]{$|||$}(7,3) -- node[rotate = 90]{$||$}(6,1) -- node{$|||$}(9,1) -- node[rotate = 90]{$|$}(7,0) -- node{$|||$}(4,0) -- node{$|$} cycle;
        \draw[<->|,line width = 1pt] (-0.5,1) -- node[left]{$h_{1}$}(-0.5,3);
        \draw[|<->|,line width = 1pt] (-0.5,0) -- node[left]{$h_{2}$}(-0.5,1);
        \draw[|<->|,line width = 1pt] (3,3.35) -- node[below]{$w_{1}$}(7,3.35);
        \draw[|<->|,line width = 1pt] (0,-0.5) -- node[below]{$w_{2}$}(7,-0.5);
        \draw[|<->,line width = 1pt] (2,0.65) -- node[above]{$t_{1}$}(3,0.65);
        \draw[line width = 1pt] (3.025,0.475) -- (3.025,2.65);
        \draw[<->|,line width = 1pt] (7,-0.5) -- node[below]{$t_{2}$}(9,-0.5);
        \foreach \x/\y in {0/0,2/1,3/3,0/3,2/4,5/4,9/4,7/3,6/1,9/1,7/0,4/0}{
            \node at (\x,\y) {$\bullet$};
        }
    \end{tikzpicture}
    \begin{tikzpicture}[scale = 0.7, line width = 1.5pt]
        \draw (0,0) -- node[rotate = 90]{$||$}(1,2) -- node{$|$}(-1,2) -- node[rotate = 90]{$|$}(1,3) -- node[rotate = 90]{$|||$}(2,5) -- node{$|$}(4,5) -- node[rotate = 90]{$|||$}(3,3) -- node{$||$}(5,3) -- node{$|||$}(8,3) -- node[rotate = 90]{$|$}(6,2) -- node{$|||$}(3,2) -- node[rotate = 90]{$||$}(2,0) -- node{$||$} cycle;
        \draw[|<->|,line width = 1pt] (0.5,3) -- node[left]{$h_{1}$}(0.5,5);
        \draw[|<->|,line width = 1pt] (-1.5,2) -- node[left]{$h_{2}$}(-1.5,3);
        \draw[|<->|,line width = 1pt] (2,5.5) -- node[above]{$w_{1}$}(4,5.5);
        \draw[<->,line width = 1pt] (-1,-0.5) -- node[below]{$w_{2}$}(6,-0.5);
        \draw[|<->,line width = 1pt] (1,5.5) -- node[above]{$t_{1}$}(2,5.5);
        \draw[<->,line width = 1pt] (6,-0.5) -- node[below]{$t_{2}$}(8,-0.5);
        \draw[line width = 1pt] (-1.025,-0.675) -- (-1.025,1.65);
        \draw[line width = 1pt] (6,-0.675) -- (6,1.65);
        \draw[line width = 1pt] (8.025,-0.675) -- (8.025,2.65);
        \foreach \x/\y in {0/0,1/2,-1/2,1/3,2/5,4/5,3/3,5/3,8/3,6/2,3/2,2/0}{
            \node at (\x,\y) {$\bullet$};
        }
    \end{tikzpicture}
    \begin{tikzpicture}[scale = 0.7, line width = 1.5pt]
        \draw (0,0) -- node[rotate = 90]{$||$}(1,2) -- node{$|$}(-1,2) -- node[rotate = 90]{$|$}(1,3) -- node[rotate = 90]{$|||$}(2,5) -- node{$|$}(4,5) -- node{$||$}(5,5) -- node[rotate = 90]{$||$}(4,3) -- node{$|||$}(6,3) -- node[rotate = 90]{$|$}(4,2) -- node[rotate = 90]{$||$}(3,0) -- node{$|||$}(1,0) -- node{$||$} cycle;
        \draw[|<->|,line width = 1pt] (0.5,3) -- node[left]{$h_{1}$}(0.5,5);
        \draw[|<->|,line width = 1pt] (-1.5,2) -- node[left]{$h_{2}$}(-1.5,3);
        \draw[|<->|,line width = 1pt] (2,5.5) -- node[above]{$w_{1}$}(5,5.5);
        \draw[<->,line width = 1pt] (-1,-0.5) -- node[below]{$w_{2}$}(4,-0.5);
        \draw[|<->,line width = 1pt] (1,5.5) -- node[above]{$t_{1}$}(2,5.5);
        \draw[<->,line width = 1pt] (4,-0.5) -- node[below]{$t_{2}$}(6,-0.5);
        \draw[line width = 1pt] (-1.025,-0.675) -- (-1.025,1.65);
        \draw[line width = 1pt] (4,-0.675) -- (4,1.65);
        \draw[line width = 1pt] (6.025,-0.675) -- (6.025,2.65);
        \foreach \x/\y in {0/0,1/2,-1/2,1/3,2/5,4/5,5/5,4/3,6/3,4/2,3/0,1/0}{
            \node at (\x,\y) {$\bullet$};
        }
    \end{tikzpicture}
    \caption{Three-cylinder surface parameters in $\calH(4)$.}
    \label{fig:H4-params}
\end{figure}

\begin{figure}[t]
    \centering
    \begin{tikzpicture}[scale = 0.55, line width = 1.5pt]
        \draw (0,0) -- node[rotate = 90]{$|$}(2,1) -- node[rotate = 90]{$||$}(3,3) -- node{$|$}(5,3) -- node[rotate = 90]{$||$}(4,1) -- node{$||$}(6,1) -- node{$|||$}(9,1) -- node[rotate = 90]{$|$}(7,0) -- node{$|||$}(4,0) -- node[rotate = 90]{$|||$}(2,-1) -- node[rotate = 90]{$||||$}(1,-3) -- node{$||$}(-1,-3) -- node[rotate = 90]{$||||$}(0,-1) -- node{$|$}(-2,-1) -- node{$||||$}(-5,-1) -- node[rotate = 90]{$|||$}(-3,0) -- node{$||||$} cycle;
        \draw[|<->|,line width = 1pt] (1.5,1) -- node[left]{$h_{1}$}(1.5,3);
        \draw[|<->|,line width = 1pt] (-5.5,-1) -- node[left]{$h_{2}$}(-5.5,0);
        \draw[|<->|,line width = 1pt] (3,3.5) -- node[above]{$w_{1}$}(5,3.5);
        \draw[<->,line width = 1pt] (-5,-3.5) -- node[below]{$w_{2}$}(2,-3.5);
        \draw[|<->,line width = 1pt] (2,3.5) -- node[above]{$t_{1}$}(3,3.5);
        \draw[|<->|,line width = 1pt] (-5,0.5) -- node[above]{$t_{2}$}(-3,0.5);
        \draw[line width = 1pt] (-5.025,-3.675) -- (-5.025,-1.35);
        \draw[line width = 1pt] (2.025,-3.675) -- (2.025,-1.35);
        \foreach \x/\y in {0/0,2/1,3/3,5/3,4/1,6/1,9/1,7/0,4/0,2/-1,1/-3,-1/-3,0/-1,-2/-1,-5/-1,-3/0}{
            \node at (\x,\y) {$\bullet$};
        }
        \node at (2,6.5) {$\,$};
        \node at (2,-6.5) {$\,$};
    \end{tikzpicture}
    \begin{tikzpicture}[scale = 0.55, line width = 1.5pt]
        \draw (0,0) -- node[rotate = 90]{$|$}(1,3) -- node[rotate = 90]{$||$}(2,5) -- node{$|$}(3,5) -- node{$||$}(5,5) -- node[rotate = 90]{$||$}(4,3) -- node{$|||$}(6,3) -- node[rotate = 90]{$|$}(5,0) -- node{$\bigtimes$}(6,0) -- node[rotate = 90]{$|||$}(5,-3) -- node[rotate = 90]{$\bigtimes$}(4,-5) -- node{$\bigtimes$}(3,-5) -- node{$|||$}(1,-5) -- node[rotate = 90]{$\bigtimes$}(2,-3) -- node{$||$}(0,-3) -- node[rotate = 90]{$|||$}(1,0) -- node{$|$} cycle;
        \draw[|<->|,line width = 1pt] (6.5,3) -- node[right]{$h_{1}$}(6.5,5);
        \draw[|<->|,line width = 1pt] (-0.5,0) -- node[left]{$h_{2}$}(-0.5,3);
        \draw[|<->|,line width = 1pt] (2,5.5) -- node[above]{$w_{1}$}(5,5.5);
        \draw[<->,line width = 1pt] (0,-5.5) -- node[below]{$w_{2}$}(5,-5.5);
        \draw[<->,line width = 1pt] (1,5.5) -- node[above]{$t_{1}$}(2,5.5);
        \draw[|<->,line width = 1pt] (0,3.5) -- node[above]{$t_{2}$}(1,3.5);
        \draw[line width = 1pt] (-0.025,-5.675) -- (-0.025,-3.35);
        \draw[line width = 1pt] (5.025,-5.675) -- (5.025,-3.35);
        \draw[line width = 1pt] (1,3.5-0.225) -- (1,5.725);
        \foreach \x/\y in {0/0,1/3,2/5,3/5,5/5,4/3,6/3,5/0,6/0,5/-3,4/-5,3/-5,1/-5,2/-3,0/-3,1/0}{
            \node at (\x,\y) {$\bullet$};
        }
    \end{tikzpicture}
    \caption{Four-cylinder surface parameters in $\calH(6)$.}
    \label{fig:H6-params}
\end{figure}

%%%%%%%%%%%%%%%%%%%%%%%%%%%%%%%%%%%%%%%%%

\subsection{Modifications to ${\calH(2)}$ argument}

Here, we modify the bounds given for $\calH(2)$ above.

\subsubsection{Hyperbolic faces} As above, Corollary~\ref{cor:hyperbolic} proves that faces arising from words that give rise to hyperbolic elements of $\SL(2,\Z)$ do not appear in $\mathcal{G}_{n}$ once $n$ is large enough.

\subsubsection{Parabolic faces} Since, the 1-cylinder and 2-cylinder cusps in the Prym loci in $\calH(4)$ and $\calH(6)$ are of width $n$ or $n/2$, respectively, they asymptotically do not contribute to the small faces under consideration. As such, we need only bound the number of 3-cylinder and 4-cylinder cusps of length at most four. Since the formula for cusp width is the same here as it was for 2-cylinder cusps in $\calH(2)$, the same Bombieri-Pila arguments and resulting upper bounds hold in this setting as well. There is a slight modification of the coefficients in the ellipses (taking into account the new formulas for area), but the changes are bounded.
\begin{remark}As in Remark~\ref{rem:cusps}, it is known by~\cite{LN14,LN20} that the number of cusps in this situation is again $o(|\mathcal{G}_{n}|)$ which we could have also used.
\end{remark}

\subsubsection{Elliptic faces} In this setting, we have so-called Prym-\teichmuller curves $W_{n^{2}}$ and $W_{\frac{n^2}{4}}$ in $\mathcal{M}_{3}$ and $W_{\frac{n^2}{4}}$ in $\mathcal{M}_{4}$ whose orbifold points correspond to $n$-squared primitive origamis fixed by elliptic elements of $\SL(2,\Z)$.

In genus four, Torres-Teigell--Zachhuber~\cite{TTZ19} determined the number of orbifold points on $W_{D}$ for any discriminant $D>12$. The results of interest for our sake are~\cite[Lemma 7.6]{TTZ19} that $W_{D}$ has less than $D/2$ orbifold points of order two, and~\cite[Lemma 7.7]{TTZ19} that $W_{D}$ has less than $D/6$ orbifold points of order three. So, we see that the number of orbifold points on $W_{\frac{n^{2}}{4}}$ is $O(n^2)$ and we get the same bound for the number of elliptic faces as in $\calH(2)$.

In genus three, Torres-Teigell--Zachhuber~\cite{TTZ18} determined the number of orbifold points on $W_{D}$, but only for non-square discriminants $D$. We extend their work here to the case $D = n^2$ or $\frac{n^2}{4}$.

They prove~\cite[Section 5.2]{TTZ18} that the number $e_{3}(D)$ of orbifold points of order three on $W_{D}$ is $|H_{3}(D)|$ where
\begin{multline*}
    H_{3}(D) := \{(a,b,c)\in\Z^{3}:2a^2-3b^2-c^2=D,\gcd(a,b,c,f) = 1, -3\sqrt{D}<a<-\sqrt{D}, \\
    c< b\leq 0, (4a-3b-3c<0)\vee(4a-3b-3c = 0 \wedge c<3b) \},
\end{multline*}
and $f$ is the conductor of the discriminant $D$. It turns out that their proof of this fact does not require that $D$ is not a square and holds true if you take $f = \sqrt{D}$ when $D$ is a square. Half of these orbifold points will correspond to $n$-squared origamis where $D = n^{2}$ and half will correspond to $m$-squared origamis where $D = \frac{m^{2}}{4}$; i.e., $m = 2n$. The set $H_{3}(D)$ has size $O(D)$, and so we get that the number of elliptic faces in each graph is $O(n^2)$, where $n$ is the number of squares of the origamis in the graph.

\begin{remark}
    In~\cite{Zac}, Zachhuber demonstrates that, for non-square $D\equiv 1\!\!\mod 8$, the two components of the Prym-\teichmuller curve $W_{D}$ are Galois conjugate. It is possible to use the proven Galois conjugate action of~\cite[Proposition 4.7]{Zac} as the definition of a fake Galois conjugate action for three cylinder origamis. This takes an origami with an odd number of squares to an origami with twice that many squares but corresponding to the same discriminant. That is, we map a three-cylinder origami with HLK-invariant $(0,[1,1,1])$ to an origami with HLK-invariant $(3,[0,0,0])$. We can extend this map by sending one-cylinder origamis with HLK-invariant $(0,[1,1,1])$ to two-cylinder origamis with HLK-invariant $(3,[0,0,0])$. It can be checked that this map induces an isomorphism of the corresponding orbit graphs.
\end{remark}

Their proof of the count of orbifold points of order two does require that $D$ is not a square. However, it is easily modifiable to the case of square $D$. Torres-Teigell--Zachhuber define a family $\mathcal{X}\to\mathbb{P}^{1}\setminus\{0,1,\infty\}$ of genus two curves called the \textit{clover family} with the fibre above $t$ being the curve
\[\mathcal{X}_{t}: y^4 = x(x-1)(x-t).\]
We must then consider $\End(\mathcal{P}(\mathcal{X}_{t}))$, the endomorphism ring of the Prym variety of $\mathcal{X}_{t}$. For $(a,b,c)\in\Z^{3}$, they define
\[A_{\sqrt{D}}(a,b,c):=\begin{pmatrix}
    b & a\cdot\frac{1-i}{2}-c\cdot\frac{1+i}{2} \\
    a(1+i)-c(1-i) & -b
\end{pmatrix}\in\End(\mathcal{P}(\mathcal{X}_{t}))\]
and prove the following lemmas.

\begin{lemma}[{\cite[Lemma 5.2]{TTZ18}}]\label{lem:TTZ1}
    Let $A$ be an element of $\End(\mathcal{P}(\mathcal{X}_{t}))$. The following are equivalent:
    \begin{itemize}
        \item[(i)] $A$ is a self-adjoint endomorphism such that $A^2 = D$;
        \item[(ii)] $A = A_{\sqrt{D}}(a,b,c)$ for some $(a,b,c)\in\Z^{3}$ with $a^2+b^2+c^2 = D$.
    \end{itemize}
\end{lemma}

\begin{lemma}[{\cite[Lemma 5.4]{TTZ18}}]\label{lem:TTZ2}
    Suppose that $\mathcal{P}(\mathcal{X}_{t})$ admits real multiplication by $\mathcal{O}_{D}$. Then $D\equiv 0\!\!\mod 4$.
    Moreover, there is a bijection between the choices of real multiplication $\mathcal{O}_{D}\hookrightarrow \End(\mathcal{P}(\mathcal{X}_{t}))$ and triples $(a,b,c)$ appearing in Lemma~\ref{lem:TTZ1}.
\end{lemma}

Since we have $D$ being a square and $a^2+b^2+c^2 = D$, the condition $D \equiv 0 \!\!\mod 4$ forces $D = n^2$ for $n$ even and then $a,b$, and $c$ must all be even as well. In particular, $\gcd(a,b,c,n)\geq 2$.

Now, $A_{\sqrt{D}}(a,b,c)$ has eigenvectors
\[\omega(a,b,c)^{+} = \begin{pmatrix}
    \frac{i-1}{2}\cdot\frac{a-ci}{b+\sqrt{D}} \\
    1
\end{pmatrix}\,\,\,\,\,\text{and}\,\,\,\,\,\omega(a,b,c)^{-} = \begin{pmatrix}
    \frac{i-1}{2}\cdot\frac{a-ci}{b-\sqrt{D}} \\
    1
\end{pmatrix}.\]

The part of the proof of Torres-Teigell--Zachhuber that needs modifying for $D = n^2$ is the proof of \cite[Lemma 5.3]{TTZ18}. The statement and modified proof are as follows.

\begin{lemma}\label{lem:TTZ3}
    Let $D$ be a square discriminant. Then the matrices $A_{\sqrt{D}}(a,b,c)$ and $A_{\sqrt{D'}}(a',b',c')$ have the same eigenvectors if and only if
    \begin{itemize}
        \item[(i)] $D = m^2E$ and $D' = (m')^2E$ for some square discriminant $E$ and with $\gcd(m,m') = 1$; and 
        \item[(ii)] both $(a,b,c)$ and $(a',b',c')$ are integer multiples of a triple $(a_0,b_0,c_0)\in\Z^{3}$ with $a_0^{2}+b_0^{2}+c_0^{2} = E$.
    \end{itemize}
    In particular, the matrices $A_{\sqrt{D}}(a,b,c)$ and $A_{\sqrt{D}}(a',b',c')$ have the same eigenvectors if and only if $(a',b',c') = \pm(a,b,c)$ and we have $\omega(a,b,c)^{+} = \omega(-a,-b,-c)^{-}$ and $\omega(a,b,c)^{-} = \omega(-a,-b,-c)^{+}$.
\end{lemma}

\begin{proof}
    Suppose that $D = n^2$ and $D' = (n')^2$.
    
    If the matrices $A_{\sqrt{D}}(a,b,c)$ and $A_{\sqrt{D'}}(a',b',c')$ have the same eigenvectors then
    \begin{equation}\label{eqn:eigvals}
        \frac{a-ci}{b+n} = \frac{a'-c'i}{b'\pm n'}\,\,\,\,\text{and}\,\,\,\,\frac{a-ci}{b-n} = \frac{a'-c'i}{b'\mp n'}.
    \end{equation}
    This immediately forces $D = m^2E$ and $D' = (m')^2E$ for some square discriminant $E$ and we can force $\gcd(m,m') = 1$.

    Equating the real and imaginary parts of the first equality in Equation~\ref{eqn:eigvals} implies that
    \[a(b'\pm n') = a'(b+n)\,\,\,\,\,\text{and}\,\,\,\,\,c(b'\pm n')=c'(b+n).\]
    Multiplying by $c$ and $a$, respectively, and subtracting gives
    \[(a'c-ac')(b+n) = 0.\]

    If $b + n = 0$, then $a^2+b^2+c^2 = n^2$ forces $(a,b,c) = (0,-n,0)$ and also $(a',b',c') = (0,\pm n',0)$. Hence, both triples are integer multiples of $(a_0,b_0,c_0) = (0,\frac{n}{m},0)$ with $a_0^2+b_0^2+c_0^2 = E$.

    If $(a'c-ac') = \begin{vmatrix}
        a' & c' \\
        a & c
    \end{vmatrix} = 0$, then the vectors $(a,c)$ and $(a',c')$ are collinear. So $(a',c') = t(a,c)$ for some $t\in\Q$.
    
    If $a = c = 0$, then $(a,b,c) = (0,\pm n,0)$ and $(a',b',c') = (0,\pm n',0)$ and we are essentially in the previous situation again.
    
    So we may assume that one of $a$ or $c$ is non-zero. Hence, letting $t = \frac{a'}{a}$ or $\frac{c'}{c}$ appropriately, we have
    \[b'\pm n' = t(b+n) = tb+tn.\]
    Applying the same reasoning as above to the second equality in Equation~\ref{eqn:eigvals}, we will arrive at
    \[b'\mp n'= tb-tn.\]
    Therefore, we have $b' = tb$ and $n' = \pm tn$, giving $(a',b',c') = t(a,b,c) = tm(a_0,b_0,c_0)$ with $a_0^2+b_0^2+c_0^2 = E$.

    The converse follows immediately.
\end{proof}

We then obtain the following analogue of \cite[Lemma 5.5]{TTZ18}.

\begin{lemma}
    Let $D = n^2$ for even $n$. A form $\omega$ is an eigenform for real multiplication of discriminant $D$ if and only if it is the eigenform of some $A_{\sqrt{D}}(a,b,c)$ with $\gcd(a,b,c,n) = 2$.
\end{lemma}

\begin{proof}
    By Lemma~\ref{lem:TTZ2}, any choice of real multiplication corresponds to a triple $(a,b,c)\in\Z^3$ as in Lemma~\ref{lem:TTZ1}. Since, as discussed above, we have $\gcd(a,b,c,n)\geq 2$, it follows from Lemma~\ref{lem:TTZ3} that any embedding $\mathcal{O}_{D} \hookrightarrow \End(\mathcal{P}(\mathcal{X}_{t}))$ is proper if and only if $\gcd(a,b,c,n) = 2$.
\end{proof}

Replacing the definition of 
\[H_{2}(D) = \{(a,b,c)\in\Z^3:a^2+b^2+c^2 = D, \gcd(a,b,c,f) = 1\}\]
in the proof of \cite[Theorem 5.1]{TTZ18} with
\[H^{sq}_{2}(n^2) = \{(a,b,c)\in\Z^{3}:a^2+b^2+c^2 = n^2, \gcd(a,b,c,n) = 2\}\]
we get that the number of orbifold points of order two for square discriminants is
\[e_{2}(n^2) = \left\lbrace\begin{array}{cl}
    0, & \text{if }n\text{ is odd}\\
    |H^{sq}_{2}(n^2)|/24, & \text{if }n\text{ is even}.
\end{array}\right.\]

Since $|H^{sq}_{2}(n^2)| = O(n^2)$, we again get that the number of elliptic faces in the orbit graphs corresponding to orbifold points of order two is $O(n^2)$.

So, in total, the number of elliptic faces is $O(n^2)$ and hence is $o(|\mathcal{G}_{n}|) = o(n^3)$.

\subsubsection{Conclusion} We conclude that the number of faces of length at most four in the $\SL(2,\Z)$-orbit graphs $\mathcal{G}_{n}$ of primitive $n$-squared origamis in the Prym loci of $\calH(4)$ and $\calH(6)$ behaves as $o(|\mathcal{G}_{n}|)$. This completes the proof of Theorem~\ref{thm:submain-1}.

%%%%%%%%%%%%%%%%%%%%%%%%%%%%%%%%%%%%%%%%%
%%%%%%%%%%%%%%%%%%%%%%%%%%%%%%%%%%%%%%%%%

\section{Orbits in $\calH(1,1)$}\label{sec:H(1,1)}

In the stratum $\calH(1,1)$, it is known by work of McMullen~\cite{McM-dec} that the only primitive \teichmuller curve is the $\SL(2,\R)$-orbit of the decagon; that is, the one-form $dx/y$ on the curve $y^2 = x^6 - x$. The remaining \teichmuller curves in $\calH(1,1)$ project to connected components of subloci of $\mathcal{M}_{2}$ of the form
\[W_{d^2}[n] := \left\{X\in\mathcal{M}_{2}\;\middle|\;\!\begin{array}{c}\exists\text{ a primitive degree }d\text{ cover }\pi:X\to E
\text{ of an}\\ \text{elliptic curve }E,\text{ and critical points }x_1\neq x_2\\\text{ in }X\text{ with }(\pi(x_1)-\pi(x_2))\in\text{Jac}(E)[n] \end{array}\!\right\}.\]
Note that the spaces $W_{d^2}[1]$ are the projections of the $\SL(2,\R)$-orbits of primitive $d$-squared origamis in $\calH(1,1)$. In general, for $N = dn$, components of $W_{d^2}[n]$ are projections of $\SL(2,\R)$-orbits of $N$-squared origamis in $\calH(1,1)$ that admit a primitive (i.e., with no intermediate covers) degree $d$ cover of the $n$-squared torus.

The classification of the connected components of such loci is still open in general. It is known that $W_{2^2}[1]$ and $W_{3^2}[1]$ are both empty, and that $W_{4^2}[1]$ and $W_{5^2}[1]$ are both irreducible. It is also known that $W_{d^2}[n]$ for odd $n$ has at least two components. Indeed, it follows from \cite[Theorem 3.1]{Dur} that an origami in $\Omega W_{d^2}[n]$ has one of the following HLK-invariants:
\[\left\lbrace\begin{array}{cl}
   (3,[1,1,1])\text{ or }(1,[3,1,1]),  & \text{if $d$ and $n$ are both odd; or} \\
   (0,[2,2,2])\text{ or }(2,[2,2,0]), & \text{if $d$ is even and $n$ is odd; or} \\
   (0,[4,2,0]), & \text{if $n$ is even.}
\end{array}\right.\]
For odd $n$, we then define $W_{d^2}^{0}[n]$ to be the projection of the $\SL(2,\R)$-orbit of the origamis with HLK-invariant $(3,[1,1,1])$ or $(0,[2,2,2])$, and $W_{d^2}^{1}[n]$ to be the projection of the $\SL(2,\R)$-orbit of the origamis with HLK-invariant $(1,[3,1,1])$ or $(2,[2,2,0])$.

The general situation is the subject of the so-called Parity Conjecture.

\begin{conjPar}[Parity Conjecture]
    Provided that $(d,n) \neq (2,1),(3,1),(4,1)$ or $(5,1)$, it holds that $W_{d^2}[n]$ is irreducible when $n$ is even, and consists of two irreducible components when $n$ is odd (exactly $W_{d^2}^{\epsilon}[n]$ for $\epsilon\in\{0,1\}$).
\end{conjPar}

This conjecture was partially established in work of Duryev~\cite{Dur} in which he proved the following. The case of $d = 2$ and $n$ arbitrary had already been established by Huang--Wu--Zhong~\cite{HWZ}.

\begin{theorem}[{\cite[Theorem 1.2]{Dur}}]\label{thm:Dur}
    The parity conjecture holds for all $(d,n)$ such that
    \begin{itemize}
        \item[(i)] $d = 2,3,4,5$; or
        \item[(ii)] $d$ and $n$ are prime and $n > (d^3-d)/4$; or
        \item[(iii)] $d$ is prime and $n> C_{d}$, where $C_{d}$ is a constant that depends on $d$.
    \end{itemize}
\end{theorem}

In the same work, Duryev established the following result concerning orbit growth (extending previous work of Kappes--M\"oller~\cite[Theorem 1.1]{KM} to even $d$).

\begin{theorem}[{\cite[Theorem 1.8]{Dur}}]
    Let $t_{d,n,\epsilon}$ be the number of origamis in $\Omega W_{d^2}^{\epsilon}[n]$. Then for arbitrary $d$ and $n>1$ odd, we have
    \[t_{d,n,\epsilon} = \left\lbrace\begin{array}{cl}
        \frac{d-1}{12n}\cdot|\PSL(2,\Z/d\Z)|\cdot|\SL(2,\Z/n\Z)|, & \text{if $\epsilon = 0$} \\
        \\
    \frac{d-1}{4n}\cdot|\PSL(2,\Z/d\Z)|\cdot|\SL(2,\Z/n\Z)|, & \text{if $\epsilon = 1$.}
    \end{array}\right.\]
\end{theorem}

The total for even $n$ is the sum of the two numbers in the formulae above.

For $n = 1$, and odd $d\geq 5$, Kappes--M\"oller also established the formulae
\[t_{d,1,\epsilon} = \left\lbrace\begin{array}{cl}
        \frac{(d-3)(d-5)}{12d}\cdot|\PSL(2,\Z/d\Z)|, & \text{if $\epsilon = 0$} \\
        \\
    \frac{(d-1)(d-3)}{4d}\cdot|\PSL(2,\Z/d\Z)|, & \text{if $\epsilon = 1$,}
    \end{array}\right..\]
Notice that these are the formulae conjectured by Zmiaikou in Conjecture~\ref{conj:zmiaikou}. To the best of our knowledge, the case $n = 1$ for even $d$ (also conjectured by Zmiaikou) is still open.

We notice that families of orbits with fixed $n$ and increasing $d$ have growth $\Omega(d^{4}) = \Omega(N^4)$ where $N=dn$ is the number of squares (fixing $n$ is key here). Unfortunately, no such family is covered by Theorem~\ref{thm:Dur}, and so we are not able to give a theorem without dependence on the Parity Conjecture. All possible families covered by Theorem~\ref{thm:Dur} have growth too slow for the bounds on parabolic faces given below to be useful.

So, for a fixed $n\geq 1$, we let $(\mathcal{G}_{N})_{N}$ be a family of orbit graphs associated to $W_{d^{2}}^{(\epsilon)}[n]$ with $N = dn$ increasing with $d$ and using the parabolic generators.

%%%%%%%%%%%%%%%%%%%%%%%%%%%%%%%%%%%%%%%%%

\subsection{Hyperbolic faces}

The fact that the hyperbolic faces of length at most 4 eventually vanish follows again by Corollary~\ref{cor:hyperbolic}.

%%%%%%%%%%%%%%%%%%%%%%%%%%%%%%%%%%%%%%%%%

\subsection{Parabolic faces}\label{subsec:H(1,1)-paras}

As above, it will suffice to bound the number of origamis fixed by $T$. Here, we have no known bounds on the number of cusps on the associated \teichmuller curves. In $\calH(1,1)$, an origami can have one, two or three cylinders. Again, we will have that a 1-cylinder $N$-squared origami has cusp width of length $N$ and so it will not be fixed by $T$. Hence, we must bound the number of two- and three-cylinder origamis fixed by $T$.

By considering the possible separatrix diagrams in $\calH(1,1)$ (see for example~\cite[\S 3]{Shr}), an origami with two cylinders has one of the two forms shown in Figure~\ref{fig:H(1,1)-two-cylinder}. Letting $l_i$ be the lengths of the horizontal saddle connections labelled by $i$ and $w_{i}$ be the widths of the cylinders, we see that in the first case $w_{1} = l_1+l_3$ and $w_{2} = l_2+l_3$. Hence, given $(w_1,w_2)$ a choice of $1\leq l_3\leq \min\{w_1,w_2\}-1$ determines the origami (up to cylinder twists). In the second case, $w_1 = l_1$ and $w_2 = l_1+l_2+l_3$. So, given $(w_1,w_2)$, a choice of $1\leq l_3\leq w_2-w_1-1$ determines the origami (up to cylinder twists). In both cases, we have $w_1w_2$ choices of twists $(t_1,t_2)$ for the two cylinders.

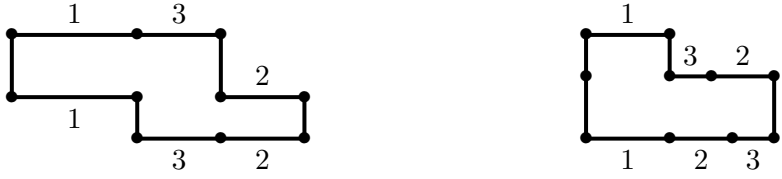
\begin{figure}[t]
    \centering
    \begin{tikzpicture}[scale = 0.55, line width = 1.5pt]
        \draw (0,0)node{$\bullet$} -- (0,1.5)node{$\bullet$} -- node[above]{1}(3,1.5)node{$\bullet$} -- node[above]{3}(5,1.5)node{$\bullet$} -- (5,0)node{$\bullet$} -- node[above]{2}(7,0)node{$\bullet$} -- (7,-1)node{$\bullet$} -- node[below]{2}(5,-1)node{$\bullet$} -- node[below]{3}(3,-1)node{$\bullet$} -- (3,0)node{$\bullet$} -- node[below]{1}(0,0);
    \end{tikzpicture}
    \hspace{3cm}
    \begin{tikzpicture}[scale = 0.55, line width = 1.5pt]
        \draw (0,0)node{$\bullet$} -- (0,1.5)node{$\bullet$} -- (0,2.5)node{$\bullet$} -- node[above]{1}(2,2.5)node{$\bullet$} -- (2,1.5)node{$\bullet$} -- node[above]{3}(3,1.5)node{$\bullet$} -- node[above]{2}(4.5,1.5)node{$\bullet$} -- (4.5,0)node{$\bullet$} -- node[below]{3}(3.5,0)node{$\bullet$} -- node[below]{2}(2,0)node{$\bullet$} -- node[below]{1}(0,0);
    \end{tikzpicture}
    \caption{Two-cylinder origamis in $\calH(1,1)$. Each cylinder can also be assigned a twist $1\leq t_i<w_i$ but we have not drawn them here.}
    \label{fig:H(1,1)-two-cylinder}
\end{figure}

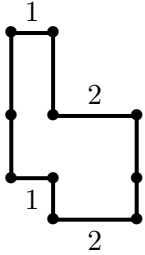
\begin{figure}[t]
    \centering
    \begin{tikzpicture}[scale = 0.55, line width = 1.5pt]
        \draw (0,0)node{$\bullet$} -- (0,1.5)node{$\bullet$} -- (0,3.5)node{$\bullet$} -- node[above]{1}(1,3.5)node{$\bullet$} -- (1,1.5)node{$\bullet$} -- node[above]{2}(3,1.5)node{$\bullet$} -- (3,0)node{$\bullet$} -- (3,-1)node{$\bullet$} -- node[below]{2}(1,-1)node{$\bullet$} -- (1,0)node{$\bullet$} -- node[below]{1}(0,0);
    \end{tikzpicture}
    \caption{A three-cylinder origami in $\calH(1,1)$. Each cylinder can also be assigned a twist $1\leq t_i<w_i$ but we have not drawn them here.}
    \label{fig:H(1,1)-three-cylinder}
\end{figure}

In both cases, for a cusp of width one, we are solving the equation
\[xw_1^2+yw_2^2 = N\]
for some $1\leq x,y\leq N$.

In the regime $1\leq \min\{x,y\}<\sqrt{N}$, the ellipses fit inside a box of side length $M = \sqrt{N}$. Hence, Bombieri--Pila again gives us that we have $O(N^{\frac{1}{4}+\epsilon})$ lattice points. We have $O(N^{\frac{3}{2}})$ ellipses and each solution $(w_1,w_2)$ determines a distinct origami (up to twists) for each of the choices of $l_3$. In both situations, we have $O(\sqrt{N})$ choices for $l_3$. Finally, we have $w_1w_2 = O(N)$ choices for the twists. So, we obtain a bound of $O(N^{\frac{13}{4}+\epsilon})$ origamis fixed by $T$.

If $\sqrt{N}\leq x,y\leq N$, then the ellipses fit inside a box of side length $M = N^{\frac{1}{4}}$. Hence, Bombieri--Pila again gives us that we have $O(N^{\frac{1}{8}+\epsilon})$ lattice points. We have $O(N^2)$ ellipses and each solution $(w_1,w_2)$ determines a distinct origami (up to twists) for each of the choices of $l_3$. This time, we have $O(N^{\frac{1}{4}})$ choices for $l_3$. Finally, we have $w_1w_2 = O(N^{\frac{1}{2}})$ choices for the twists. So, we obtain a bound of $O(N^{\frac{23}{8}+\epsilon})$ origamis fixed by $T$.

A three-cylinder origami in $\calH(1,1)$ has the form shown in Figure~\ref{fig:H(1,1)-three-cylinder}. Here, $w_1 = l_1$ and $w_3 = l_2$ and $w_2 = w_1+w_3$. Hence, the equation
\[xw_1^2+yw_2^2+zw_3^2 = N\]
becomes the conic
\[(x+y)w_1^2+2yw_1w_3+(y+z)w_3^2 = N\]
and we could apply Bombieri--Pila. However, it turns out that the bounds from Bombieri--Pila are not tight enough for our purposes and so we will instead appeal to a result of Browning--Gorodnik~\cite[Theorem 1.11]{BG} that for a quadric of rank $r$ gives $O(M^{r-2+\epsilon})$ lattice points inside a box of side length $M$. Here, we have $r = 2$. Notice that, up to twists, the origami is determined by $(w_1,w_3)$.

If $1\leq \min\{x,y,z\}\leq N^{\frac{5}{12}}$, then the conic fits into a box of side length $M = \sqrt{N}$. Hence, Browning--Gorodnik~\cite[Theorem 1.11]{BG} gives a bound of $O(M^{\epsilon}) = O(N^{\frac{1}{2}\epsilon})$ lattice points. We have $O(N^{2+\frac{5}{12}})$ conics and each solution $(w_1,w_3)$ determines a unique origami up to twists. We have $w_1w_2w_3= w_1(w_1+w_3)w_3 = O(N^{\frac{3}{2}})$ choices of twists. Hence, we have $O(N^{\frac{47}{12}+\frac{1}{2}\epsilon})$ origamis fixed by $T$.

If $N^{\frac{5}{12}}\leq x,y,z\leq N$, then the conic fits into a box of side length $M = N^{\frac{7}{24}}$. Hence, Browning--Gorodnik gives a bound of $O(N^{\frac{7}{24}\epsilon})$ lattice points. We have $O(N^{3})$ conics and each solution $(w_1,w_3)$ determines a unique origami up to twists. We have $w_1w_2w_3= w_1(w_1+w_3)w_3 = O(N^{\frac{7}{8}})$ choices of twists. Hence, we have $O(N^{\frac{31}{8}+\frac{7}{24}\epsilon})$ origamis fixed by $T$.

So, in total, we have $O(N^{\frac{47}{12}+\frac{1}{2}\epsilon})$ loops, and hence also $O(N^{\frac{47}{12}+\frac{1}{2}\epsilon})$ parabolic faces of length at most 4. This is $o(|\mathcal{G}_{N}|) = o(N^{4})$.
 
%%%%%%%%%%%%%%%%%%%%%%%%%%%%%%%%%%%%%%%%%

\subsection{Elliptic faces}

Here, we must count orbifold points on $W_{d^2}[n]$.

It can be checked (by considering the action of an automorphism on the zeros of a one-form) that there are no orbifold points of order two. Indeed, this would require an automorphism of the origami of order four. The square of this automorphism will fix the zeros of the one-form. However, the square of the automorphism must also be the hyperelliptic involution which is required to exchange the zeros, and this is a contradiction.

Therefore, we need only count the orbifold points of order three. In~\cite[Appendix A]{Muk}, Mukamel gives the following description (without proof) of genus two curves admitting an order six automorphism whose square is the hyperelliptic involution. We recall his definition of $\mathcal{M}_{2}(D_{12})$, where $D_{12}$ is the dihedral group of order 12:
\[\mathcal{M}_{2}(D_{12}) := \{(X,\rho):X\in\mathcal{M}_{2}\text{ and }\rho:D_{12}\xhookrightarrow{homo.}\Aut(X)\}/\!\sim,\]
where $(X_{1},\rho_{1})\sim(X_{2},\rho_{2})$ if there exists an isomorphism $f:X_{1}\to X_{2}$ satisfying $f\circ\rho_{1}\circ f^{-1} = \rho_{2}$.

\begin{theorem}\label{thm:M2(D12)}
    For a smooth curve $X\in\mathcal{M}_{2}$, the following are equivalent:
    \begin{itemize}
        \item $\exists\,\rho:D_{12}\xhookrightarrow{homo.}\Aut(X)$; i.e., $(X,\rho)\in\mathcal{M}_{2}(D_{12})$.
        \item The field of functions $\C(X)$ is isomorphic to the field
        \[\widetilde{K}_{a} = \C(x,z)\text{ with }z^2 = x^6-ax^3+1\]
        for some $a\in\C\setminus\{\pm2\}$.
        \item The Jacobian $\text{Jac}(X)$ is isomorphic to the principally polarised abelian variety $\widetilde{A}_{\tau} = \C^2/\widetilde{\Lambda}_{\tau}$, where
        \[\widetilde{\Lambda}_{\tau} = \Z\left\langle\begin{pmatrix}
            1 \\
            1/\sqrt{3}
        \end{pmatrix},\begin{pmatrix}
            \tau \\
            \sqrt{3}\tau
        \end{pmatrix},\begin{pmatrix}
            1 \\
            -1/\sqrt{3}
        \end{pmatrix},\begin{pmatrix}
        \tau \\
            -\sqrt{3}\tau
        \end{pmatrix}\right\rangle\]
        and the polarisation
        \[\left\langle\begin{pmatrix}
            a \\
            b
        \end{pmatrix},\begin{pmatrix}
            c \\
            d
        \end{pmatrix}\right\rangle = \frac{-{\rm{Im}}(a\bar{c}+b\bar{d})}{2\,{\rm{Im}}(\tau)}.\]
        \item The curve $X$ is isomorphic to a curve $\widetilde{X}_{\tau}$ obtained by gluing the hexagonal pinwheel $H_{\tau}$ in Figure~\ref{fig:pinwheel} to $-H_{\tau}$ for some $\tau$ in the domain
        \[\widetilde{U} := \{\tau\in\mathbb{H}: |{\rm{Re}}(\tau)|\leq 1/2, |\tau|^2\geq1/3,\text{ but } \tau\not\in\{\zeta_{12}/\sqrt{3},\zeta_{12}^5/\sqrt{3}\}\}.\]
    \end{itemize}
\end{theorem}

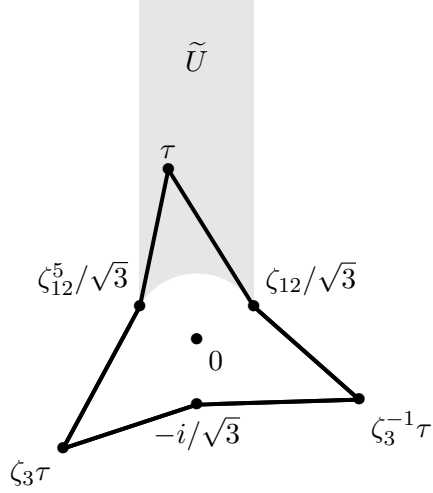
\begin{figure}[t]
    \centering
    \begin{tikzpicture}[scale = 1.5, line width = 1.5pt]
        \draw (1/2,0.288675134) -- (-1/4,1.5) -- (-1/2,0.288675134) -- (-1.1740381,-0.96650635) -- (0,-0.57735) -- (1.4240381,-0.5334936490) -- (1/2,0.288675134);
        \fill[gray!20]
        (-0.5,3)
        -- (-0.5,0.288675134)
        -- (0.5,0.288675134)
        -- (0.5,3)
        -- cycle;
        \draw [color = white, line width = 0pt] (-1.5,-1) -- (1.5,-1);
        \draw [fill = white, color = white, line width = 0pt] (0,0) circle (0.577350269);
        \node at (0,2.5) {$\widetilde{U}$};
        \node at (0,0) {$\bullet$};
        \node at (0,0) [below right]{0};
        \node at (1/2,0.288675134) {$\bullet$};
        \node at (1/2,0.288675134) [above right]{$\zeta_{12}/\sqrt{3}$};
        \node at (-1/2,0.288675134) {$\bullet$};
        \node at (-1/2,0.288675134) [above left]{$\zeta_{12}^{5}/\sqrt{3}$};
        \node at (0,-0.57735) {$\bullet$};
        \node at (0,-0.57735) [below]{$-i/\sqrt{3}$};
        \node at (-1/4,1.5) {$\bullet$};
        \node at (-1/4,1.5) [above]{$\tau$};
        \node at (-1.1740381,-0.96650635) {$\bullet$};
        \node at (-1.1740381,-0.96650635) [below left]{$\zeta_{3}\tau$};
        \node at (1.4240381,-0.5334936490) {$\bullet$};
        \node at (1.4240381,-0.5334936490) [below right]{$\zeta_{3}^{-1}\tau$};
        \draw (1/2,0.288675134) -- (-1/4,1.5) -- (-1/2,0.288675134) -- (-1.1740381,-0.96650635) -- (0,-0.57735) -- (1.4240381,-0.5334936490) -- (1/2,0.288675134);
    \end{tikzpicture}
    \caption{The hexagonal pinwheel $H_{\tau}$ for $\tau\in\widetilde{U}$ (the shaded domain). It has vertices $\{\zeta_{12}/\sqrt{3},\tau,\zeta_{12}^{5}/\sqrt{3},\zeta_3\tau,-i/\sqrt{3},\zeta_{3}^{-1}\tau\}$.}
    \label{fig:pinwheel}
\end{figure}

We will give the proof of this theorem following Mukamel's arguments for $\mathcal{M}_{2}(D_{8})$ that he studied to establish the counts of order two orbifold points that we used in Subsection~\ref{subsec:ell-H2}. The details of these proofs will allow us to prove the following result. Recall that $\tilde{h}(-D)$ is the reduced class number defined by
\[\tilde{h}(-D) = \frac{2h(-D)}{|\mathcal{O}_{-D}^{\times}|}.\]

\begin{theorem}\label{thm:e3}
    Let $e_{3}$ denote the number of order three orbifold points. If $n$ is even, then $e_{3}(W_{d^2}[n]) = 0$ for all $d\geq 2$. Otherwise, for odd $n$, if $d\equiv 0\!\!\mod 3$, we have
    \[e_{3}(W_{d^2}^{\epsilon}[n]) = \left\lbrace\begin{array}{cl}
        \frac{1}{2}(3\tilde{h}(-d^2/3) + \tilde{h}(-3d^2)), & \text{if $n = 3$ and $\epsilon = 0$} \\
        0, & \text{otherwise},
    \end{array}\right.\]
    if $d\equiv 1\!\!\mod 3$, we have
    \[e_{3}(W_{d^2}^{\epsilon}[n]) = \left\lbrace\begin{array}{cl}
        \frac{1}{2}\tilde{h}(-3d^2), & \text{if $n = 1$ and $\epsilon = 0$} \\
        0, & \text{otherwise},
    \end{array}\right.\]
    and, if $d\equiv 2\!\!\mod 3$, we have
    \[e_{3}(W_{d^2}^{\epsilon}[n]) = \left\lbrace\begin{array}{cl}
        \frac{1}{2}\tilde{h}(-3d^2), & \text{if $n = 3$ and $\epsilon = 0$} \\
        0, & \text{otherwise}.
    \end{array}\right.\]
\end{theorem}

In each case, we get a bound on the number of elliptic faces of $O(d^2)$ which is $o(|\mathcal{G}_{N}|) = o(N^{4}) = o(d^{4})$ and this completes the proof of Theorem~\ref{thm:submain-2}.

\subsubsection{Proof of Theorem~\ref{thm:M2(D12)}}

Here, we prove Theorem~\ref{thm:M2(D12)}. It follows the structure of~\cite[\S 2]{Muk}. We consider $D_{12}$ with the following presentation:
\[D_{12} = \langle r, Z\mid r^2 = Z^6 = (Zr)^2 = 1\rangle.\]

\begin{proposition}\label{prop:M2(D12)-1}
    Suppose that $X\in\mathcal{M}_{2}$ has an order six automorphism $\phi$. Then either:
    \begin{enumerate}
        \item $X$ is one of the algebraic curves $y^2 = x^6\pm 1$ with $\phi(x,y) = (\zeta_{6}x,y)$; or
        \item The restriction of $\phi$ to the Weierstrass points $X^{W}$ acts as a product of two 3-cycles, the eigenforms for $\phi$ have simple zeros, $\phi^3 = \eta$ the hyperelliptic involution of $X$, and there exists an injective homomorphism $\rho:D_{12}\hookrightarrow\Aut(X)$ with $\rho(Z) = \phi$.
    \end{enumerate}
\end{proposition}

\begin{proof}
    (1) The curves $y^2 = x^6\pm 1$ admit $\phi(x,y) = (\zeta_6,y)$. Notice though that $\phi^3\neq\eta$ where $\eta(x,y) = (x,-y)$ is the hyperelliptic involution.
    
    (2) Otherwise, by~\cite[Table 2]{Muk}, we have that $X$ is the curve
    \[y^2 = (x^3-t^3)(x^3-t^{-3}),\,\,\,\, t\in\C^{*}\]
    and $\phi(x,y) = (\zeta_{3}x,-y)$. Here, $\phi^3 = \eta$ and $\phi$ acts as a product of 3-cycles on the Weierstrass points $X^{W}$. The eigenforms of $\phi$ are
    \begin{itemize}
        \item $\frac{dx}{y}$ with two simple zeros at $x^{-1}(\infty)$; and
        \item $\frac{xdx}{y}$ with two simple zeros at $x^{-1}(0)$.
    \end{itemize}
    Moreover, we can define $\rho:D_{12}\to\Aut(X)$ by $\rho(Z) = \phi$ and $\rho(r)(x,y) = \left(\frac{1}{x},\frac{y}{x^3}\right)$. It can be checked that this is an injective homomorphism.
\end{proof}

\begin{remark}
    Recall that $y^2 = x^6 + 1$ and $y^2 = x^6-1$ are isomorphic over $\C$ and the former can be written in the form of (2) with $t^6 = -1$.
\end{remark}

Now define $\widetilde{K}_{a} = \C(x,z)\text{ with }z^2 = x^6-ax^3+1$ for $a \in\C\setminus\{\pm 2\}$. Let $Y_{a}$ be the curve with $\C(Y_{a}) = \widetilde{K}_{a}$. Let $\rho_{a}:D_{12}\hookrightarrow\Aut(Y_{a})$ be induced by the actions
\[Z(x,z) = (\zeta_{3}x,-z)\]
and
\[r(x,z) = \left(\frac{1}{x},\frac{z}{x^3}\right)\]
on $\widetilde{K}_{a}$.

\begin{proposition}\label{prop:M2(D12)-2}
    The map $a\mapsto(Y_{a},\rho_{a})$ defines a surjective holomorphic map $f:\C\setminus\{\pm 2\}\to\mathcal{M}_{2}(D_{12})$. In particular, $\dim_{\C}\mathcal{M}_{2}(D_{12}) = 1$ and $\mathcal{M}_{2}(D_{12})$ has one irreducible component.
\end{proposition}

\begin{proof}
    Fix $(X,\rho)\in\mathcal{M}_{2}(D_{12})$. By Proposition~\ref{prop:M2(D12)-1}, we know that there exists $t\in\C^{*}$ such that $X$ is the curve
    \[y^2 = (x^3+t^3)(x^3-t^{-3}) = x^6 - (t^3+t^{-3})x^3+1.\]
    Since $X$ is non-singular, $t^3+t^{-3}\neq\pm 2$ and so $(X,\rho)$ is in the image of $f$ and is isomorphic to $(Y_{a},\rho_{a})$ with $a = t^3+t^{-3}$. It can be checked that $Y_{a}$ depends holomorphically on $a$.
\end{proof}

We now consider the quotient of a curve $X$ in $\mathcal{M}_{2}(D_{12})$ by the action of $\rho(r)$.

\begin{proposition}\label{prop:M2(D12)-3}
    For $(X,\rho)\in\mathcal{M}_{2}(D_{12})$ define $E_{\rho} = X/\rho(r)$. Then $E_{\rho}$ is an elliptic curve and exhibits a canonical choice of subgroup of order 3 in $E_{\rho}[3]$ denoted $C_{\rho}$.
\end{proposition}

\begin{proof}
    By Proposition~\ref{prop:M2(D12)-2}, we have $(X,\rho)\cong(Y_{a},\rho_{a})$ for some $a\in\C\setminus\{\pm 2\}$. Recall that
    \[\C(Y_{a}) = \widetilde{K}_{a} = \C(x,z): z^2 = x^6-ax^3+1\]
    and $\rho_{a}(r)(x,z) = \left(\frac{1}{x},\frac{z}{x^3}\right)$. We must determine $\C(x,z)^{r}:=\widetilde{K}_{a}^{\rho_{a}(r)}$.

    We observe that $u = x + \frac{1}{x}$ is invariant. So
    \[\C(u)\subseteq\C(x,z)^{r}\subseteq\C(x,z).\]
    
    Since $\rho_{a}(r)^2 = {\rm id}$, $\C(x,z)/\C(x,z)^{r}$ is a degree two extension.

    Now, $\C(x)/\C(u)$ is also a degree two extension since $x^2 = ux - 1$. Hence, $[\C(x,z):\C(u)] = 4$ and so $[\C(x,z)^{r}:\C(u)] = 2$.

    Define $v = z\cdot\frac{1}{x^3+1}\cdot(x+\frac{1}{x}+1)$. It can then be checked that $\rho_{a}(r)^{*}(v) = v$ and 
    \[v^2 = \frac{u^3-3u-a}{u+2}.\]
    Hence, $v\in\C(x,z)^{r}$ and lies in a quadratic extension over $\C(u)$. Therefore,
    \[\C(x,z)^{r} \cong \C(u,v): v^2 = \frac{u^3-3u-a}{u+2}.\]

    Now, the cubic $(u+2)v^2 - u^3+3u+a = 0$ corresponds to the Weierstrass form
    \[y^2 = w^3 + (6a-15)w + a^2-14a+22.\]

    Hence, $E_{\rho}$ is an elliptic curve with $j$-invariant
    \[\frac{6912(2a-5)^3}{(a+2)^3(a-2)}.\]

    The third division polynomial is
    \begin{align*}\psi_{3}(w) &= 3w^4+(36a-90)w^2+(12a^2-168a+264)w-36a^2+180a-225 \\
    &= 3(w-3)(w^3+3w^2+(12a-21)w+4(a-5/2)^2).
    \end{align*}
    Hence, $w = 3$ gives rise to a 3-torsion point on $E_{\rho}$. At $w = 3$, $y^2 = (a+2)^2$ and we have the canonical subgroup $C_{\rho} = \langle(3,a+2)\rangle\leq E_{\rho}[3].$
\end{proof}

The modular curve $Y_{0}(3) = \mathbb{H}/\Gamma_{0}(3)$ parameterises elliptic curves with a choice of order three subgroup of the 3-torison points. $Y_{0}(3)$ has two cusps and one elliptic point of order 3. A fundamental domain for $\Gamma_{0}(3)$ is shown in Figure~\ref{fig:Y_0(3)}.

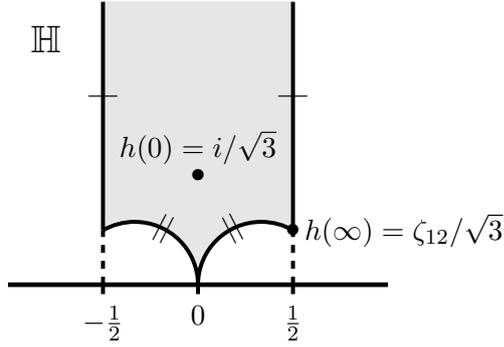
\begin{figure}[t]
    \centering
    \begin{tikzpicture}[scale = 2.5, line width = 1.5pt]
        \fill[gray!20]
          (-1/2,0)
          -- (-1/2,1.5)
          -- (1/2,1.5)
          -- (1/2,0)
          -- cycle;
        \draw [color = white, line width = 0pt] (-2.5,-0.25) -- (2.5,-0.25);
        \draw [fill = white, color = white, line width = 0pt] (0,0) arc (0:180:1/3) -- (0,0);
        \draw [fill = white, color = white, line width = 0pt] (2/3,0) arc (0:180:1/3) -- (2/3,0);
        \node [font = \Large] at (-0.8,1.3) {$\mathbb{H}$};
        \draw [line width = 1.75pt] (-1,0) -- (1,0);
        \draw [dashed] (-1/2,0) -- (-1/2,0.288675134);
        \draw (-1/2,0.288675134) -- (-1/2,1.5);
        \draw [dashed] (1/2,0) -- (1/2,0.288675134);
        \draw (1/2,0.288675134) -- (1/2,1.5);
        \draw (0,0) arc (180:60:1/3);
        \draw (0,0) arc (0:120:1/3);
        \node at (0,0.57735) [above]{$h(0) = i/\sqrt{3}$};
        \node at (0,0.57735) {$\bullet$};
        \node at (1/2,0.288675134) {$\bullet$};
        \node at (1/2,0.288675134) [right]{$h(\infty) = \zeta_{12}/\sqrt{3}$};
        \node at (1/2,1) [rotate = 90]{$|$};
        \node at (-1/2,1) [rotate = 90]{$|$};
        \node at (1/3-0.15,0.297676184991529) [rotate = 27]{$||$};
        \node at (-1/3+0.15,0.297676184991529) [rotate = -27]{$||$};
        \draw (-1/2,0) -- (-1/2,-0.05)node[below]{$-\frac{1}{2}$};
        \draw (0,0) -- (0,-0.05)node[below]{$0$};
        \draw (1/2,0) -- (1/2,-0.05)node[below]{$\frac{1}{2}$};
    \end{tikzpicture}
    \caption{A fundamental domain for $\Gamma_{0}(3)$. The point $h(\infty)$ is the elliptic point. The point $h(0)$ is the fixed point of the Fricke involution $w_{3}$ on $Y_{0}(3)$.}
    \label{fig:Y_0(3)}
\end{figure}

\begin{proposition}\label{prop:M2(D12)-4}
    Let $g:\mathcal{M}_{2}(D_{12})\to Y_{0}(3)$ be defined by $g(X,\rho) = (E_{\rho},C_{\rho})$. Then the composition $g\circ f:\C\setminus\{\pm 2\}\to Y_{0}(3)$ extends to a biholomorphism $h:\mathbb{P}^{1}\setminus\{\pm 2\}\to Y_{0}(3)$ with $h(\infty)$ equal to the elliptic point of $Y_{0}(3)$.
\end{proposition}

\begin{proof}
    Recall that $E_{\rho}$ is isomorphic to the elliptic curve
    \[y^2 = w^3 + (6a-15)w + a^2-14a+22.\]
    Let $a = \frac{1}{t^3}$ and define $Y = t^3y$ and $W = t^2w$. Then we have
    \[Y^2 = W^3 + (6t-15t^4)W + 22t^6-14t^3+1.\]
    The 3-torsion point $(3,a+2)$ maps to $(3t^2,1+2t^3)$. As $t\to0$, we get the curve $Y^2 = W^3+1$ with order three subgroup $\langle(0,1)\rangle$. This is the elliptic point on $Y_{0}(3)$.

    The coarse spaces $\mathbb{P}^{1}\setminus\{\pm 2\}$ and $Y_{0}(3)$ are biholomorphic so we need only check that the map $h$ has degree one.

    Let $j:Y_{0}(3)\to\C$ be the map $j(E,C) = j(E)$. We calculated earlier that
    \[j\circ h(a) = \frac{6912(2a-5)^3}{(a+2)^3(a-2)}\]
    has degree 4. Since $[\SL(2,\Z):\Gamma_{0}(3)] = 4$, $j$ has degree 4. Hence, $h$ has degree one as required.
\end{proof}

We then obtain the following that is proved exactly as~\cite[Corollary 2.7]{Muk}.

\begin{corollary}
    The map $f:\C\setminus\{\pm 2\}\to\mathcal{M}_{2}(D_{12})$ is a biholomorphism and $g:\mathcal{M}_{2}(D_{12})\to Y_{0}(3)$ is a biholomrphism onto its image which is the complement of the elliptic point $(\zeta_{12}/\sqrt{3})\cdot\Gamma_{0}(3)$.
\end{corollary}

We now introduce some involutions on our spaces of interest.

On $\C\setminus\{\pm 2\}$, we define $\mu(a) = -a$.

Now, any injective group homomorphism $h: D_{12} \to D_{12}$ gives rise to a holomorphic map $\mathcal{M}_{2}(D_{12})\to\mathcal{M}_{2}(D_{12})$ by $(X,\rho)\mapsto(X,\rho\circ h)$. In particular, $\Aut(D_{12})$ acts on $\mathcal{M}_{2}(D_{12})$. By the definition of the equivalence relation on $\mathcal{M}_{2}(D_{12})$, the inner automorphisms of $D_{12}$ fix every point of $\mathcal{M}_{2}(D_{12})$ and so we get an action of $\Out(D_{12})\cong \Z/2\Z$ on $\mathcal{M}_{2}(D_{12})$. The non-trivial outer automorphism of $D_{12}$ is represented by $\sigma(Z) = Z$ and $\sigma(r) = Z^3 r$. Therefore, $\sigma$ acts as an involution on $\mathcal{M}_{2}(D_{12})$.

Finally, on $Y_{0}(3)$, we have the Fricke involution $w_{3}$ that, in the fundamental domain of Figure~\ref{fig:Y_0(3)}, acts by
\[w_{3}(\tau) = \frac{-1}{3\tau}.\]
It acts on pairs by $(E,C)\mapsto (E/C,E[3]/C)$.

These involutions are related by the following proposition.

\begin{proposition}
    The following diagram commutes:
    \begin{center}
    \begin{tikzcd}
        \C\setminus\{\pm 2\} \arrow[r, "f"] \arrow[d, "\mu"'] 
        & \mathcal{M}_{2}(D_{12}) \arrow[r, "g"] \arrow[d, "\sigma"] 
        & Y_{0}(3) \arrow[d, "w_{3}"] \\
        \C\setminus\{\pm 2\} \arrow[r, "f"'] 
        &  \mathcal{M}_{2}(D_{12}) \arrow[r, "g"'] 
        & Y_{0}(3)
    \end{tikzcd}
    \end{center}
\end{proposition}

\begin{proof}
    Let $a\in\C\setminus\{\pm 2\}$ and $f(a) = (X_{a},\rho)$ be such that $X_{a}$ is the curve
    \[y^2 = x^6-ax^3+1\]
    with $\rho:D_{12}\to\Aut(X_{a})$ defined by
    \[\rho(Z)(x,y) = (\zeta_{3}x,y)\,\,\,\,\,\,\,\text{and}\,\,\,\,\,\,\,\rho(r)(x,y) = \left(\frac{1}{x},\frac{y}{x^3}\right).\]
    Now, $f\circ\mu(a) = (X_{-a},\rho)$ with $X_{-a}$ given by
    \[y^2 = x^6+ax^3+1\]
    and $\rho$ as before.
    Let $s(x,y) = (-x,y)$. Then $s$ is an isomorphism from $X_{-a}$ to $X_{a}$ and we observe that
    \[s\circ\rho(Z)\circ s^{-1} (x,y) = (\zeta_{3}x,-y) = \rho(Z)(x,y)\]
    and
    \[s\circ\rho(r)\circ s^{-1}(x,y) = \left(\frac{1}{x},\frac{-y}{x^3}\right) = \rho(Z^3r)(x,y) = \rho(\sigma(r))(x,y).\]
    Hence, we have $\sigma\circ f= f\circ\mu$.

    Now, two points $(E_1,C_1), (E_2,C_2)\in Y_{0}(3)$ are sent to one another by $w_{3}$ if and only if $\Phi_{3}(j(E_1),j(E_2)) = 0$, where $\Phi_{3}$ is the modular polynomial of $\Gamma_{0}(3)$. We have worked out the $j$-invariants of our elliptic curves above and it can be checked that $\Phi_{3}(j\circ g\circ f(a),j\circ g\circ f(a')) = 0$ if and only if $a' = -a$ (we have removed the other two solutions $a'=\pm 2$ from our domain). Hence, we see that $w_{3}\circ g = g\circ \sigma$, as claimed.
\end{proof}

From this, we obtain the following.

\begin{corollary}
    Fix $(X,\rho)\in\mathcal{M}_{2}(D_{12})$. The following are equivalent:
    \begin{itemize}
        \item $\exists\,\phi\in\Aut(X)$ satisfying $\phi\neq\rho(Z)$ but $\phi$ has order six;
        \item $\sigma(X,\rho)\sim(X,\rho)$;
        \item $(X,\rho) = f(0)$;
        \item $g(X,\rho) = \frac{i}{\sqrt{3}}\cdot\Gamma_{0}(3).$
    \end{itemize}
\end{corollary}

The coarse space $\mathcal{M}_{2}(D_{12})/\sigma\cong(\C\setminus\{\pm 2\})/\mu$ is isomorphic to $\C^{*}$ and has two cusps. Given
\[y^2 = x^6-ax^3+1,\]
to understand the behaviour as $a\to\infty$, make the coordinate change $y = \frac{v}{t^3}, x = \frac{u}{t}, a = \frac{1}{t^3}$ to get
\[v^2 = u^6 - u^3 + t^6 \to v^2 = u^6-u^3\]
which normalises to
\[Y^2 = X^6+1,\]
a genus two curve. This is not surprising since we have already argued that $g(X_{a},\rho_{a})$ converges to the elliptic point on $Y_{0}(3)$ as $a\to \infty$.

As we move towards the other cusp --- that is, as $a\to\{\pm2\}\cdot\mu$ --- we have
\[y^2= x^6-ax^3+1\to y^2 = (x^3\pm 1)^2\]
which normalises to
\[(y = x^3\pm 1)\sqcup (y = -x^3\mp 1)\]
a curve of geometric genus zero. Moreover, we see that the quotient elliptic curves $E_{\rho_{a}}:y^2 = w^3+(6a-15)w+a^2-14a+22$ diverge in $\mathcal{M}_{1}$.

We will use this difference in the limiting behaviour at the cusps below.

Recall that we can take
\[t_{3}(\tau):=-27\left(\frac{\eta(3\tau)}{\eta(\tau)}\right)^{12}\]
as a Hauptmodul for $\Gamma_{0}(3)$, where $\eta$ is the Dedekind eta function.

\begin{proposition}
    If $(X,\rho)\in\mathcal{M}_{2}(D_{12})$ and $\tau\in\mathbb{H}$ is such that $g(X,\rho) = (E_{\tau},\langle\frac{1}{3}+(\Z\oplus\tau\Z)\rangle)$, then $\C(X)\cong\widetilde{K}_{a}$ where
    \[a(\tau):= -2\cdot\frac{t_{3}(\tau)+1}{t_{3}(\tau)-1}.\]
\end{proposition}

\begin{proof}
    Up to precomposition by $w_{3}$, $a$ is the unique meromorphic function such that
    \begin{enumerate}
        \item $a$ covers an isomorphism $\bar{a}:Y_{0}(3)\to\mathbb{P}^{1}\setminus\{\pm 2\}$; and
        \item $a(\zeta_{12}/\sqrt{3}) = \infty.$
    \end{enumerate}
    The biholomorphic extension $h:\mathbb{P}^{1}\setminus\{\pm 2\}\to Y_{0}(3)$ has $h(\infty) = \zeta_{12}/\sqrt{3}$ and so $h^{-1}$ satisfies (1) and (2). Therefore, $f(\bar{a}(E_{\tau},\langle\frac{1}{3}+(\Z\oplus\tau\Z)\rangle))$ is equal to either $f(a(\tau))$ or $f(a(w_3(\tau))) = \sigma(f(a(\tau)))$. So, in either case, $\C(X)\cong\widetilde{K}_{a(\tau)}$.
\end{proof}

Note also that $a(i\infty) = -2$, $a(0) = 2$, $a(\frac{i}{\sqrt{3}}) = 0$, and $a(\frac{1}{2}+\frac{\sqrt{3}}{2}i) = \frac{5}{2}$.

We now attempt to understand the Jacobians of the curves $X\in\mathcal{M}_{2}(D_{12})$. The following paragraph and proposition are lifted directly from~\cite[\S 2]{Muk}.

Let $X\in\mathcal{M}_{2}$ with $\phi\in\Aut(X)$ such that $\phi^{2} = {\rm id}$ and $\phi\neq \eta$ the hyperelliptic involution. Let $E = X/\phi$ and $F = X/\eta\phi$. Let $\psi:X\to E\times F$ be the obvious map and set $\Gamma^{W} = \psi(X^{W})$ be the image of the Weierstrass points of $X$. Mukamel proves the following.

\begin{proposition}[{\cite[Proposition 2.12]{Muk}}]
    Let $X$ be as just discussed. Then
    \[\Jac(X)\cong (E\times F)/\Gamma^{W}\]
    and the principal polarisation on $\Jac(X)$ pullsback to twice the product polarisation on $E\times F$.
\end{proposition}

We can use this to prove the following.

\begin{proposition}
    For $\tau\in\mathbb{H}$, define $\widetilde{A}_{\tau} = \C^2/\widetilde{\Lambda}_{\tau}$, where
        \[\widetilde{\Lambda}_{\tau} = \Z\left\langle\begin{pmatrix}
            1 \\
            1/\sqrt{3}
        \end{pmatrix},\begin{pmatrix}
            \tau \\
            \sqrt{3}\tau
        \end{pmatrix},\begin{pmatrix}
            1 \\
            -1/\sqrt{3}
        \end{pmatrix},\begin{pmatrix}
        \tau \\
            -\sqrt{3}\tau
        \end{pmatrix}\right\rangle\]
        with
        \[\left\langle\begin{pmatrix}
            a \\
            b
        \end{pmatrix},\begin{pmatrix}
            c \\
            d
        \end{pmatrix}\right\rangle = \frac{-{\rm{Im}}(a\bar{c}+b\bar{d})}{2\,{\rm{Im}}(\tau)}.\]
        Let $(X,\rho)\in\mathcal{M}_{2}(D_{12})$ and suppose that $g(X,\rho) = (E_{\tau},\langle\frac{1}{3}+(\Z\oplus\tau\Z)\rangle)$. Then
        \[\Jac(X)\cong\widetilde{A}_{\tau}.\]
\end{proposition}

\begin{proof}
    Set $\phi = \rho(r)$, $E = X/\phi$ and $F = X/\eta\phi$. By definition, $E\cong E_{\tau}$. Now,
    \[\rho(Z)\circ\sigma(\phi)\circ\rho(Z^{-1}) = \eta\phi\]
    so that $\rho(Z)$ induces an isomorphism between $F$ and $g(\sigma(X,\rho))\cong E_{w_{3}(\tau)} = E_{{-1}/{3\tau}}$. Therefore, $E\times F\cong E_{\tau}\times E_{{-1}/{3\tau}}$. After multiplying the standard lattice for $E_{{-1}/{3\tau}}$ by $\sqrt{3}\tau$, the standard product polarisation is
    \[\left\langle\begin{pmatrix}
            a \\
            b
        \end{pmatrix},\begin{pmatrix}
            c \\
            d
        \end{pmatrix}\right\rangle = \frac{-{\rm{Im}}(a\bar{c}+b\bar{d})}{{\rm{Im}}(\tau)}.\]
    Finally, since $\widetilde{\Lambda}_{\tau}$ contains the vectors $\begin{pmatrix}
            2 \\
            0
        \end{pmatrix},\begin{pmatrix}
            2\tau \\
            0
        \end{pmatrix},\begin{pmatrix}
            0 \\
            -2/\sqrt{3}
        \end{pmatrix}$, and $\begin{pmatrix}
        0 \\
        2\sqrt{3}\tau
        \end{pmatrix}$
    we get that $\widetilde{A}_{\tau}$ is isomorphic to $(E_{\tau}\times E_{{-1}/{3\tau}})/\Gamma^{W}$ with half the standard product polarisation.
\end{proof}

Now, let
\[\widetilde{U} := \{\tau\in\mathbb{H}: |{\rm{Re}}(\tau)|\leq 1/2, |\tau|^2\geq1/3,\text{ but } \tau\not\in\{\zeta_{12}/\sqrt{3},\zeta_{12}^5/\sqrt{3}\}\}.\]
This is a fundamental domain for $\Gamma_{0}(3)\cup\{w_{3}\}$. Recall that, for $\tau\in\widetilde{U}$, $\widetilde{X}_{\tau}$ is the curve obtained by gluing the hexagonal pinwheel $H_{\tau}$ in Figure~\ref{fig:pinwheel} to $-H_{\tau}$. Define $Z_{\tau}$ to be the automorphism of $\widetilde{X}_{\tau}$ obtained by rotation by $2\pi/3$ followed by negation.

\begin{proposition}
    Fix $\tau\in\widetilde{U}$. There is an injective homomorphism $\rho_{\tau}:D_{12}\hookrightarrow\Aut(\widetilde{X}_{\tau})$ with $\rho_{\tau}(Z) = Z_{\tau}$ and $g(\widetilde{X}_{\tau},\rho_{\tau}) = (E_{\tau},\langle\frac{1}{3}+(\Z\oplus\tau\Z)\rangle)$.
\end{proposition}

\begin{proof}
    By Proposition~\ref{prop:M2(D12)-1}, since $\widetilde{X}_{\tau}$ admits an order six automorphism cubing to the hyperelliptic involution, we known that such a $\rho_{\tau}$ exists.
    
    Since $\widetilde{U}$ is simply connected, we can choose $\rho_{\tau}$ so that $\tau\mapsto(\widetilde{X}_{\tau},\rho_{\tau})$ gives a holomorphic map $p:\widetilde{U}\to\mathcal{M}_{2}(D_{12})$.

    It can be checked that $\widetilde{X}_{\tau}$ and $\widetilde{X}_{\tau+1}$ are isomorphic by cut-and-paste, and simiarly for $\widetilde{X}_{\tau}$ and $\widetilde{X}_{w_{3}(\tau)}$. So $p$ covers a holomorphic map $\bar{p}:\widetilde{U}/{\sim}\to\mathcal{M}_{2}(D_{12})/\sigma$, where $\tau\sim\tau+1$ and $\tau\sim w_{3}(\tau)$.

    The point $\tau = \frac{i}{\sqrt{3}}\in\widetilde{U}$ is the unique point such that $\widetilde{X}_{\tau}$ admits an order six automorphism distinct from $Z_{\tau}$ (since it is glue from two regular hexagons). So $p^{-1}(\text{Fix}(\sigma)) = \frac{i}{\sqrt{3}}$. Hence, $\bar{p}$ has degree one and is therefore a biholomorphism. There are exactly two such biholomorphisms $\widetilde{U}/{\sim}\to\mathcal{M}_{2}(D_{12})/\sigma$ with $\bar{p}(\frac{i}{\sqrt{3}}) = \text{Fix}(\sigma)/\sigma$. They are distinguished by the stable limit of $p(\tau)$ as $\tau\to i\infty$. Here, we have that the geometric genus of $\widetilde{X}_{\tau}$ goes to zero.

    We also have the map $p_{1}:\widetilde{U}\to\mathcal{M}_{2}(D_{12})$ given by $p_{1}(\tau) = g^{-1}(E_{\tau},\langle\frac{1}{3}+(\Z\oplus\tau\Z)\rangle)$. This map intertwines $w_{3}$ and $\sigma$, and we also have
    $$\left(E_{\tau},\left\langle\frac{1}{3}+(\Z\oplus\tau\Z)\right\rangle\right)\cong\left(E_{\tau+1},\left\langle\frac{1}{3}+(\Z\oplus(\tau+1)\Z)\right\rangle\right).$$
    So, $p_{1}$ covers $\bar{p}_{1}:\widetilde{U}/{\sim}\to\mathcal{M}_{2}(D_{12})/\sigma$. We saw earlier that $g^{-1}(E_{\tau},\langle\frac{1}{3}+(\Z\oplus\tau\Z)\rangle)$ has geometric genus zero as $\tau\to i\infty$ (equivalently, as $a(\tau)\to\{\pm 2\}\cdot\mu$). Moreover, $p_{1}(\frac{i}{\sqrt{3}}) = \text{Fix}(\sigma)/\sigma$.

    Therefore, $\bar{p} = \bar{p}_{1}$, and we are done.
\end{proof}

The following proposition then easily follows and completes the proof of Theorem~\ref{thm:M2(D12)}.

\begin{proposition}
    Fix $\tau\in\widetilde{U}$. Then there exists an injective homomorphism $\rho:D_{12}\hookrightarrow\Aut(\widetilde{X}_{\tau})$ with $\rho(Z) = Z_{\tau}$ and with $\widetilde{X}_{\tau}$ satisfying:
    \[\C(\widetilde{X}_{\tau}) = \widetilde{K}_{a(\tau)}\,\,\,\,\,\text{and}\,\,\,\,\,\Jac(\widetilde{X}_{\tau}) = \widetilde{A}_{\tau}.\]
    For any $(X,\rho)\in\mathcal{M}_{2}(D_{12})$, $\exists\,\tau\in\widetilde{U}$ and an isomorphism $\varphi:X\to\widetilde{X}_{\tau}$ such that $\varphi\circ\rho(Z)\circ\varphi^{-1} = Z_{\tau}$.
\end{proposition}

\subsubsection{Determining $e_{3}$}

We consider the Hilbert modular surface $X_{D}$ given by
\[\mathbb{H}\times\mathbb{H}/\PSL(\mathcal{O}_{D}^{\vee}\oplus\mathcal{O}_{D}).\]
As a quotient stack, this is the moduli space parameterising principally polarised abelian surfaces admitting real multiplication by $\mathcal{O}_{D}$. We use $\PSL(\mathcal{O}_{D}^{\vee}\oplus\mathcal{O}_{D})$ in order to kill the global order two automorphism $-I$. We let $[\tau]$ be the class of the point $\tau\in\mathbb{H}\times\mathbb{H}$, and let $B_{\tau}$ denote the corresponding abelian surface with $\iota_{\tau}:\mathcal{O_{D}}\to\End(B_{\tau})$.

For $\tau\in\mathbb{H}\times\mathbb{H}$, we define the orbifold order of $[\tau]$ in $X_D$ to be the order of the group ${\rm Stab}(\tau)\leq\PSL(\mathcal{O}_{D}^{\vee}\oplus\mathcal{O}_{D})$. We will call $[\tau]\in X_D$ an orbifold point if the orbifold order of $[\tau]$ is greater than one.

We first consider complex multiplication on the variety $\widetilde{A}_{\tau}$. This is analogous to~\cite[Proposition 3.2]{Muk}. Recall that we say that $\widetilde{A}_{\tau}$ admits complex multiplication if there exists a degree two extension of $\mathcal{O}_{D}$ that embeds into $\End(\widetilde{A}_{\tau})$.

\begin{proposition}\label{prop:CM}
    Fix $\tau\in\mathbb{H}$. The abelian variety $\widetilde{A}_{\tau}$ has complex multiplication if and only if $\tau$ is imaginary quadratic.
\end{proposition}

\begin{proof}
    Suppose that $\tau$ is imaginary quadratic, then $\widetilde{\Lambda}_{\tau}\otimes\Q$ is stabilised by $\frac{1}{2}\begin{psmallmatrix}
        1 & -\sqrt{3} \\
        \sqrt{3} & 1
    \end{psmallmatrix}$ and $\begin{psmallmatrix}
        \tau & 0 \\
        0 & \tau
    \end{psmallmatrix}$. This gives a Hermitian-adjoint embedding $\iota:\Q(\zeta_{6},\tau)\to\End(\widetilde{A}_{\tau})\otimes\Q$.
    Restricting $\iota$ to $\mathcal{O}:=\iota^{-1}(\End(\widetilde{A}_{\tau}))$ gives complex multiplication by $\mathcal{O}$ on $\widetilde{A}_{\tau}$.

    If $\tau$ is not imaginary quadratic then both $E_{\tau}$ and $E_{w_{3}(\tau)}$ do not have complex multiplication. Hence, $\End(E_{\tau}\times E_{w_{3}(\tau)})\otimes\Q\cong M_{2}(\Q)$. We have a degree four surjection from $E_{\tau}\times E_{w_{3}(\tau)}$ to $\widetilde{A}_{\tau}$ which gives an isomorphism between their rational endomorphism rings. So any commutative ring in $\End(\widetilde{A}_{\tau})\otimes\Q$ has rank at most two over $\Q$ and we cannot find the image of an order giving complex multiplication.
\end{proof}

We then have the following propositions from~\cite{Muk}.

\begin{proposition}[{\cite[Proposition 3.3]{Muk}}]\label{prop:Muk-3.3}
    Fix $\tau\in\mathbb{H}\times\mathbb{H}$ and an integer $n>2$. The following are equivalent:
    \begin{enumerate}
        \item The point $\tau$ is fixed by an $A\in\SL(\mathcal{O}_{D}^{\vee}\oplus\mathcal{O}_{D})$ of order $n$.
        \item There is an automorphism $\phi\in\Aut(B_{\tau})$ of order $n$ that commutes with $\iota_{\tau}(\mathcal{O}_{D})$.
        \item The homomorphism $\iota_{\tau}:\mathcal{O}_{D}\to\End(B_{\tau})$ extends to complex multiplication by an order containing $\mathcal{O}_{D}[\zeta_{n}]$ where $\zeta_{n}$ is a primitive $n^{\text{th}}$ root of unity.
    \end{enumerate}
\end{proposition}

\begin{proposition}[{\cite[Proposition 3.4]{Muk}}]
    The Jacobians of surfaces in $\mathcal{M}_{2}(D_{8})$ and $\mathcal{M}_{2}(D_{12})$ with complex multiplication are labeled by orbifold points in $\cup_{D}X_D$.
\end{proposition}

\begin{proposition}[{\cite[Proposition 3.5]{Muk}}]
    Fix an orbifold point $[\tau]\in X_{D}$. At least one of the following holds:
    \begin{itemize}
        \item $B_{\tau}$ is a product of elliptic curves;
        \item $[\tau]$ is a point of orbifold order five on $X_{5}$;
        \item $B_{\tau}$ is the Jacobian of a surface in $\mathcal{M}_{2}(D_8)$ with complex multiplication; or
        \item $B_{\tau}$ is the Jacobian of a surface in $\mathcal{M}_{2}(D_{12})$ with complex multiplication.
    \end{itemize}
\end{proposition}

The notation $W_{d^{2}}[n]$ comes from the fact that, since each $X\in W_{d^{2}}[n]$ admits a primitive degree $d$ cover of an elliptic curve, the Jacobian of every curve in $W_{d^2}[n]$ admits real multiplication by $\mathcal{O}_{d^{2}}$. As discussed in~\cite[\S 4]{Muk}, given a curve $X\in\mathcal{M}_{2}$ whose Jacobian admits real multiplication by $\mathcal{O}_{D}$ and an $\mathcal{O}_{D}$-eigenform $[\omega]\in\mathbb{P}\Omega(X)$, we can send $(X,[\omega])$ to $(\Jac(X),\iota^{[\omega]})\in X_{D}$, where $\iota^{[\omega]}(x)\omega = \sigma_{1}(x)\omega$, for $\sigma_{1},\sigma_{2}:\mathcal{O}_{D}\otimes\Q\to\R$ the two places of $\mathcal{O}_{D}\otimes\Q$. For $D = d^{2}$, $\sigma_{1}$ is just the projection onto the first coordinate of $\mathcal{O}_{D}\otimes\Q = \Q\times\Q$.

Recall the following definitions for $X\in\mathcal{M}_{2}$ and $\omega\in\Omega(X)$:
\begin{itemize}
    \item $\Aut(X,\{\pm\omega\}) = \{\phi\in\Aut(X):\phi^{*}\omega = \pm \omega\}$;
    \item $\SO(X,\omega) = \{\phi\in\Aut(X):\phi^{*}\omega \in\C^{*} \omega\}$; and
    \item $\PSO(X,\omega) = \SO(X,\omega)/\Aut(X,\{\pm\omega\})$.
\end{itemize}
For a point $(X,[\omega]) \in W_{d^{2}}[n]$, we define the orbifold order of $(X,[\omega])$ to be the order of the group $\PSO(X,\omega)$ and we will call $(X,[\omega])$ an orbifold point if its orbifold order is greater than one. By Proposition~\ref{prop:Muk-3.3}, it
is straightforward to check that $(X,[\omega])\in W_{d^{2}}[n]$ is an orbifold point if and only if the pair $(\Jac(X),\iota^{[\omega]})$ is an orbifold point on $X_{d^2}$.

\begin{proposition}
    Given $\tau\in\widetilde{U}$, the group $\PSO(\widetilde{X}_{\tau},\omega_{\tau})$ is cyclic of order 3.
\end{proposition}

\begin{proof}
    Let $\tau\neq\frac{i}{\sqrt{3}}$. It can be checked that $\Aut(\widetilde{X}_{\tau}) \cong \langle\rho(Z),\rho(r)\rangle\cong D_{12}$, $\SO(\widetilde{X}_{\tau},\omega_\tau)\cong\langle\rho(Z)\rangle\cong\Z/6\Z$, and $\Aut(\widetilde{X}_{\tau},\{\pm\omega_{\tau}\})\cong \langle\rho(Z^{3})\rangle\cong \Z/2\Z$. Hence, $\PSO(\widetilde{X}_{\tau},\omega_\tau)\cong\Z/3\Z$.

    For $\tau = \frac{i}{\sqrt{3}}$,
    $$\Aut(\widetilde{X}_{\tau}) = \langle t:(x,y)\mapsto(\zeta_{6}x,y), s:(x,y)\mapsto(x,-y),r:(x,y)\mapsto(1/x,y/x^3)\rangle,$$ $\SO(\widetilde{X}_{\tau},\omega_\tau)\cong\langle t, s\rangle\cong\Z/6\Z\times\Z/2\Z$, and $\Aut(\widetilde{X}_{\tau},\{\pm\omega_{\tau}\})\cong \langle t^3, s\rangle\cong\Z/2\Z\times\Z/2\Z$ . Hence, again we have $\PSO(\widetilde{X}_{\tau},\omega_\tau)\cong\Z/3\Z$.
\end{proof}

\begin{proposition}
    Let $\tau\in\widetilde{U}$ be imaginary quadratic. Then $(\widetilde{X}_{\tau},[\omega_{\tau}])$ maps to an orbifold point on $X_{D}$ for some discriminant $D>0$.
\end{proposition}

\begin{proof}
    By Proposition~\ref{prop:CM}, we know that $\Jac(\widetilde{X}_{\tau})$ admits complex multiplication by an order $\mathcal{O}\subset\Q(\zeta_{6},\tau)$. Letting $D$ be the discriminant of $\mathcal{O}_{D}:=\mathcal{O}\cap\R$, we get that $\omega_{\tau}$ is an $\mathcal{O}_{D}$-eigenform. Hence, since $|\PSO(\widetilde{X}_{\tau},\omega_{\tau})| = 3 > 1$, $(\widetilde{X}_{\tau},[\omega_{\tau}])$ maps to an orbifold point on $X_{D}$.
\end{proof}

From~\cite[Table 2]{Muk}, it can be checked that the only automorphisms $\phi$ of a genus two curve having eigenforms $\omega$ with simple poles, and with $\phi^{*}\omega\neq \pm \omega$ are the order 5 and order 10 automorphisms of $y^2 = x^5+1$, and the order 3 and order 6 automorphisms on our curves in $\mathcal{M}_{2}(D_{12})$.

\begin{proposition}
    Let $(X,[\omega])\in\mathbb{P}\calH(1,1)$ map to an orbifold point in $X_{D}$ for some $X>0$. Then either
    \begin{enumerate}
        \item $(X,[\omega])$ maps to an order 5 orbifold point on $X_{5}$ and, up to scale, $(X,\omega)$ is the decagon; or
        \item $(X,[\omega]) = (\widetilde{X}_{\tau},\omega_{\tau})$ for some imaginary quadratic $\tau\in\widetilde{U}$ and maps to an orbifold point of order 3 on $X_{D}$.
    \end{enumerate}
\end{proposition}

\begin{proof}
    From the paragraph above, $X$ either admits an automorphism of order 5 or 10 and is (up to scale) the decagon curve $(X:y^2=x^5+1,xdx/y)$ which maps to $X_{5}$, or $X\in\mathcal{M}_{2}(D_{12})$ and, from the above propositions, the second claim holds.
\end{proof}

Given a CM elliptic curve $E$, there are four curves $(X_i,\rho_i)\in\mathcal{M}_{2}(D_{12})$ that cover $E$ --- one for each choice of order three subgroup $C_{i}\leq E[3]$. Namely, if $E\cong\C/\mathcal{O}_{C}$ for some discriminant $C<0$, then we have $g(X_i,\rho_i) = (E,C_i)$ with 
\[C_i = \left\lbrace\begin{array}{cl}
    \langle\frac{1}{3}+\mathcal{O}_{C}\rangle, & i = 1 \\
    \langle\frac{C+\sqrt{C}}{6}+\mathcal{O}_{C}\rangle, & i = 2 \\
    \langle\frac{1}{3}+\frac{C+\sqrt{C}}{6}+\mathcal{O}_{C}\rangle, & i = 3 \\
    \langle\frac{1}{3}+\frac{C+\sqrt{C}}{3}+\mathcal{O}_{C}\rangle, & i = 4.
\end{array}\right.\]
Hence, $X_{i} = \widetilde{X}_{\tau}$ for 
\[\tau = \left\lbrace\begin{array}{cl}
    \frac{C+\sqrt{C}}{2}, & i = 1 \\
    \\
    \frac{2\sqrt{C}-2C}{C(C-1)}, & i = 2 \\
    \\
    \frac{2\sqrt{C}+C^2+C}{C^2+3C+4}, & i = 3 \\
    \\
    \frac{2\sqrt{C}+3C-C^2}{C^2-5C+4}, & i = 4.
\end{array}\right.\]
Since CM elliptic curves correspond to imaginary quadratics in $\mathbb{H}$, the Jacobians of the curves $\widetilde{X}_{\tau}$ will admit complex multiplication. The following proposition determines the discriminant of the corresponding real multiplication.

\begin{proposition}
    Suppose that $g(X_{i},\rho_i) = (E,C_{i})$, with $E\cong\C/\mathcal{O}_{C}$ and let $D_{i}$ be the discriminant of real multiplication on $\Jac(X_{i})$. Then $D_{i}$ is determined by Table~\ref{tab:discriminants}.
\end{proposition}

\begin{table}[H]
    \centering
    \begin{tabular}{|c|c|c|c|c|}
        \hline
        $C\!\!\mod 3$ & $D_1$ & $D_2$ & $D_3$ & $D_4$ \\
        \hline
        $0$ & $-3C$ & $-C/3$ & $-3C$ & $-3C$ \\
        $1,2$ & $-3C$ & $-3C$ & $-3C$ & $-3C$ \\
        \hline
    \end{tabular}
    \vspace{5pt}
    \caption{The elliptic curve $E= \C/\mathcal{O}_C$ is covered by four curves $(X_i,\rho_i)\in\mathcal{M}_{2}(D_{12})$. The discriminant $D_i$ of the order for real multiplication on $\Jac(X_i)$ is determined from $C$.}
    \label{tab:discriminants}
\end{table}

\begin{proof}
    Let $\tau = \frac{C+\sqrt{C}}{2}$. It can be checked that the largest order for complex multiplication in $\Q(\zeta_6,\tau)$ is $\mathcal{O} = \Z[\zeta_6,\sqrt{3}\tau i]$ giving the order for real multiplication $\mathcal{O}_{D} = \mathcal{O}\cap\R = \Z\left[\frac{-3C-\sqrt{-3C}}{2}\right]\cong\mathcal{O}_{-3C}$. This completes the $D_{1}$ column of Table~\ref{tab:discriminants}.

    Now let $\tau = \frac{2\sqrt{C}-2C}{C(C-1)}$. It can be checked that the maximal order for complex multiplication is
    \[\mathcal{O} = \left\lbrace\begin{array}{ll}
        \Z\left[\zeta_{6},\sqrt{3}\frac{C(C-1)}{4}\tau i\right] = \Z\left[\zeta_6,\frac{-\sqrt{-3C}+\sqrt{3}Ci}{2}\right], & 3 \not{\mid}\,\,C,(C-1) \\
        \\
        \Z\left[\zeta_{6},\sqrt{3}\frac{C(C-1)}{12}\tau i\right]= \Z\left[\zeta_6,\frac{-\sqrt{-3C}+\sqrt{3}Ci}{6}\right], & 3 \mid (C-1) \\
        \\
        \Z\left[\zeta_{6},\sqrt{3}\frac{C(C-1)}{12}\tau i\right]= \Z\left[\zeta_6,\frac{-\sqrt{-C/3}+\sqrt{3}(C/3)i}{2}\right], & 3 \mid C. \\
    \end{array}\right.\]
    In the first two cases, $\mathcal{O}_{D}\cong\mathcal{O}_{-3C}$ while in the latter case it is isomorphic to $\mathcal{O}_{-C/3}$. This completes the proof of the $D_{2}$ column of Table~\ref{tab:discriminants}.

    For $\tau = \frac{2\sqrt{C}+C^2+C}{C^2+3C+4}$, $\mathcal{O} = \Z\left[\zeta_6,\sqrt{3}\frac{C^2+3C+4}{4}\tau i\right]$ and $\mathcal{O}_{D}\cong\mathcal{O}_{-3C}.$

    Finally, for $\tau = \frac{2\sqrt{C}+3C-C^2}{C^2-5C+4}$, $\mathcal{O} = \Z\left[\zeta_6,\sqrt{3}\frac{C^2-5C+4}{4}\tau i\right]$ and again $\mathcal{O}_{D}\cong\mathcal{O}_{-3C}.$

    For elliptic curves $E_{\tau}$ with $\End(E_\tau)\cong\mathcal{O}_{C}$ but not of the form $\C/\mathcal{O}_{C}$, one can apply the same arguments of Mukamel --- namely, the proof of~\cite[Proposition 4.4]{Muk} and the statement of~\cite[Proposition 4.5]{Muk} --- to deduce that the associated discriminants of real multiplication behave as in Table~\ref{tab:discriminants}.
\end{proof}

Solving $d^2 = -3C$ or $-C/3$ gives the dependence on $d^2$ seen in Theorem~\ref{thm:e3}. It is well known that the class number counts the number of elliptic curves with endomorphism ring an imaginary quadratic extension of a given discriminant. The factor of a $1/2$ arises because $(X,\rho)$ and $\sigma\cdot(X,\rho)$ have isomorphic $\rho(Z)$-eigenforms and so map to the same orbifold point. The reduced class number handles the additional symmetries of elliptic curves with complex multiplication of small discriminant.

The dependence on $\epsilon$ can be argued as follows. For $\epsilon = 1$, the associated HLK-invariants satisfy the conditions giving rise to the block system of Figure~\ref{fig:blocksys4}. This block system rules out all of the words that give rise to order 3 or order 6 elements of $\SL(2,\Z)$ and so we cannot have orbifold points of order 3 in these orbits.

We are then left to determine the dependence on $n$. When $n$ is even, the same block system argument of the previous paragraph holds to rule out orbifold points. For odd $n$, we determine the index of the absolute periods ${\rm Per}(\omega)$ inside the relative periods ${\rm RPer}(\omega)$ for the curves $\widetilde{X}_{\tau}$. We use the pinwheel description of Figure~\ref{fig:pinwheel} to carry out this calculation.

We first remark that $C\equiv 0\!\!\mod 3$ when the discriminants $D_i$ are square. Indeed, $d^2\equiv0,1\!\!\mod 3$ and so $-d^2\equiv0,2\!\!\mod 3$. In the first case, either $C = -d^2/3$ in which case it is also divisible by $3$ since $d^2$ is divisible by 9, or $C = -3d^2$. In the latter case, we can only have $C = -3d^2$. So we need only care about the top row of Table~\ref{tab:discriminants}.

In the case of $D_{1}$, we have $d\equiv 0\!\!\mod 3$, $C = -\frac{d^2}{3}$ and $\tau = \frac{C+\sqrt{C}}{2} = -\frac{d^2}{6}+d\frac{\sqrt{3}}{6}i$. By applying powers of $\begin{bsmallmatrix}1 & 1 \\ 0 & 1\end{bsmallmatrix}$, the representative of $\tau$ inside $\widetilde{U}$ will be
\[\left\lbrace\begin{array}{cc}
    d\frac{\sqrt{3}}{6}i, & d\equiv0\!\!\mod 6 \\
    \\
    \frac{1}{2}+d\frac{\sqrt{3}}{6}i, & d\equiv3\!\!\mod 6. 
\end{array}\right.\]
In both cases, it can be checked that ${\rm Per}(\omega) = \Z\oplus\zeta_{3}\Z$ while ${\rm RPer}(\omega) = \Z\oplus(\zeta_{12}/\sqrt{3})\Z$. Hence, $[{\rm RPer}(\omega):{\rm Per}(\omega)] = 3$ and we have $\widetilde{X}_{\tau}\in W_{d^2}[3]$.

For $D_2$, $d$ can have any value modulo 3, $C = -3d^2$ and $\tau = \frac{2\sqrt{C}-2C}{C(C-1)}$. The representative in $\widetilde{U}$ can be calculated to be
\[\left\lbrace\begin{array}{cc}
    d\frac{\sqrt{3}}{6}i, & d\equiv0,2,4\!\!\mod 6 \\
    \\
    \frac{1}{2}+d\frac{\sqrt{3}}{6}i, & d\equiv1,3,5\!\!\mod 6. 
\end{array}\right.\]
In all cases, ${\rm Per}(\omega) = \Z\oplus\zeta_{3}\Z$, as above. For $d\equiv0,2,3,5\!\!\mod 6$, ${\rm RPer}(\omega) = \Z\oplus(\zeta_{12}/\sqrt{3})\Z$ and we again have that $\widetilde{X}_{\tau}\in W_{d^2}[3]$. While, for $d\equiv1,4\!\!\mod 6$, we get ${\rm RPer}(\omega) =  \Z\oplus\zeta_{3}\Z = {\rm Per}(\omega)$ and so $\widetilde{X}_{\tau}\in W_{d^2}[1]$.

For $D_3$, again we have $d\equiv 0\!\!\mod 3$ and $C = -\frac{d^2}{3}$. This time $\tau = \frac{2\sqrt{C}+C^2+C}{C^2+3C+4}$ has $\widetilde{U}$ representative
\[\left\lbrace\begin{array}{cc}
    \frac{1}{3}+d\frac{\sqrt{3}}{18}i, & d\equiv0\!\!\mod 6 \\
    \\
    -\frac{1}{6}+d\frac{\sqrt{3}}{18}i, & d\equiv3\!\!\mod 6. 
\end{array}\right.\]
Here, ${\rm Per}(\omega) = \Z\oplus(\zeta_{12}/\sqrt{3})\Z$ and ${\rm RPer}(\omega) = (1/3)\Z\oplus(\zeta_{12}/\sqrt{3})\Z$. So again, $[{\rm RPer}(\omega):{\rm Per}(\omega)] = 3$ and we have $\widetilde{X}_{\tau}\in W_{d^2}[3]$.

For $D_4$, the $\widetilde{U}$ representatives of $\tau = \frac{2\sqrt{C}+3C-C^2}{C^2-5C+4}$ are 
\[\left\lbrace\begin{array}{cc}
    -\frac{1}{3}+d\frac{\sqrt{3}}{18}i, & d\equiv0\!\!\mod 6 \\
    \\
    \frac{1}{6}+d\frac{\sqrt{3}}{18}i, & d\equiv3\!\!\mod 6. 
\end{array}\right.\]
and similar calculations to the above give $[{\rm RPer}(\omega):{\rm Per}(\omega)] = 3$ and we have $\widetilde{X}_{\tau}\in W_{d^2}[3]$.

This completes the proof of Theorem~\ref{thm:e3}.

%%%%%%%%%%%%%%%%%%%%%%%%%%%%%%%%%%%%%%%%%
%%%%%%%%%%%%%%%%%%%%%%%%%%%%%%%%%%%%%%%%%

\section{Two families of non-Prym orbits in $\calH(4)$}\label{sec:non-prym}

Here, we consider $\SL(2,\Z)$-orbits of primitive origamis in $\calH(4)$ outside of the Prym locus. Recall that these are subject to Conjecture~\ref{conj:Del-Lel} of Delecroix--Leli\`evre.

\begin{conjDelLel}
    For $n > 8$, the number of $\SL(2,\Z)$-orbits of primitive $n$-squared origamis in $\calH(4)$ is as follows:
    \begin{itemize}
        \item there are precisely two such $\SL(2,\Z)$-orbits in $\odd(4)$ outside of the Prym locus, distinguished by their monodromy groups being $\Alt(n)$ or $\Sym(n)$;
        \item for odd $n$, there are precisely four such $\SL(2,\Z)$-orbits in $\hyp(4)$, distinguished by their HLK-invariant being $(4,[1,1,1]), (2,[3,1,1]), (0,[5,1,1])$ or $(0,[3,3,1])$;
        \item for even $n$, there are precisely three such $\SL(2,\Z)$-orbits in $\hyp(4)$, distinguished by their HLK-invariant being $(3,[2,2,0]), (1,[4,2,0])$ or $(1,[2,2,2])$.
    \end{itemize}    
\end{conjDelLel}

We will assume that this conjecture holds throughout this section and also that the conjecture is implicitly stating that each orbit has growth $\Omega(n^{\dim_\C\mathcal{H}(4)-1}) = \Omega(n^{5})$.

\subsection{The symmetric monodromy orbits}\label{subsec:H4-sym}

Assuming Conjecture~\ref{conj:Del-Lel}, let $(\mathcal{G}_{n})_{n}$ be the family of orbits in $\odd(4)$ with monodromy group $\Sym(n)$.

Hyperbolic faces of length at most 4 are handled as above and asymptotically do not appear.

By Proposition~\ref{prop:symblocksys}, such an orbit permits the block system of Figure~\ref{fig:blocksys4}. This block system then rules out all words in $T$ and $S$ that gives rise to elements of order 3 or 6. Moreover, since the orbit is non-Prym and non-hyperelliptic there are no elements of order 2 or 4 in the Veech groups of any surface in the orbit. Indeed, if $-I$ (the only element of order 2 in $\SL(2,\Z)$ and the square of any order 4 element) is in the Veech group of some origami in the orbit, then it appears in the Veech group of all origamis in the orbit because $-I$ lies in the center of $\SL(2,\Z)$ and so commutes with $T$ and $S$. If the orbit permitted $-I$ as a global symmetry, then it would be Prym of hyperelliptic, which is a contradiction. Hence, $\mathcal{G}_{n}$ has no elliptic faces.

Notice that the previous argument proves that the associated \teichmuller curves have no orbifold points.

We are left to bound the parabolic faces of length at most 4, which again can be done by bounding the number of $T$-loops. This is carried out in Subsection~\ref{subsec:H4-paras}.

\subsection{The $(1,[4,2,0])$-orbits}

Consider now, for even $n$, the orbit graphs $(\mathcal{G}_{n})_n$ of the $\SL(2,\Z)$-orbits in $\hyp(4)$ with HLK-invariant $(1,[4,2,0])$.

By Proposition~\ref{prop:all_diff}, we have the block system of Figure~\ref{fig:blocksys1}. This block system rules out all faces of length at most 4 apart from $T^2,S^2, T^4$ and $S^4$. In particular, there are no hyperbolic faces of length at most 4, and no elliptic faces, which also implies no orbifold points on the associated \teichmuller curves.

Although there are no loops, the count of loops in the next subsection also bounds the cusps of length two or four.

\subsection{Parabolic faces in $\calH(4)$ outside of the Prym locus}\label{subsec:H4-paras}

We will refer to the separatrix diagrams in $\calH(4)$ determined by Leli\`evre in~\cite[Appendix C]{MMY}. We will not reproduce the figures in this paper.

An origami in $\calH(4)$ can have one, two or three cylinders. One cylinder origamis will not asymptotically contribute to small faces. So we count the number of two- and three-cylinder origamis fixed by $T$.

In the case of three cylinders and checking the appropriate figures in~\cite[Appendix C]{MMY}, each origami is uniquely determined modulo twists by the cylinder widths $w_i$. There are also only finitely many three-cylinder configurations. The equations for loops are now
\[xw_1^2+yw_2^2+zw_3^2 = n\]
and they do not simplify to conics. As such, we can no longer make use of Bombieri--Pila. Instead, we again make use of the result of Browning--Gorodnik~\cite[Theorem 1.11]{BG} that for a quadric of rank $r$ gives $O(N^{r-2+\epsilon})$ lattice points inside a box of side length $N$. Here, we have $r = 3$.

If $1\leq \min\{x,y,z\}< \sqrt{n}$, then the quadric lies inside a box of side length $N=\sqrt{n}$. So, by Browning--Gorodnik, we have $O(n^{\frac{1}{2}+\frac{1}{2}\epsilon})$ lattice points. We have $O(n^{\frac{5}{2}})$ quadrics. Each solution determines a unique origami modulo cylinder twists. We have $w_1w_2w_3 = O(n^{\frac{3}{2}})$ choices of cylinder twists. Hence, we get $O(n^{\frac{9}{2}+\frac{1}{2}\epsilon})$ origamis fixed by $T$.

If $\sqrt{n}\leq x, y, z\leq n$, the box now has side length $N = n^{\frac{1}{4}}$. So, we have $O(n^{\frac{1}{4}+\frac{1}{4}\epsilon})$ lattice points. We have $O(n^3)$ quadrics. We have $w_1w_2w_3 = O(n^{\frac{3}{4}})$ choices of twists. So, in total, we get $O(n^{4+\frac{1}{4}\epsilon})$ origams fixed by $T$.

For two-cylinder origamis, given the cylinder widths $(w_1,w_2)$ there are two saddle connection lengths, say $l_1$ and $l_2$, that can be varied to achieve different origamis modulo cylinder twists. Each $l_i$ is certainly bounded above by $\max\{w_1,w_2\}$. Here, we are back to our familiar equation
\[xw_1^2+yw_2^2 = n.\]
Although we could now apply Bombieri--Pila again, we will continue with the improved bounds of Browning--Gorodnik.

If $1\leq \min\{x,y\}< \sqrt{n}$, the ellipse fits inside a box of side length $\sqrt{n}$ and so we have $O(n^{\frac{1}{2}\epsilon})$ lattice points. We have $O(n^{\frac{3}{2}})$ ellipses. Each solution $(w_1,w_2)$ determines $(\max\{w_1,w_2\})^{2} = O(n)$ origamis modulo twists and we have $w_1w_2 = O(n)$ choices of twists. So we have $O(n^{\frac{7}{2}+\frac{1}{2}\epsilon})$ origamis fixed by $T$.

If $\sqrt{n}\leq x,y\leq n$, the ellipse fits inside a box of side length $n^{\frac{1}{4}}$ and so we have $O(n^{\frac{1}{4}\epsilon})$ lattice points. We have $O(n^2)$ ellipses. Each solution $(w_1,w_2)$ determines $(\max\{w_1,w_2\})^{2} = O(n^{\frac{1}{2}})$ origamis modulo twists and we have $w_1w_2 = O(n^{\frac{1}{2}})$ choices of twists. So we have $O(n^{3+\frac{1}{4}\epsilon})$ origamis fixed by $T$.

Hence, we have at most $O(n^{\frac{9}{2}+\frac{1}{4}\epsilon})$ loops in the orbit graphs and the same bound works for all parabolic faces of length at most four. This completes the proof of Theorem~\ref{thm:submain-4}.

%%%%%%%%%%%%%%%%%%%%%%%%%%%%%%%%%%%%%%%%%
%%%%%%%%%%%%%%%%%%%%%%%%%%%%%%%%%%%%%%%%%

\section{Low-complexity strata}\label{sec:general}

In this section, we consider a family of $\SL(2,\Z)$-orbit graphs of primitive $n$-squared origamis in a stratum $\calH(k_1,\ldots,k_s)$ with $1\leq s\leq 2$ and assume that these orbits are not contained in any strictly smaller arithmetic subvariety. We will throughout assume Conjecture~\ref{conj:wild}.

\begin{conjwild}
    Given a family $(\mathcal{G}_{n})_{n}$ of $\SL(2,\Z)$-orbit graphs of primitive $n$-squared origamis in a connected component of a stratum $\mathcal{H}:=\calH_{g}(k_{1},\dots,k_{s})$ not contained in a strictly smaller arithmetic subvariety, we have $|\mathcal{G}_{n}| = \Omega(n^{\dim_{\C}\mathcal{H}-1}) = \Omega(n^{2g+s-2})$.
\end{conjwild}

\subsection{Hyperbolic faces}

As we have seen throughout this work, Corollary~\ref{cor:hyperbolic} is sufficient to rule out all hyperbolic faces of length at most four in a family of orbit graphs once the number of squares is large enough.

\subsection{Parabolic faces}

The behaviour that we saw in $\calH(2), \calH(1,1)$ and $\calH(4)$ continues in a general stratum $\calH(k_1,\ldots,k_s)$. In this case, any origami has between $1$ and $g+s-1$ cylinders.

If we have a one-cylinder origami, then the top and bottom of the cylinder are both composed of $2g+s-2$ horizontal saddle connections. If we have fixed the width of the cylinder, then we have $2g+s-3$ saddle connections whose lengths can be altered.

Going up to a two-cylinder origami where we have fixed the cylinder widths, we will have $2g+s-5$ saddle connections that can be altered and the cusp equation --- that is, solving for cusps of width one --- will be a rank two quadric.

As a sanity check, compare this to the two-cylinder calculations we performed above. In $\calH(2)$, we solved a rank two quadric (in particular an ellipse) for the widths and had no degrees of freedom in the saddle connection lengths, and indeed $2g+s-5 = 0$ here. In $\calH(1,1)$, we solved a rank two quadric for the widths and had $1 = 2g+s-5$ degrees of freedom in the saddle connection lengths. While in $\calH(4)$, we solved a rank two quadric for the widths and had $2 = 2g+s-5$ degrees of freedom in saddle connection lengths.

Going further still, a three-cylinder origami with fixed cylinder widths has $2g+s-7$ degrees of freedom in the saddle connection lengths. The cusp equation is now a rank 3 quadric, up to some modification if $2g+s-7 < 0$.

To make sense of the previous sentence, let us compare the three-cylinder calculations we performed in $\calH(1,1)$ and $\calH(4)$. We start with the latter. Here, we solved a rank 3 quadric for the widths and we had $0 = 2g+s-7$ degrees of freedom in the saddle connection lengths. For $\calH(1,1)$, $2g+s-7 = -1$. So what does it mean to have $-1$ degrees of freedom in the saddle connection lengths? Recall from Subsection~\ref{subsec:H(1,1)-paras}, that the cusp equation was initially a rank 3 quadric. However, we realised that one of the cylinder widths was equal to the sum of the other two widths and this allowed us to reduce to a rank 2 quadric. So the negative degrees of freedom remove rank from the cusp equation.

Carrying on in this fashion, assume that we have $2\leq c\leq g+s-1$ cylinders. The cusp equation will initially have rank $c$ and we will have $2g+s-1-2c$ degrees of freedom in saddle connection lengths, with the caveat that if $2g+s-1-2c = -t < 0$ then the cusp equation instead reduces to a quadric of rank $c-t$.

Suppose then that we have an $n$-squared origami in $\calH(k_1,\ldots,k_s)$ with $2\leq c\leq g+s-1$ cylinders. The non-reduced cusp equation has coefficients $x_{i}$ for $1\leq i \leq c$. The rank of the equation will be $r = \min\{c,2g+s-1-c\}$. There will be $k = \max\{0,2g+s-1-2c\}$ degrees of freedom in saddle connection lengths. It can be checked that $r = c \Leftrightarrow k = 2g+s-1-2c$, so that $r+k = 2g+s-1-c$.

Next, we fix $0<\beta<1$ such that $\frac{s-1}{2}<\beta<\frac{2g-2}{2g+s-3}$: note that the existence of such a parameter follows from the fact that $(s-1)(2g-2+s-1) = (s-1)(2g-2)+(s-1)^2 < 2(2g-2)$ for $s\leq 2\leq g$.

Suppose $1\leq \min x_{i}<n^{1-\beta}$. Then we have $O(n^{c-\beta})$ quadrics. Each quadric sits inside a box of side length $N = \sqrt{n}$. Each quadric has $O(n^{\frac{r}{2}-1+\frac{\epsilon}{2}})$ lattice points. Each solution gives rise to $O(n^{\frac{k}{2}})$ origamis modulo twists. Each origami has $O(n^{\frac{c}{2}})$ choices of twists. So we get $O(n^{c-\beta+\frac{2g+s-1}{2}-1+\frac{\epsilon}{2}}) \leq O(n^{2g+s-2-\beta+\frac{s-1}{2}+\frac{\epsilon}{2}}) = o(n^{2g+s-2})$ origamis fixed by $T$ thanks to our choice of $\beta>\frac{s-1}{2}$.

If instead $n^{1-\beta}\leq x_{i}\leq n$, then we have $O(n^c)$ quadrics. Each quadric sits inside a box of side length $N = n^{\frac{\beta}{2}}$. Each quadric has $O(n^{\frac{\beta}{2}(r-2+\epsilon)})$ lattice points. Each solution gives rise to $O(n^{\frac{\beta}{2}k})$ origamis modulo twists. Each origami has $O(n^{\frac{\beta}{2}c})$ choices of twists. So we get $O(n^{c+\frac{\beta}{2}(r+k+c-2+\epsilon)}) \leq O(n^{g+s-1+\frac{\beta}{2}(2g+s-3+\epsilon)}) = o(n^{2g+s-2})$ origamis fixed by $T$ due to our choice of $\beta<\frac{2g-2}{2g+s-3}$.

Hence, assuming Conjecture~\ref{conj:wild}, the number of parabolic faces of length at most four is $o(|\mathcal{G}_{n}|) = o(n^{\dim_\C\mathcal\calH(k_1,\ldots,k_s) -1}) = o(n^{2g+s-2})$.

\subsection{Elliptic faces}\label{subsec:gen-ells}

We will now argue that for a density one subsequence of $n\in\N$, the number of $n$-squared origamis in $\calH(k_1,\dots,k_s)$ that have finite order elements of the Veech group (of order at least three) behaves as $o(n^{\dim_\C(k_1,\dots,k_s) - 1})$.

Denote by ${\rm Cov}_d(\underline{\kappa})$ the set of $d$-squared origamis in $\calH(\underline{\kappa})$, $\underline{\kappa}=(k_1,\dots,k_s)$, and let $\mu_n$ be the measure 
$$\mu_n=\frac{1}{n^{\dim_\C\calH(k_1,\ldots,k_s)}}\sum\limits_{d=1}^n \sum\limits_{S\in {\rm Cov}_d(\underline{\kappa})}\delta_S.$$ 
As it is shown in \cite[Section 3]{EsOk}, this sequence of measures converges in the weak-$\ast$ topology to a multiple of the Masur-Veech probability measure $\mu_{\underline{\kappa}}$: in particular, the cardinality of ${\rm Cov}_d(\underline{\kappa})$ has order $d^{\dim_\C\calH(k_1,\ldots,k_s)-1}$ on average. On the other hand, it is known (see, e.g., the proof of~\cite[Theorem 1.1]{Mo}) that the locus of orbifold points in $\calH(\underline{\kappa})$ is a locus of zero Masur--Veech measure consisting of a locally finite union of countably many subvarieties: consequently, the average of the number of orbifold points in ${\rm Cov}_d(\underline{\kappa})$ for $1\leq d\leq n$ is $o(n^{\dim_\C\calH(k_1,\ldots,k_s)-1})$.

Thus, assuming Conjecture~\ref{conj:wild}, we conclude that the number of elliptic faces in $\mathcal{G}_{n}$ is $o(|\mathcal{G}_{n}|)$ for a typical large $n$.

This completes the proof of Theorem~\ref{thm:general}.

\begin{remark}
    If we were in a situation as in Subsection~\ref{subsec:H4-sym} ---  that is, if our origamis had $-I\not\in\SL(X,\omega)$ and had symmetric monodromy --- then we would have no orbifold points and so would not need to pass to a density one subsequence. There are reasons to believe that such a situation is reasonably generic. 
\end{remark}

%%%%%%%%%%%%%%%%%%%%%%%%%%%%%%%%%%%%%%%%%
%%%%%%%%%%%%%%%%%%%%%%%%%%%%%%%%%%%%%%%%%

\section{Elliptic generators and arithmetic \teichmuller curve genera}\label{sec:curve-genus}

In this final section, we will prove the elliptic generator part of Theorem~\ref{thm:main} and Theorems~\ref{thm:prym-genus}, ~\ref{thm:H(1,1)-genus} and~\ref{thm:general-genus} at the same time. Here, we consider orbit graphs $\mathcal{G}_{n}$ with the elliptic generators $R:=\begin{bsmallmatrix}
    0 & -1 \\
    1 & 0
\end{bsmallmatrix}$ and $U:=\begin{bsmallmatrix}
    0 & 1 \\
    -1 & 1
\end{bsmallmatrix}$ and restrict to orbits where $-I$ is in the Veech group of every origami. This includes $\calH(2)$, $\calH(1,1)$, the Prym loci of $\calH(4)$ and $\calH(6)$, and hyperelliptic connected components in general. Let $\mathcal{C}_{n}$ denote the arithmetic \teichmuller curve associated to the orbit $\mathcal{G}_{n}$.

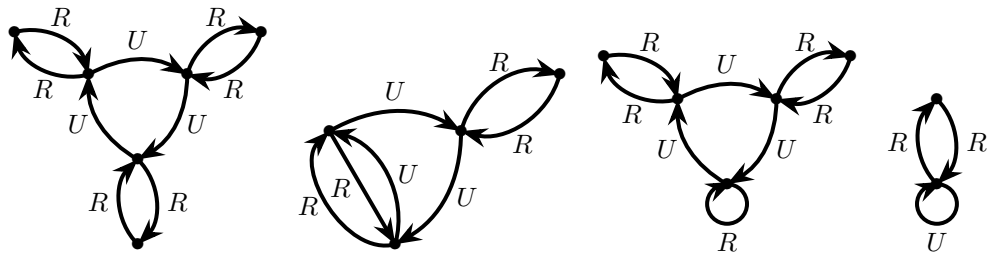
\begin{figure}[b]
    \centering
    \begin{tikzpicture}[scale = 0.75, line width = 1.5pt]
        \draw [-Stealth, bend left = 30, font = \small] (0,-1) to node[midway, left]{$U$} (-0.8660254038,0.5);
        \draw [-Stealth, bend left = 30, font = \small] (-0.8660254038,0.5) to node[midway, above]{$U$} (0.8660254038,0.5);
        \draw [-Stealth, bend left = 30, font = \small] (0.8660254038,0.5) to node[midway, right]{$U$} (0,-1);
        \draw [-Stealth, bend left = 45, font = \small] (0,-1) to node[midway, right]{$R$} (0,-2.5);
        \draw [-Stealth, bend left = 45, font = \small] (0,-2.5) to node[midway, left]{$R$} (0,-1);
        \draw [-Stealth, bend left = 45, font = \small] (-0.8660254038,0.5) to node[midway, below]{$R$} (-0.8660254038*2.5,0.5*2.5);
        \draw [-Stealth, bend left = 45, font = \small] (-0.8660254038*2.5,0.5*2.5) to node[midway, above]{$R$} (-0.8660254038,0.5);
        \draw [-Stealth, bend left = 45, font = \small] (0.8660254038,0.5) to node[midway, above]{$R$} (0.8660254038*2.5,0.5*2.5);
        \draw [-Stealth, bend left = 45, font = \small] (0.8660254038*2.5,0.5*2.5) to node[midway, below]{$R$} (0.8660254038,0.5);
        \draw (0.8660254038,0.5)node{$\bullet$};
        \draw (-0.8660254038,0.5)node{$\bullet$};
        \draw (0,-1)node{$\bullet$};
        \draw (0,-2.5)node{$\bullet$};
        \draw (0.8660254038*2.5,0.5*2.5)node{$\bullet$};
        \draw (-0.8660254038*2.5,0.5*2.5)node{$\bullet$};
    \end{tikzpicture}
    \begin{tikzpicture}[scale = 1, line width = 1.5pt]
        \draw [-Stealth, bend right = 45, font = \small] (0,-1) to node[midway, right]{$U$} (-0.8660254038,0.5);
        \draw [-Stealth, bend left = 30, font = \small] (-0.8660254038,0.5) to node[midway, above]{$U$} (0.8660254038,0.5);
        \draw [-Stealth, bend left = 30, font = \small] (0.8660254038,0.5) to node[midway, right]{$U$} (0,-1);
        \draw [-Stealth, font = \small] (-0.8660254038,0.5) to node[midway, left]{$R$} (0,-1);
        \draw [-Stealth, bend left = 75, font = \small] (0,-1) to node[midway, left]{$R$} (-0.8660254038,0.5);
        \draw [-Stealth, bend left = 45, font = \small] (0.8660254038,0.5) to node[midway, above]{$R$} (0.8660254038*2.5,0.5*2.5);
        \draw [-Stealth, bend left = 45, font = \small] (0.8660254038*2.5,0.5*2.5) to node[midway, below]{$R$} (0.8660254038,0.5);
        \draw (0.8660254038,0.5)node{$\bullet$};
        \draw (-0.8660254038,0.5)node{$\bullet$};
        \draw (0,-1)node{$\bullet$};
        \draw (0.8660254038*2.5,0.5*2.5)node{$\bullet$};
    \end{tikzpicture}
    \begin{tikzpicture}[scale = 0.75, line width = 1.5pt]
        \draw [-Stealth, bend left = 30, font = \small] (0,-1) to node[midway, left]{$U$} (-0.8660254038,0.5);
        \draw [-Stealth, bend left = 30, font = \small] (-0.8660254038,0.5) to node[midway, above]{$U$} (0.8660254038,0.5);
        \draw [-Stealth, bend left = 30, font = \small] (0.8660254038,0.5) to node[midway, right]{$U$} (0,-1);
        \draw [-Stealth, bend left = 45, font = \small] (-0.8660254038,0.5) to node[midway, below]{$R$} (-0.8660254038*2.5,0.5*2.5);
        \draw [-Stealth, bend left = 45, font = \small] (-0.8660254038*2.5,0.5*2.5) to node[midway, above]{$R$} (-0.8660254038,0.5);
        \draw [-Stealth, bend left = 45, font = \small] (0.8660254038,0.5) to node[midway, above]{$R$} (0.8660254038*2.5,0.5*2.5);
        \draw [-Stealth, bend left = 45, font = \small] (0.8660254038*2.5,0.5*2.5) to node[midway, below]{$R$} (0.8660254038,0.5);
        \draw (0,-1) arc[start angle=90, end angle=-270, radius=0.35];
        \draw[-Stealth, rotate around={25:(0,-1)}] (0,-1) -- ++(0.1,0);
        \draw (0,-1.65)node[below, font = \small]{$R$};
        \draw (0.8660254038,0.5)node{$\bullet$};
        \draw (-0.8660254038,0.5)node{$\bullet$};
        \draw (0,-1)node{$\bullet$};
        \draw (0.8660254038*2.5,0.5*2.5)node{$\bullet$};
        \draw (-0.8660254038*2.5,0.5*2.5)node{$\bullet$};
    \end{tikzpicture}
    \begin{tikzpicture}[scale = 0.75, line width = 1.5pt]
        \draw [-Stealth, bend left = 45, font = \small] (0,0.5) to node[midway, right]{$R$} (0,-1);
        \draw [-Stealth, bend left = 45, font = \small] (0,-1) to node[midway, left]{$R$} (0,0.5);
        \draw (0,-1) arc[start angle=90, end angle=-270, radius=0.35];
        \draw[-Stealth, rotate around={25:(0,-1)}] (0,-1) -- ++(0.1,0);
        \draw (0,-1.65)node[below, font = \small]{$U$};
        \draw (0,-1)node{$\bullet$};
        \draw (0,0.5)node{$\bullet$};
    \end{tikzpicture}
    \caption{The possible local structures of the elliptic generator orbit graphs when $-I$ is a global symmetry.}
    \label{fig:R-U-local}
\end{figure}

Since we have that $-I$ is in the Veech group of every origami in the orbit, and since $R^{2} = -I = U^3$, local pictures of the graph look as in Figure~\ref{fig:R-U-local}.

We see that there will be $e_{2}(\mathcal{C}_{n})$ $R$-loops. Every other vertex will lie in an $R^{2}$ 2-cycle, and so we will have $\frac{V-e_{2}(\mathcal{C}_{n})}{2}$ such 2-cycles. There will be $e_{3}(\mathcal{C}_{n})$ $U$-loops. Every other vertex will lie in a $U^{3}$ 3-cycle, and so there will be $\frac{V-e_{3}(\mathcal{C}_{n})}{3}$ such 3-cycles.

From Figure~\ref{fig:R-U-local}, we see that a minimal embedding can be chosen so that all other faces correspond to words of the form $(RU)^{k}$ for some $k$. Now, $RU = T^{-1}$. So we see that every other face of the graph corresponds exactly to a cusp of $\mathcal{C}_{n}$ and we also see that every cusp will appear as such a face. So, letting $c(\mathcal{C}_{n})$ denote the number of cusps of $\mathcal{C}_{n}$, we have $c(\mathcal{C}_{n})$ remaining faces all of the form $(RU)^{k}$ for some $k$.

Hence, for such a minimal embedding, we have
\begin{align*}
    2 - 2 g(\mathcal{G}_{n}) &= V-E+F \\
    &= V - 2V + e_{2}(\mathcal{C}_{n}) + \frac{V-e_{2}(\mathcal{C}_{n})}{2} + e_{3}(\mathcal{C}_{n}) + \frac{V-e_{3}(\mathcal{C}_{n})}{3} + c(\mathcal{C}_{n}) \\
    & = -\frac{V}{6} + \frac{e_{2}(\mathcal{C}_{n})}{2} + \frac{2e_{3}(\mathcal{C}_{n})}{3} + c(\mathcal{C}_{n}) \\
    \Rightarrow g(\mathcal{G}_{n}) &= \frac{V}{12} - \frac{e_{2}(\mathcal{C}_{n})}{4} - \frac{e_{3}(\mathcal{C}_{n})}{3} - \frac{c(\mathcal{C}_{n})}{2} +1\\
    &= -\frac{\chi(\mathcal{C}_{n})}{2} - \frac{e_{2}(\mathcal{C}_{n})}{4} - \frac{e_{3}(\mathcal{C}_{n})}{3} - \frac{c(\mathcal{C}_{n})}{2} + 1 \\ 
    & = g(\mathcal{C}_{n}).
\end{align*}
The fact that $V = -6\chi(\mathcal{C}_{n})$ follows as in~\cite[Proof of Theorem 1.1]{KM}, and the last inequality holds by the definition of the orbifold Euler characteristic.

Now, the genus two case of Theorem~\ref{thm:main} for the elliptic generators follows by the work of Mukamel~\cite{Muk}. For the Prym loci in $\calH(4)$ and $\calH(6)$, Torres-Teigel--Zachhuber do not establish the same (they only handle non-square discriminants). However, we will now give a more general argument that these graphs still have genus going to infinity. The argument generalises the one we gave in the introduction for proving Theorem~\ref{thm:main} given Theorem~\ref{thm:submain-1}.

So, let $\chi = \chi(\mathcal{G}_{n}) = V -E+F$ for a minimal embedding. Recall that we have a 4-valent graph and so
\[2E = 4V \Rightarrow 11E = 22V.\]
Again, let $f_{i}$ denote the number of faces of length $i$ so that $F = \sum_i f_i$. Similar to before,
\begin{equation*}
    \begin{split}
    2E = \sum_{i} i\cdot f_{i} = \sum_{i\geq 13}i\cdot f_{i} + \sum_{i=1}^{12} i \cdot f_{i} &\geq \sum_{i\geq 13}13\cdot f_{i} + \sum_{i=1}^{12} (13-(13-i)) \cdot f_{i} \\
    &= 13\sum_{i}f_{i} - \sum_{i = 1}^{12}(13-i)\cdot f_{i} \\
    &= 13F - \sum_{i = 1}^{12}(13-i)\cdot f_{i}. 
    \end{split}
\end{equation*}
From this, we get 
\[
    2E\geq 13E - 13V + 13\chi - \sum_{i = 1}^{12}(13-i)\cdot f_{i} \Rightarrow -13\chi \geq 11E - 13V - \sum_{i = 1}^{12}(13-i)\cdot f_{i},
\]
which simplifies to
\[-13\chi \geq 9V - \sum_{i = 1}^{12}(13-i)\cdot f_{i}.\]

Now, from the discussion above, we know that $f_1$ is the number of loops in the graph which is equal to $e_{2}(\mathcal{C}_{n})+e_{3}(\mathcal{C}_{n})$. In all cases under consideration, this behaves as $o(V)$ (either via the explicit orbifold point counts in $\calH(2)$ and the Prym loci of $\calH(4)$ and $\calH(6)$, the orbifold point counts in $\calH(1,1)$ with the Parity Conjecture, or for a density one subsequence via the work of the previous section assuming Conjecture~\ref{conj:wild}).

We have that $f_{2}$ is equal to the number of $R^{2}$ 2-cycles plus the number of $RU$ 2-cycles. The count of the former is $\frac{V-e_{2}(\mathcal{C}_{2})}{2} = \frac{V}{2}-o(V)$, while the count of the latter is $o(V)$ since such a 2-cycle corresponds to a cusp of width one which we have bounded via the Browning--Gorodnik counting methods above.

We have that $f_{3}$ is the number of $U^{3}$ 3-cycles which is equal to $\frac{V-e_{3}(\mathcal{C}_{n})}{3} = \frac{V}{3}-o(V)$.

Finally, for $4\leq k\leq 12$ even, $f_k$ is equal to the number of cusps of width $\frac{k}{2}$ which is $o(V)$ since our Browning--Gorodnik counting methods above still work as long as the cusp width is bounded in terms of the number of squares $n$. If $4\leq k\leq 12$ is odd then $f_k = 0$.

Hence, we see that 
\[-13\chi \geq 9V - 12\cdot o(V) - 11\left(\frac{V}{2}-o(V)\right) - 10\left(\frac{V}{3}-o(V)\right) - 9\cdot o(V) = \frac{V}{6}-o(V)\]
and so $-\chi$ goes to infinity with $V$. Hence, $g(\mathcal{C}_{n}) = g(\mathcal{G}_{n})$ goes to infinity with $n$ (or along a density one subsequence), and we are done.

\end{document}